\documentclass[12pt]{article}
\usepackage{array}
\newcolumntype{L}[1]{>{\raggedright\let\newline\\\arraybackslash\hspace{0pt}}m{#1}}
\newcolumntype{C}[1]{>{\centering\let\newline\\\arraybackslash\hspace{0pt}}m{#1}}
\newcolumntype{R}[1]{>{\raggedleft\let\newline\\\arraybackslash\hspace{0pt}}m{#1}}

\makeatletter
\newcommand{\thickhline}{%
	\noalign {\ifnum 0=`}\fi \hrule height 1pt
	\futurelet \reserved@a \@xhline
}
\newcolumntype{"}{@{\hskip\tabcolsep\vrule width 1pt\hskip\tabcolsep}}
\makeatother

\newcolumntype{?}{!{\vrule width 1pt}}

\usepackage{extarrows}
\usepackage{tikzsymbols,tikz,pgfplots}
\usetikzlibrary{snakes}
\usepackage{graphicx}
\usepackage{caption}
\usepackage{subcaption}
\usetikzlibrary{arrows.meta}
\pgfmathdeclarefunction{norm}{3}{%
	\pgfmathparse{sqrt(0.5*#3/pi)*exp(-0.5*#3*(#1-#2)^2)}%
}

\newcommand{\ie}{\textit{i.e.}}
\newcommand{\eg}{\textit{e.g.}}

\newcommand{\Ro}{\uppercase\expandafter{\romannumeral1}}
\newcommand{\Rt}{\uppercase\expandafter{\romannumeral2}}
\newcommand{\Rthree}{\uppercase\expandafter{\romannumeral3}}
\newcommand{\Rfour}{\uppercase\expandafter{\romannumeral4}}
\newcommand{\ro}{\romannumeral1}
\newcommand{\rt}{\romannumeral2}

\newcommand{\Rd}{\mathbb{R}^d}

\def\calH{\mathcal{H}}

\usepackage{bbm}
\usepackage{bm}
\usepackage[toc,page]{appendix}
\usepackage{amsthm}
\usepackage{amssymb}

\RequirePackage[colorlinks,linkcolor=blue,citecolor=blue,urlcolor=blue]{hyperref}
\usepackage{latexsym,epsf,dsfont,color,eurosym,mathrsfs,url}
\mathchardef\mhyphen="2D
\usepackage[round,colon,authoryear]{natbib}
\usepackage{amsmath}
\usepackage{mathtools}
\usepackage{graphicx}
\usepackage{float} 
\usepackage{caption}
\usepackage{subcaption}
\usepackage{color,soul}
\usepackage{enumerate}
\usepackage{verbatim}
\usepackage{tabularx}
\usepackage{tabu}
\usepackage{diagbox}

\graphicspath{{figures/}}

\newlength{\fixboxwidth}
\newtheorem{lemma}{Lemma}
\newtheorem{theorem}{Theorem}

\newcommand{\NN}{{\mathbb N}}

\newcommand{\RR}{{\mathbb R}}
\newcommand{\EE}{{\mathbb E}}

\newcommand{\PP}{\mathbb P}
\newcommand{\QQ}{\mathbb Q}

\newcommand{\HEE}{\hat{\EE}}

\def\Fcal{\mathcal F}
\def\Gcal{\mathcal G}
\def\Hcal{\mathcal H}

\def\Wcal{\mathcal W}
\def\Xcal{\mathcal X}

\def\hat{\widehat}
\def\epsilon{\varepsilon}

\def\var{{\rm var}}

\newcommand{\bGnn}{\bar{G}_{\nu_n}}

\usepackage{thmtools}
\declaretheoremstyle[notefont=\bfseries,notebraces={}{},%
headpunct={},postheadspace=1em]{mystyle}
\declaretheorem[style=mystyle,numbered=no,name=Theorem]{thm-hand}

\newcounter{lastnote}

\begin{document}

\title{On the Optimality of Gaussian Kernel Based Nonparametric Tests against Smooth Alternatives$^\ast$}

\date{(\today)}

\author{Tong Li and Ming Yuan$^\dag$\\
Columbia University}
\maketitle

\begin{abstract}
Nonparametric tests via kernel embedding of distributions have witnessed a great deal of practical successes in recent years. However, statistical properties of these tests are largely unknown beyond consistency against a fixed alternative. To fill in this void, we study here the asymptotic properties of goodness-of-fit, homogeneity and independence tests using Gaussian kernels, arguably the most popular and successful among such tests. Our results provide theoretical justifications for this common practice by showing that tests using Gaussian kernel with an appropriately chosen scaling parameter are minimax optimal against smooth alternatives in all three settings. In addition, our analysis also pinpoints the importance of choosing a diverging scaling parameter when using Gaussian kernels and suggests a data-driven choice of the scaling parameter that yields tests optimal, up to an iterated logarithmic factor, over a wide range of smooth alternatives. Numerical experiments are also presented to further demonstrate the practical merits of the methodology.
\end{abstract}

\footnotetext[1]{
Research supported in part by NSF Grant DMS-1803450.}
\footnotetext[2]{
Address for Correspondence: Department of Statistics, Columbia University, 1255 Amsterdam Avenue, New York, NY 10027.}

\newpage
\section{Introduction}
Tests for goodness-of-fit, homogeneity and independence are central to statistical inferences. Numerous techniques have been developed for these tasks and are routinely used in practice. In recent years, there is a renewed interest on them from both statistics and other related fields as they arise naturally in many modern applications where the performance of the classical methods are less than satisfactory. In particular, nonparametric inferences via the embedding of distributions into a reproducing kernel Hilbert space (RKHS) have emerged as a popular and powerful technique to tackle these challenges. The approach immediately allows for easy access to the rich machinery for RKHS and has found great successes in a wide range of applications from causal discovery to deep learning. See, e.g., \cite{muandet2017kernel} for a recent review.

More specifically, let $K(\cdot,\cdot)$ be a symmetric and positive definite function defined over $\Xcal\times\Xcal$, that is $K(x,y)=K(y,x)$ for all $x,y\in \Xcal$, and the Gram matrix $[K(x_i,x_j)]_{1\le i,j\le n}$ is positive definite
for any distinct $x_1,\ldots,x_n\in \Xcal$. The Moore-Aronszajn Theorem indicates that such a function, referred to as a kernel, can always be uniquely identified with a RKHS $\calH_K$ of functions over $\Xcal$. The embedding
$$
\mu_{\PP}(\cdot) :=\int_{\Xcal} K(x,\cdot)\PP(dx),
$$
maps a probability distribution $\PP$ into $\calH_K$. The difference between two probability distributions $\PP$ and $\QQ$ can then be conveniently measured by
$$
\gamma_K(\PP,\QQ):=\|\mu_\PP-\mu_\QQ\|_{\calH_K}.
$$
Under mild regularity conditions, it can be shown that $\gamma_K(\PP,\QQ)$ is an integral probability metric so that it is zero if and only if $\PP=\QQ$, and
$$
\gamma_K(\PP,\QQ)=\sup_{f\in \Hcal_K: \|f\|_{\calH_K}\le 1} \int_{\Xcal} fd\left(\PP-\QQ\right).
$$
As such, $\gamma_K(\PP,\QQ)$ is often referred to as the \emph{maximum mean discrepancy} (MMD) between $\PP$ and $\QQ$. See, \eg, \cite{sriperumbudur2010hilbert} or \cite{gretton2012kernel} for details. In what follows, we shall drop the subscript $K$ whenever its choice is clear from the context. It was noted recently that MMD is also closely related to the so-called energy distance between random variables \citep{szekely2007measuring, szekely2009brownian} commonly used to measure independence. See, e.g., \cite{sejdinovic2012equivalence, lyons2013distance}.

Given a sample from $\PP$ and/or $\QQ$, estimates of the $\gamma(\PP,\QQ)$ can be derived by replacing $\PP$ and $\QQ$ with their respective empirical distributions. These estimates can subsequently be used for various statistical inferences. Here are several notable examples that we shall focus on in this work.

\paragraph{Goodness-of-fit tests.} The goal of goodness-of-fit tests is to check if a sample comes from a pre-specified distribution. Let $X_1,\cdots,X_n$ be $n$ independent $\Xcal$-valued samples from a certain distribution $\PP$. We are interested in testing if the hypothesis $H_0^{\rm GOF}:\ \PP=\PP_0$ holds for a fixed $\PP_0$. Deviation from $\PP_0$ can be conveniently measured by $\gamma(\PP,\PP_0)$ which can be readily estimated by:
\begin{equation*}
\gamma(\hat{\PP}_n,\PP_0):=\sup_{f\in \Hcal(K): \|f\|_K\le 1} \int_\Xcal fd\left(\hat{\PP}_n-\PP_0\right),
\end{equation*}
where $\hat{\PP}_n$ is the empirical distribution of $X_1,\cdots,X_n$. A natural procedure is to reject $H_0$ if the estimate exceeds a threshold calibrated to ensure a certain significance level, say $\alpha$ ($0<\alpha<1$).

\paragraph{Homogeneity tests.}  Homogeneity tests check if two independent samples come from a common population. Given two independent samples $X_1,\cdots,X_n\sim_{\rm iid}\PP$ and $Y_1,\cdots, Y_m\sim_{\rm iid}\QQ$, we are interested in testing if the null hypothesis $H_0^{\rm HOM}: \PP=\QQ$ holds. Discrepancy between $\PP$ and $\QQ$ can be measured by $\gamma(\PP, \QQ)$, and similar to before, it can be estimated by the MMD between $\hat{\PP}_n$ and $\hat{\QQ}_m$:
\begin{equation*}
\gamma(\hat{\PP}_n,\hat{\QQ}_m):=\sup_{f\in \Hcal(K): \|f\|_K\le 1} \int_\Xcal fd\left(\hat{\PP}_n-\hat{\QQ}_m\right).
\end{equation*}
Again we reject $H_0$ if the estimate exceeds a threshold calibrated to ensure a certain significance level.

\paragraph{Independence tests.}  How to measure or test of independence among a set of random variables is another classical problem in statistics. Let $X=(X^1,\ldots, X^k)^\top \in \Xcal_1\times\cdots\times\Xcal_k$ be a random vector. If the random vectors $X^1,\ldots,X^k$ are jointly independent, then the distribution of $X$ can be factorized:
$$H_0^{\rm IND}:\qquad \PP^{X}=\PP^{X^1}\otimes \cdots\otimes \PP^{X^k}.$$
Dependence among $X^1,\ldots, X^k$ can be naturally measured by the difference between the joint distribution and the product distribution evaluated under MMD:
$$
\gamma(\PP^{X},\PP^{X^1}\otimes \cdots\otimes \PP^{X^k})=\|\mu_{\PP^{X}}-\mu_{\PP^{X^1}\otimes \cdots\otimes \PP^{X^k}}\|_{\calH_K}.
$$
When $d=2$, $\gamma^2(\PP^{X},\PP^{X^1}\otimes \PP^{X^2})$ can be expressed as the squared Hilbert-Schmidt norm of the cross-covariance operator associated with $X^1$ and $X^2$ and is therefore referred to as Hilbert-Schmidt independence criterion \citep[HSIC;][]{gretton2005measuring}. The more general case as given above is sometimes referred to as dHSIC \citep[see, e.g.,][]{pfister2018kernel}. As before, we proceed to reject the independence assumption when $\gamma(\hat{\PP}^{X}_n,\hat{\PP}^{X^1}_{n}\otimes \cdots\otimes \hat{\PP}^{X^k}_n)$ exceed a certain threshold where $\hat{\PP}_n^{X}$ and $\hat{\PP}^{X^j}_{n}$ are the empirical distribution of $X$ and $X^j$ respectively.
\vskip 20pt

In all these cases the test statistic, namely $\gamma^2(\hat{\PP}_n,\PP_0)$, $\gamma^2(\hat{\PP}_n,\hat{\QQ}_m)$ or $\gamma^2(\hat{\PP}_n,\hat{\PP}_n^{X^1}\otimes\cdots\otimes\hat{\PP}_n^{X^k})$, is a V-statistic. Following standard asymptotic theory for V-statistics \citep[see, e.g.,][]{serfling2009approximation}, it can be shown that under mild regularity conditions, when appropriately scaled by the sample size, they converge to a mixture of $\chi^2_1$ distribution with weights determined jointly by the underlying probability distribution and the choice of kernel $K$. In contrast, it can also be derived that for a fixed alternative, $\gamma^2(\hat{\PP}_n,\PP_0)\to_p \gamma^2(\PP,\PP_0)$, $\gamma^2(\hat{\PP}_n,\hat{\QQ}_m)\to_p \gamma^2(\PP,\QQ)$ and $\gamma^2(\hat{\PP}_n,\hat{\PP}_n^{X^1}\otimes\cdots\otimes\hat{\PP}_n^{X^k})\to_p\gamma^2(\PP,\PP^{X^1}\otimes\cdots\times \PP^{X^k})$. This immediately suggests that all aforementioned tests are consistent against fix alternatives in that their power tends to one as sample sizes increase. Although useful, such consistency results do not tell the full story about the power of these tests, and if there are yet more powerful methods.

For example, as recently shown by \cite{balasubramanian2017optimality}, any goodness-of-fit test based on statistic $\gamma^2_K(\hat{\PP}_n,\PP_0)$ with a \emph{fixed} kernel $K$ is necessarily suboptimal. \cite{balasubramanian2017optimality} also argued that much more powerful tests can be constructed by \emph{regularized embedding}. The appropriate regularization they employed, however, relies on the knowledge of $\PP_0$, and therefore is specialized to goodness-of-fit tests. While it is plausible that MMD based tests for homogeneity or independence may suffer from similar deficiencies, it remains unclear how to construct tests that are more powerful in these settings. The goal of the current work is specifically to address this question. In particular, we show that embedding using Gaussian kernel with an appropriately chosen scaling parameter provides a unified treatment to all three testing problems.

When data are continuous, e.g. $\Xcal=\RR^d$, Gaussian kernels are arguably the most popular and successful choice in practice. On the one hand, we show that this choice of kernel is justified because in all three scenarios, MMD based tests can be optimal for testing against smooth alternatives provided that an appropriate scaling parameter is elicited. On the other hand, we argue that existing ways of selecting the scaling parameter may not exploit the full potential of Gaussian kernel based approaches and yet more powerful tests can be constructed.

In particular, we investigate how the power of these tests increases with the sample size by characterizing the asymptotic behavior of the smallest amount of departure from the null hypothesis that can be consistently detected. More specifically, we adopt the minimax hypothesis testing framework pioneered by \cite{burnashev1979minimax, ingster1987minimax, ingster1993asymptotically}. See also \cite{ermakov1991minimax, spokoiny1996adaptive, lepski1999minimax, ingster2000minimax, ingster2000adaptive, baraud2002non, fromont2006adaptive, fromont2012kernels, fromont2013two}, and references therein. Within this framework, we consider testing against alternatives getting closer and closer to the null hypothesis as the sample size increases. The smallest departure from the null hypotheses that can be detected consistently, in a minimax sense, is referred to as the optimal detection boundary. In all three settings, goodness of fit, homogeneity and independence testing, we show that Gaussian kernels with an appropriately chosen scaling parameter yield tests that are rate optimal in detecting smooth departures from null hypotheses. Our results not only provide rigorous justifications to the practical successes of Gaussian kernels based testing procedures but also offer guidelines on how to choose the scaling parameter in a principled way.

The critical importance of selecting an appropriate scaling parameter is widely recognized in practice. Yet, the way it is done is usually ad hoc and how to do so in a more principled way remains one of the chief practical challenges. See, e.g., \cite{gretton2008kernel, fukumizu2009kernel, gretton2012optimal,  sutherland2016generative}. Our analysis shows that it is essential that we take a diverging scaling parameter as the sample size increases, and the choice of the scaling parameter may determine against which types of deviation from the null hypothesis the resulting test is most powerful.

This also naturally brings about the issue of adaptation and whether or not there is an agnostic approach towards testing of the aforementioned null hypotheses without the need to specify a scaling parameter. To address this challenge, we introduce a simple testing procedure by maximizing a studentized MMD over a pre-specified range of scaling parameters. Similar idea of maximizing MMD over a class of kernels was first introduced by \cite{sriperumbudur2009kernel}. Our analysis, however, suggests that it is more desirable to maximize \emph{normalized} MMD instead. More specifically, we show that the proposed procedure can attain the optimal rate, up to an iterated logarithmic factor, simultaneously over the collection of parameter spaces corresponding to different levels of smoothness.

The rest of this paper is organized as follows. In the next three sections, we shall investigate the statistical properties of Gaussian kernel based tests for goodness-of-fit, homogeneity and independence respectively, and show that with appropriate choice of the scaling parameter, these tests are minimax optimal if the underlying densities are smooth. Since the optimal choice of scaling parameter requires the knowledge of smoothness which is rarely available, in Section \ref{sec:adapt}, we introduce new tests that do not require such knowledge yet attain optimal power, up to an iterated logarithmic factor, for a wide range of smooth alternatives. Numerical experiments presented in Section \ref{sec:sim} further illustrate the practical merits of our method and theoretical developments. We conclude with some summary discussion in Section \ref{sec:disc} and all proofs are relegated to Section \ref{sec:proof}.

\section{Test for Goodness-of-fit}
\label{sec:gof}

Among the three testing problems that we consider, it is instructive to begin with the case of goodness-of-fit. Obviously, the choice of kernel $K$ plays an essential role in kernel embedding of distributions. In particular, when data are continuous, Gaussian kernels are commonly used. More specifically, a Gaussian kernel with a scaling parameter $\nu>0$ is given by
$$
G_{d,\nu}(x,y)=\exp\left(-\nu\|x-y\|_d^2\right),\qquad \forall x,y\in \RR^d.
$$
Hereafter $\|\cdot\|_d$ stands for the usual Euclidean norm in $\RR^d$. For brevity, we shall suppress the subscript $d$ in both $\|\cdot\|$ and $G$ when the dimensionality is clear from the context. When $\PP$ and $\QQ$ are probability distributions defined over $\Xcal=\RR^d$, we shall write the MMD between them with a Gaussian kernel and scaling parameter $\nu$ as $\gamma_\nu(\PP,\QQ)$ where the subscript signifies the specific value of the scaling parameter. 

We shall restrict our attention to distributions with smooth densities. Denote by $\Wcal^{s,2}_d$ the $s$th order Sobolev space in $\RR^d$, that is
$$
\Wcal^{s,2}_d=\left\{f:\RR^d\to \RR\big|f\ \text{is almost surely continuous and} \int (1+\|\omega\|^2)^{s/2} \|\Fcal(f)(\omega)\|^2d\omega<\infty\right\},
$$
where $\Fcal(f)$ is the Fourier transform of $f$:
$$
\Fcal(f)(\omega)=\frac{1}{(2\pi)^{d/2}}\int_{\RR^d} f(x)e^{-i x^\top\omega}dx.
$$
In what follows, we shall again abbreviate the subscript $d$ in $\Wcal^{s,2}_d$ when it is clear from the context. For any $f\in \Wcal^{s,2}$, we shall write
$$
\|f\|_{\Wcal^{s,2}}^2=\int_{\RR^d} (1+\|\omega\|^2)^s \|\Fcal(f)(\omega)\|^2d\omega.
$$
Let $p$ and $p_0$ be the density functions of $\PP$ and $\PP_0$ respectively. We are interested in the case when both $p$ and $p_0$ are elements from $\Wcal^{s,2}$. 

Note that we can rewrite the null hypothesis $H_0^{\rm GOF}$ in terms of density functions: $H_0^{\rm GOF}: p=p_0$ for some prespecified denstiy $p_0\in \Wcal^{s,2}$. To better quantify the power of a test, we shall consider testing against an alternative that is increasingly closer to the null as the sample size $n$ increases:
$$
H_1^{\rm GOF}(\Delta_n;s): p\in \Wcal^{s,2}(M), \quad \|p-p_0\|_{L_2}\ge \Delta_n,
$$
where
$$\Wcal^{s,2}(M)=\left\{f\in \Wcal^{s,2}: \|f\|_{\Wcal^{s,2}}\le M\right\}.$$
and
$$
\|f\|_{L_2}^2=\int_{\RR^d} f^2(x)dx.
$$
The alternative hypothesis $H_1^{\rm GOF}(\Delta_n; s)$ is composite and the power of a test $\Phi$ based on $X_1,\ldots, X_n\sim p$ is therefore defined as
$$
{\rm power}(\Phi; H_1^{\rm GOF}(\Delta_n;s)):=\inf_{p\in \Wcal^{s,2}(M), \|p-p_0\|_{L_2}\ge \Delta_n}\PP\{\Phi {\rm \ rejects\ } H_0^{\rm GOF}\}
$$
Of particular interest here is the smallest $\Delta_n$ so that a test is consistent in that the above quantity converges to one. 

Consider embedding with Gaussian kernel and a fixed scaling parameter $\nu>0$. Following standard asymptotic theory for V-statistics \citep[see, e.g.,][]{serfling2009approximation}, it can be shown that under $H_0^{\rm GOF}$ and certain regularity conditions,
$$
n\gamma^2_\nu(\hat{\PP},\PP_0)\to_d \sum_{k\ge 1}\lambda_k^2 Z_k^2
$$
where $\lambda_1\ge\lambda_2\ge\cdots$ are the singular values of the linear operator:
$$
{\mathcal L}_\nu f=\int_{\RR^d} \bar{G}_\nu(x,x')f(x')dx',\qquad \forall f\in L_2(\RR^d)
$$
and
$$
\bar{G}_\nu(x,y;\PP_0)=G_\nu(x,y)-\EE_{X\sim \PP_0}G_\nu(X,y)-\EE_{X\sim \PP_0}G_\nu(x,X)+\EE_{X,X'\sim_{\rm iid}\PP_0}G_\nu(X,X').
$$
and $Z_k$s are independent standard normal random variables. Hereafter, for brevity, we shall omit the last argument of $\bar{G}$ when it is clear from the context. As such, we may proceed to reject $H_0^{\rm GOF}$ if and only if $n\hat{\gamma}_\nu^2(\hat{\PP}_n, \PP_0)$ exceeds the upper $\alpha$ quantile of its asymptotic distribution, which yields an (asymptotic) $\alpha$-level test. Following the same argument as that from \cite{balasubramanian2017optimality}, we can show that under mild regularity conditions such a test has power tending to one if and only if $\Delta_n\gg n^{-1/4}$. In addition, as shown by \cite{balasubramanian2017optimality}, much more powerful tests exist when assuming that the underlying densities are compactly supported and bounded away from 0 and 1. Here we show that the same is true for broader classes of distributions using Gaussian kernel embedding with a diverging scaling parameter.

Recall that
$$
\gamma^2_\nu(\hat{\PP}_n,\PP_0)={1\over n^2}\sum_{i,j=1}^n \bar{G}_\nu(X_i,X_j).
$$
It is not hard to see that this is a biased estimate of $\gamma^2_\nu(\PP,\PP_0)$ due to the oversized influence of the summands when $i=j$. It is often common to correct for bias and use instead the following $U$-statistic:
$$	
\hat{\gamma_\nu^2}(\PP,\PP_0):={1\over n(n-1)}\sum_{1\le i\neq j\le n}^n \bar{G}_\nu(X_i,X_j),
$$
which we shall focus on in what follows.

The choice of the scaling parameter $\nu$ is essential when using RKHS embedding for goodness-of-fit test. While the importance of data-driven choice of $\nu$ is widely recognized in practice, almost all existing theoretical studies assume that a fixed kernel, therefore a fixed scaling parameter, is used. Here we shall demonstrate the benefit of using a data-driven scaling parameter, and especially choosing a scaling parameter that diverges with the sample size.

More specifically, we argue that, with appropriate scaling, $\hat{\gamma_\nu^2}(\PP,\PP_0)$ can be viewed as an estimate of $\|p-p_0\|_{L_2}^2$ when $\nu\to\infty$ as $n\to\infty$. Note that
$$
\int(p-p_0)^2=\int p^2 -2\int p\cdot p_0+\int p_0^2.
$$
The first term can be estimated by
$$
\int p^2\approx {1\over n}\sum_{i=1}^n p(X_i)\approx{1\over n}\sum_{i=1}^n \hat{p}_{h,-i}(X_i)
$$
where $\hat{p}_{h,-i}$ is a kernel density estimate of $p$ with the $i$th observation removed and bandwidth $h$:
$$
\hat{p}_{h,-i}(x)={1\over n(2\pi h^2)^{d/2}}\sum_{j\neq i} G_{(2h^2)^{-1}}(x-X_j).
$$
Thus, we can estimate $\int p^2$ by
$$
{1\over n(n-1)(2\pi h^2)^{d/2}}\sum_{1\le i\neq j\le n} G_{(2h^2)^{-1}}(X_i,X_j).
$$
Similarly, the cross-product term can be estimated by
$$
\int p\cdot p_0\approx \int \hat{p}_h(x)p_0(x)dx={1\over n(2\pi h^2)^{d/2}}\sum_{i=1}^n \int G_{(2h^2)^{-1}}(x,X_i)p_0(x)dx.
$$
Together, we can view
$$
{1\over n(n-1)(2\pi h^2)^{d/2}}\sum_{1\le i\neq j\le n} \bar{G}_{(2h^2)^{-1}}(X_i,X_j)
$$
as an estimate of $\int (p-p_0)^2$. Following standard asymptotic properties of the kernel density estimator \citep[see, e.g.,][]{tsybakov2008introduction}, we know that
$$
(\pi/\nu)^{-d/2} \hat{\gamma_\nu^2}(\PP,\PP_0)\to_p \|p-p_0\|_{L_2}^2
$$
if $\nu\to \infty$ in such a fashion that $\nu=o(n^{4/d})$. Motivated by this observation, we shall now consider testing $H_0^{\rm GOF}$ using $\hat{\gamma_\nu^2}(\PP,\PP_0)$ with a diverging $\nu$. To signify the dependence of $\nu$ on the sample size, we shall add a subscript $n$ in what follows.

Under $H_0^{\rm GOF}$, it is clear $\EE\hat{\gamma_{\nu_n}^2}(\PP,\PP_0)=0$. Note also that
\begin{align}
&{\rm var}(\hat{\gamma_{\nu_n}^2}(\PP,\PP_0))\notag\\=&{2\over n(n-1)}\EE\left[\bar{G}_{\nu_n}(X_1,X_2)\right]^2 \nonumber\\
=&{2\over n(n-1)}\left[\EE\left[G_{\nu_n}(X_1,X_2)\right]^2-2\EE[G_{\nu_n}(X_1,X_2)G_{\nu_n}(X_1,X_3)]+\left(\EE\left[G_{\nu_n}(X_1,X_2)\right]\right)^2\right] \nonumber\\
=&{2\over n(n-1)}\left[\EE G_{2\nu_n}(X_1,X_2)-2\EE[G_{\nu_n}(X_1,X_2)G_{\nu_n}(X_1,X_3)]+\left(\EE\left[G_{\nu_n}(X_1,X_2)\right]\right)^2\right]. \label{eq:var}
\end{align}
Simple calculations yield:
$$
{\rm var}(\hat{\gamma_{\nu_n}^2}(\PP,\PP_0))={2(\pi/(2\nu_n))^{d/2}\over n^2}\cdot \|p_0\|_{L_2}^2\cdot(1+o(1)),
$$
assuming that $\nu_n\to\infty$. We shall show that
$$
\frac{n}{\sqrt{2}}\left(2\nu_n\over \pi\right)^{d/4}\hat{\gamma_{\nu_n}^2}(\PP,\PP_0)\to_d N\left(0,\|p_0\|_{L_2}^2\right).
$$
To use this as a test statistic, however, we will need to estimate ${\rm var}(\hat{\gamma_{\nu_n}^2}(\PP,\PP_0))$. To this end, it is natural to consider estimating each of the three terms on the rightmost hand side of \eqref{eq:var} by $U$-statistics:
\begin{align*}
\tilde{s}^2_{n,\nu_n}=&\frac{1}{n(n-1)}\sum\limits_{1\leq i\neq j\leq n}G_{2\nu_n}(X_i,X_j)\\&-\frac{2(n-3)!}{n!}\sum\limits_{\substack{1\le i,j_1,j_2\le n\\ |\{i,j_1,j_2\}|=3}}G_{\nu_n}(X_i,X_{j_1})G_{\nu_n}(X_i,X_{j_2})\\&+\frac{(n-4)!}{n!}\sum\limits_{\substack{1\le i_1,i_2,j_1,j_2\le n\\ |\{i_1,i_2,j_1,j_2\}|=4}}G_{\nu_n}(X_{i_1},X_{j_1})G_{\nu_n}(X_{i_2},X_{j_2}).
\end{align*}
Note that $\tilde{s}^2_{n,\nu_n}$ is not always positive. To avoid a negative estimate of the variance, we can replace it with a sufficiently small value, say $1/n^2$, whenever it is negative or too small. Namely, let 
$$
\hat{s}^2_{n,\nu_n}=\max\left\{\tilde{s}^2_{n,\nu_n},1/n^2\right\},$$
and consider a test statistic:
$$
T_{n,\nu_n}^{\rm GOF}:={n\over\sqrt{2}}\hat{s}_{n,\nu_n}^{-1}\hat{\gamma_{\nu_n}^2}(\PP,\PP_0).
$$
We have
\begin{theorem}
\label{th:gofnull}
Let $\nu_n\to \infty$ as $n\to\infty$ in such a fashion that $\nu_n=o(n^{4/d})$ . Then, under $H_0^{\rm GOF}$,
\begin{equation}
\label{eq:gofnull1}
\frac{n}{\sqrt{2}}\left(2\nu_n\over \pi\right)^{d/4}\hat{\gamma_{\nu_n}^2}(\PP,\PP_0)\to_d N(0,\|p_0\|_{L_2}^2).
\end{equation}
Moreover,
\begin{equation}
\label{eq:gofnull2}
T_{n,\nu_n}^{\rm GOF}\to_d N(0,1).
\end{equation}
\end{theorem}

Theorem \ref{th:gofnull} immediately implies a test, denoted by $\Phi^{\rm GOF}_{n,\nu_n,\alpha}$ $(\alpha\in(0,1))$, that rejects $H_0^{\rm GOF}$ if and only if $T_{n,\nu_n}^{\rm GOF}$ exceeds $z_\alpha$, the upper $1-\alpha$ quantile of the standard normal distribution, is an asymptotic $\alpha$-level test.

We now proceed to study its power against a smooth alternative. Following the same argument as before, it can be shown that
$$
{1\over n(n-1)(\pi/\nu_n)^{d/2}}\sum_{1\le i\neq j\le n} \bar{G}_{\nu_n}(X_i,X_j)\to_p \|p-p_0\|_{L_2}^2,
$$
and
$$
(2\nu_n/\pi)^{d/2}\hat{s}^2_{n,\nu_n}\to_p \|p\|_{L_2}^2,
$$
so that
$$
n^{-1}(\nu_n/(2\pi))^{d/4}T_n^{\rm GOF}\to_p \|p-p_0\|_{L_2}^2/\|p\|_{L_2}.
$$
This immediately implies that, if $\nu_n\to\infty$ in such a manner that $\nu_n=o(n^{4/d})$, then $\Phi_{n,\nu_n,\alpha}^{\rm GOF}$ is consistent for a fixed $p\neq p_0$ in that its power converges to one. In fact, as $n$ increases, more and more subtle deviation from $p_0$ can be detected by $\Phi_{n,\nu_n,\alpha}^{\rm GOF}$. A refined analysis of the asymptotic behavior of $T_{n,\nu_n}^{\rm GOF}$ yields that
\begin{theorem}
\label{th:gofpower}
Assume that $n^{2s/(d+4s)}\Delta_n\to \infty$. Then for any $\alpha\in (0,1)$,
$$
\lim_{n\to\infty}{\rm power}\{\Phi^{\rm GOF}_{n,\nu_n,\alpha}; H_1^{\rm GOF}(\Delta_n; s)\}\to 1,
$$
provided that $\nu_n\asymp n^{{4}/(d+4s)}$.
\end{theorem}

In other words, $\Phi_{n,\nu_n,\alpha}^{\rm GOF}$ has a detection boundary of the order $O(n^{-2s/(d+4s)})$ which turns out to be minimax optimal in that no other tests could attain a detection boundary with faster rate of convergence. More precisely, we have

\begin{theorem}
\label{th:goflower}
Assume that $p_0$ is density such that $\|p_0\|_{\Wcal^{s,2}}<M$, and $\lim\inf_{n\to\infty}n^{2s/(d+4s)}\Delta_n<\infty$. Then there exists some $\alpha\in(0,1)$ such that for any test $\Phi_n$ of level $\alpha$ (asymptotically) based on $X_1,\ldots,X_n\sim p$, 
$$
\liminf_{n\to\infty}{\rm power}\{\Phi_n; H_1^{\rm GOF}(\Delta_n; s)\}<1.
$$
\end{theorem}

Together, Theorems \ref{th:gofpower} and \ref{th:goflower} suggest that Gaussian kernel embedding of distributions is especially suitable for testing against smooth alternatives, and it yields a test that could consistently detect the smallest departures, in terms of rate of convergence, from the null distribution. The idea can also be readily applied to testing of homogeneity and independence which we shall examine next.

\section{Test for Homogeneity}
\label{sec:hom}

As in the case of goodness of fit test, we shall consider the case when the underlying distributions have smooth densities so that we can rewrite the null hypothesis as $H_0^{\rm HOM}: p=q\in \Wcal^{s,2}(M)$, and the alternative hypothesis as
$$
H_1^{\rm HOM}(\Delta_n; s): p, q\in \Wcal^{s,2}(M),\quad \|p-q\|_{L_2}\ge \Delta_n.
$$
The power of a test $\Phi$ based on $X_1,\ldots, X_n\sim p$ and $Y_1,\ldots,Y_m\sim q$ is given by
$$
{\rm power}(\Phi; H_1^{\rm HOM}(\Delta_n; s)):=\inf_{p,q\in \Wcal^{s,2}(M), \|p-q\|_{L_2}\ge \Delta_n}\PP\{\Phi {\rm \ rejects\ } H_0^{\rm HOM}\}
$$
To fix ideas, we shall also assume that $c\le m/n\le C$ for some constants $0<c\le C<\infty$. In addition, we shall express explicitly only the dependence on $n$ and not $m$, for brevity. Our treatment, however, can be straightforwardly extended to more general situations.

Recall that
\begin{eqnarray*}
\gamma_{\nu_n}^2(\hat{\PP}_n,\hat{\QQ}_m)={1\over n^2}\sum_{1\le i,j\le n}G_{\nu_n}(X_i,X_j)+{1\over m^2}\sum_{1\le i,j\le m}G_{\nu_n}(Y_i,Y_j)\\
-{2\over mn}\sum_{i=1}^n\sum_{j=1}^mG_{\nu_n}(X_i,Y_j).
\end{eqnarray*}
As before, to reduce bias, we shall focus instead on a closely related estimate of $\gamma_{\nu_n}(\PP,\QQ)$:
\begin{eqnarray*}
\hat{\gamma_{\nu_n}^2}(\PP,\QQ)={1\over n(n-1)}\sum_{1\le i\neq j\le n}G_{\nu_n}(X_i,X_j)+{1\over m(m-1)}\sum_{1\le i\neq j\le m}G_{\nu_n}(Y_i,Y_j)\\
-{2\over mn}\sum_{i=1}^n\sum_{j=1}^mG_{\nu_n}(X_i,Y_j).
\end{eqnarray*}
It is easy to see that under $H_0^{\rm HOM}$,
$$
\EE\hat{\gamma_{\nu_n}^2}(\PP,\QQ)=0,
$$
and
$$
\var\left(\hat{\gamma_{\nu_n}^2}(\PP,\QQ)\right)=2\left(\frac{1}{n(n-1)}+\frac{2}{mn}+\frac{1}{m(m-1)}\right)\EE_{(X,Y)\sim \PP\otimes\QQ} \bar{G}_{\nu_n}^2(X,Y),
$$
where
$$
\bar{G}_{\nu_n}(x,y)=G_\nu(x,y)-\EE_{X\sim \PP}G_{\nu_n}(X,y)-\EE_{Y\sim \QQ}G_{\nu_n}(x,Y)+\EE_{(X,Y)\sim \PP\otimes\QQ}G_{\nu_n}(X,Y).
$$
It is therefore natural to consider estimating the variance by $\hat{s}_{n,m,\nu_n}^2=\max\left\{\tilde{s}_{n,m,\nu_n}^2,1/n^2\right\}$ where
\begin{align*}
\tilde{s}_{n,m,\nu_n}^2=&\frac{1}{N(N-1)}\sum\limits_{1\leq i\neq j\leq N}G_{2\nu_n}(Z_i,Z_j)\\&-\frac{2(N-3)!}{N!}\sum\limits_{\substack{1\le i,j_1,j_2\le N\\ |\{i,j_1,j_2\}|=3}}G_{\nu_n}(Z_i,Z_{j_1})G_{\nu_n}(Z_i,Z_{j_2})\\&+\frac{(N-4)!}{N!}\sum\limits_{\substack{1\le i_1,i_2,j_1,j_2\le N\\ |\{i_1,i_2,j_1,j_2\}|=4}}G_{\nu_n}(Z_{i_1},Z_{j_1})G_{\nu_n}(Z_{i_2},Z_{j_2}),
\end{align*}
$N=n+m$ and $Z_i=X_i$ if $i\le n$ and $Y_{i-n}$ if $i>n$. This leads to the following test statistic
\begin{eqnarray*}
T_{n,\nu_n}^{\rm HOM}={nm\over \sqrt{2}(n+m)}\cdot \hat{s}_{n,m,\nu_n}^{-1}\cdot \hat{\gamma_{\nu_n}^2}(\PP,\QQ).
\end{eqnarray*}
As before, we can show
\begin{theorem}
\label{th:homnull}
Let $\nu_n\to \infty$ as $n\to \infty$ in such a fashion that $\nu_n=o(n^{4/d})$. Then under $H_0^{\rm HOM}: p=q\in \Wcal^{s,2}(M)$, 
$$
T_{n,\nu_n}^{\rm HOM}\to_d N(0,1),\qquad {\rm as\ }n\to\infty.
$$
\end{theorem}

Motivated by Theorem \ref{th:homnull}, we can consider a test, denoted by $\Phi_{n,\nu_n,\alpha}^{\rm HOM}$, that rejects $H_0^{\rm HOM}$ if and only if $T_{n,\nu_n}^{\rm HOM}$ exceeds $z_{\alpha}$. By construction, $\Phi^{\rm HOM}_{n,\nu_n,\alpha}$ is an asymptotic $\alpha$ level test. We now turn to study its power against $H_1^{\rm HOM}$. As in the case of goodness of fit test, we can prove that $\Phi^{\rm HOM}_{n,\nu_n,\alpha}$ is minimax optimal in that it can detect the smallest difference between $p$ and $q$ in terms of rate of convergence. More precisely, we have

\begin{theorem}
\label{th:hompower}
\begin{enumerate}[(i)]
	\item Assume that $n^{2s/(d+4s)}\Delta_n\to \infty$. Then for any $\alpha\in(0,1)$,
	$$
	\lim_{n\to\infty}{\rm power}\{\Phi^{\rm HOM}_{n,\nu_n,\alpha}; H_1^{\rm HOM}(\Delta_n; s)\}\to 1,
	$$
	provided that $\nu_n\asymp n^{4/(d+4s)}$.
	\item Conversely, if $\lim\inf_{n\to\infty}n^{2s/(d+4s)}\Delta_n<\infty$, then there exists some $\alpha\in(0,1)$ such that for any test $\Phi_n$ of level $\alpha$ (asymptotically) based on $X_1,\ldots,X_n\sim p$ and $Y_1,\ldots, Y_m\sim q$, 
	$$
	\liminf_{n\to\infty}{\rm power}\{\Phi_n; H_1^{\rm HOM}(\Delta_n; s)\}<1.
	$$
\end{enumerate} 
\end{theorem}

\section{Test for Independence}
\label{sec:ind}

Similarly, we can also use Gaussian kernel embedding to construct minimax optimal tests of independence. Let $X=(X^1,\ldots, X^k)^\top \in \RR^{d}$ be a random vector where the subvectors $X^j\in \RR^{d_j}$ for $j=1,\ldots, k$ so that $d_1+\cdots+d_k=d$. Denote by $p$ the joint density function of $X$, and $p_j$ the marginal density of $X^j$. We assume that both the joint density and the marginal densities are smooth. Specifically, we shall consider testing
$$
H_0^{\rm IND}: p=p_1\otimes\cdots\otimes p_k,\ p_j\in \Wcal^{s,2}(M_j),\ 1\leq j\leq k
$$
against a smooth departure from independence:
$$
H_1^{\rm IND}(\Delta_n; s): p\in \Wcal^{s,2}(M),\ p_j\in\Wcal^{s,2}(M_j),\ 1\leq j\leq k {\rm\ and\ } \|p-p_1\otimes\cdots\otimes p_k\|_{L_2}\ge \Delta_n,
$$
where $M=\prod\limits_{j=1}^k M_j$ so that $p_1\otimes\cdots\otimes p_k\in\Wcal^{s,2}(M)$ under both null and alternative hypotheses.

Given a sample $\{X_1,\ldots, X_n\}$ of independent copies of $X$, we can naturally estimate the so-called dHSIC $\gamma_{\nu_n}^2(\PP,\PP^{X^1}\otimes\cdots\otimes \PP^{X^k})$ by
\begin{eqnarray*}
\gamma_{\nu_n}^2(\hat{\PP}_n,\hat{\PP}_n^{X^1}\otimes\cdots\otimes \hat{\PP}_n^{X^k})&=&\frac{1}{n^2}\sum_{1\le i, j\le n}G_{\nu_n}(X_i,X_j)\\
&&+{1\over n^{2k}}\sum_{1\le i_1,\ldots,i_k, j_1\ldots, j_k\le n}G_{\nu_n}((X^1_{i_1},\ldots,X^k_{i_k}),(X^1_{j_1},\ldots,X^k_{j_k}))\\
&&-{2\over n^{k+1}}\sum_{1\le i, j_1,\ldots, j_k\le n}G_{\nu_n}(X_i,(X^1_{j_1},\ldots,X^k_{j_k})).
\end{eqnarray*}
To correct for the bias, we shall consider the following estimate of $\gamma_{\nu_n}^2(\PP,\PP^{X^1}\otimes\cdots\otimes \PP^{X^k})$ instead.
\begin{eqnarray*}
\hat{\gamma_{\nu_n}^2}(\PP,\PP^{X^1}\otimes\cdots\otimes \PP^{X^k})&=&\frac{1}{n(n-1)}\sum_{1\leq i\neq j\leq n}G_{\nu_n}(X_i,X_j)\\
&&+{(n-2k)!\over n!}\sum_{\substack{1\leq i_1,\cdots,i_k,j_1,\cdots,j_k\leq n\\ |\{i_1,\cdots,i_k,j_1,\cdots,j_k\}|=2k}}G_{\nu_n}((X^1_{i_1},\ldots,X^k_{i_k}),(X^1_{j_1},\ldots,X^k_{j_k}))\\
&&-{2(n-k-1)!\over n!}\sum_{\substack{1\le i,j_1,\cdots,j_k\le n\\ |\{i,j_1,\cdots,j_k\}|=k+1}}G_{\nu_n}(X_i,(X^1_{j_1},\ldots,X^k_{j_k})).
\end{eqnarray*}

Under $H_0^{\rm IND}$, we have
$$
\EE\hat{\gamma_{\nu_n}^2}(\PP,\PP^{X^1}\otimes\cdots\otimes \PP^{X^k})=0.
$$
Deriving its variance, however, requires a bit more work. Write
$$
h_j(x^j,y)=\EE_{X\sim \PP^{X^1}\otimes\cdots\otimes\PP^{X^k}} G_{\nu_n}((X^1,\ldots, X^{j-1},x^j,X^{j+1},\ldots, X^k), y)
$$
and
$$
g_j(x^j,y)=h_j(x^j,y)-\EE_{X^j\sim \PP^{X^j}}h_j(X^j, y)-\EE_{Y\sim \PP}h_j(x^j, Y)+\EE_{(X^j,Y)\sim \PP^{X^j}\otimes \PP}h_j(X^j, Y).
$$
With slight abuse of notation, also denote by
\begin{align*}
h_{j_1,j_2}(x^{j_1},y^{j_2})=\EE_{X,Y\sim_{\rm iid} \PP^{X^1}\otimes\cdots\otimes\PP^{X^k}} G_{\nu_n}(&(X^1,\ldots, X^{j_1-1},x^{j_1},X^{j_1+1},\ldots, X^k),\\
&(Y^1,\ldots, Y^{j_2-1},y^{j_2},Y^{j_2+1},\ldots, Y^k))
\end{align*}
and
\begin{align*}
g_{j_1,j_2}(x^{j_1},y^{j_2})=&h_{j_1,j_2}(x^{j_1},y^{j_2})-\EE_{X^{j_1}\sim \PP^{X^{j_1}}}h_{j_1,j_2}(X^{j_1}, y^{j_2})\\
&-\EE_{X^{j_2}\sim \PP^{X^{j_2}}}h_{j_1,j_2}(x^{j_1}, X^{j_2})+\EE_{(X^{j_1},Y^{j_2})\sim \PP^{X^{j_1}}\otimes \PP^{X^{j_2}}}h_{j_1,j_2}(X^{j_1}, Y^{j_2}).
\end{align*}
Then we have
\begin{lemma}\label{le:var}
Under $H_0^{\rm IND}$,
\begin{align}
{\rm var}\left(\hat{\gamma^2_{\nu_n}}(\PP,\PP^{X^1}\otimes\cdots\otimes\PP^{X^k})\right)\nonumber=&\frac{2}{n(n-1)}\bigg(\EE \bar{G}_{\nu_n}^2(X,Y)-2\sum\limits_{1\leq j\leq k}\EE\left(g_j(X^{j},Y)\right)^2\nonumber\\
&+\sum\limits_{1\leq j_1,j_2\leq k}\EE \left(g_{j_1,j_2}(X^{j_1},Y^{j_2})\right)^2\bigg)+O(\EE G_{2\nu_n}(X,Y)/n^3).\label{eq:var1}
\end{align}
\end{lemma}
In light of Lemma \ref{le:var}, a variance estimator can be derived by estimating the leading term on the righthand side of \eqref{eq:var1} term by term using $U$-statistics. Formulae for estimating the variance for general $k$ are tedious and we defer them to the appendix for space consideration. In the special case when $k=2$, the leading term on the righthand side of \eqref{eq:var1} takes a much simplified form:
$$
\frac{2}{n(n-1)}\EE\bar{G}_{\nu_n}(X^1,Y^1)\cdot\EE\bar{G}_{\nu_n}(X^2,Y^2),
$$
where $X^j,Y^j\sim_{\rm iid} \PP^{X^j}$ for $j=1,2$. Thus, we can estimate $\EE[\bar{G}_{\nu_n}(X^j,Y^j)]^2$ by
\begin{align*}
\tilde{s}^2_{n,j,\nu_n}=&\frac{1}{n(n-1)}\sum\limits_{1\leq i_1\neq i_2\leq n}G_{2\nu_n}(X_{i_1}^j,X^j_{i_2})\\&-\frac{2(n-3)!}{n!}\sum\limits_{\substack{1\le i,l_1,l_2\le n\\ |\{i,l_1,l_2\}|=3}}G_{\nu_n}(X_i^j,X_{l_1}^j)G_{\nu_n}(X_i^j,X_{l_2}^j)\\&+\frac{(n-4)!}{n!}\sum\limits_{\substack{1\le i_1,i_2,l_1,l_2\le n\\ |\{i_1,i_2,l_1,l_2\}|=4}}G_{\nu_n}(X_{i_1}^j,X_{l_1}^j)G_{\nu_n}(X_{i_2}^j,X_{l_2}^j)
\end{align*}
and ${\rm var}(\hat{\gamma^2_{\nu_n}}(\PP,\PP^{X^1}\otimes\PP^{X^2}))$ by $2/[n(n-1)]\hat{s}^2_{n,\nu_n}$ where
$$
\hat{s}^2_{n,\nu_n}:=\max\left\{\tilde{s}^2_{n,1,\nu_n}\tilde{s}^2_{n,2,\nu_n}, 1/n^2\right\}.
$$
so that a test statistic for $H_0^{\rm IND}$ is
$$
T_{n,\nu_n}^{\rm IND}:={n\over\sqrt{2}}\hat{s}^{-1}_{n,\nu_n}\hat{\gamma_{\nu_n}^2}(\PP,\PP^{X^1}\otimes\PP^{X^2}).
$$
Test statistics for general $k>2$ can be defined accordingly. Again, we have
\begin{theorem}
\label{th:indnull}
Let $\nu_n\to \infty$ as $n\to \infty$ in such a fashion that $\nu_n=o(n^{4/d})$. Then under $H_0^{\rm IND}$, 
$$
T_{n,\nu_n}^{\rm IND}\to_d N(0,1),\qquad {\rm as\ }n\to\infty.
$$
\end{theorem}

Motivated by Theorem \ref{th:indnull}, we can consider a test, denoted by $\Phi_{n,\nu_n,\alpha}^{\rm IND}$, that rejects $H_0^{\rm IND}$ if and only if $T_{n,\nu_n}^{\rm IND}$ exceeds $z_{\alpha}$. By construction, $\Phi^{\rm IND}_{n,\nu_n,\alpha}$ is an asymptotic $\alpha$ level test. We now turn to study its power against $H_1^{\rm IND}$. As in the case of goodness of fit test, we can prove that $\Phi^{\rm HOM}_{n,\nu_n,\alpha}$ is minimax optimal in that it can detect the smallest departure from independence in terms of rate of convergence. More precisely, we have

\begin{theorem}
\label{th:indpower}
\begin{enumerate}[(i)]
	\item Assume that $n^{2s/(d+4s)}\Delta_n\to \infty$. Then for any $\alpha\in(0,1)$,
	$$
	\lim_{n\to\infty}{\rm power}\{\Phi^{\rm IND}_{n,\nu_n,\alpha}; H_1^{\rm IND}(\Delta_n; s)\}\to 1,
	$$ provided that $\nu_n\asymp n^{4/(d+4s)}$.
	\item Conversely, if $\lim\inf_{n\to\infty}n^{2s/(d+4s)}\Delta_n<\infty$, then there exists some $\alpha\in(0,1)$ such that for any test $\Phi_n$ of level $\alpha$ (asymptotically) based on $X_1,\ldots,X_n\sim p$, 
	$$
	\liminf_{n\to\infty}{\rm power}\{\Phi_n; H_1^{\rm IND}(\Delta_n; s)\}<1.
	$$
\end{enumerate} 
\end{theorem}

\section{Adaptation}
\label{sec:adapt}

The results presented in the previous sections not only suggest that Gaussian kernel embedding of distributions is especially suitable for testing against smooth alternatives, but also indicate the importance of choosing an appropriate scaling parameter in order to detect small deviation from the null hypothesis. To achieve maximum power, the scaling parameter should be chosen according to the smoothness of underlying density functions. This, however, presents a practical challenge because the level of smoothness is rarely known a priori. This naturally brings about the questions of adaption: can we devise an agnostic testing procedure that does not require such knowledge but still attain similar performance? We shall show in this section that this is possible, at least for sufficiently smooth densities.

\subsection{Test for Goodness-of-fit}
We again begin with the test for goodness-of-fit. As we show in Section \ref{sec:gof}, under $H_0^{\rm GOF}$, $T_{n,\nu_n}^{\rm GOF}\to_d N(0,1)$ if $1\ll\nu_n\ll n^{4/d}$; whereas for any $p\in \Wcal^{s,2}$ such that $\|p-p_0\|_{L_2}\gg n^{-2s/(d+4s)}$, $T_{n,\nu_n}^{\rm GOF}\to\infty$ provided that $\nu_n\asymp n^{4/(d+4s)}$. This motivates us to consider the following test statistic:
$$
T_n^{\rm GOF (adapt)}=\max_{1\le \nu_n\le n^{2/d}} T_{n,\nu_n}^{\rm GOF}.
$$
In light of earlier discussion, it is plausible that such a statistic could be used to detect any smooth departure from the null provided that the level of smoothness $s\ge d/4$. We now argue that this is indeed the case. More specifically, we shall proceed to reject $H_0^{\rm GOF}$ if and only if $T_n^{\rm GOF (adapt)}$ exceeds the upper $\alpha$ quantile, denoted by $q_{n,\alpha}^{\rm GOF}$, of its null distribution. In what follows, we shall call this test $\Phi^{\rm GOF (adapt)}$. Note that, even though it is hard to derive the analytic form for $q_{n,\alpha}^{\rm GOF}$, it can be readily evaluated via Monte Carlo method.

To study the power of $\Phi^{\rm GOF (adapt)}$ against $H_1^{\rm GOF}$ with different levels of smoothness, we shall consider the following alternative hypothesis
$$
H_1^{\rm GOF(adapt)}(\Delta_{n,s}: s\ge d/4): p\in \bigcup_{s\ge d/4} \{p\in \Wcal^{s,2}(M): \|p-p_0\|_{L_2}\ge \Delta_{n,s}\}.
$$
The following theorem characterizes the power of $\Phi^{\rm GOF (adapt)}$ against $H_1^{\rm GOF(adapt)}(\Delta_{n,s}: s\ge d/4)$.

\begin{theorem}
\label{th:gofadapt}
There exists a constant $c>0$ such that if
$$\liminf_{n\to\infty} \Delta_{n,s}(n/\log\log n)^{2s/(d+4s)}>c,$$
then
$$
{\rm power}\{\Phi^{\rm GOF (adapt)}; H_1^{\rm GOF(adapt)}(\Delta_{n,s}: s\ge d/4)\}\to 1.
$$	
\end{theorem}
Theorem \ref{th:gofadapt} shows that $\Phi^{\rm GOF (adapt)}$ has a detection boundary of the order $(\log\log n/n)^{\frac{2s}{d+4s}}$ when $p\in \Wcal^{s,2}$ for any $s\ge d/4$. If $s$ is known in advance, as we show in Section \ref{sec:gof}, the optimal test is based on $T_{n,\nu_n}^{\rm GOF}$ with $\nu_n\asymp n^{4/(d+4s)}$ and has a detection boundary of the order $O(n^{-2s/(d+4s)})$. The extra polynomial of iterated logarithmic factor $(\log\log n)^{2s/(d+4s)}$ is the price we pay to ensure that no knowledge of $s$ is required and $\Phi^{\rm GOF (adapt)}$ is powerful against smooth alternatives for all $s\ge d/4$.

\subsection{Test for Homogeneity}
The treatment for homogeneity tests is similar. Instead of $T_{n,\nu_n}^{\rm HOM}$, we now consider a test based on
$$
T_n^{\rm HOM (adapt)}=\max_{1\le \nu_n\le n^{2/d}} T_{n,\nu_n}^{\rm HOM}.
$$
If $T_n^{\rm HOM (adapt)}$ exceeds the upper $\alpha$ quantile, denoted by $q_{n,\alpha}^{\rm HOM}$, of its null distribution, then we reject $H_0^{\rm HOM}$. In what follows, we shall refer to this test as $\Phi^{\rm HOM (adapt)}$. As before, we do not have a closed form expression for $q_{n,\alpha}^{\rm HOM}$, and it needs to be evaluated via Monte Carlo method. In particular, in the case of homogeneity test, we can approximate $q_{n,\alpha}^{\rm HOM}$ by permutation where we randomly shuffle $\{X_1,\ldots,X_n, Y_1,\ldots,Y_m\}$ and compute the test statistic as if the first $n$ shuffled observations are from the first population whereas the other $m$ are from the second population. This is repeated multiple times in order to approximate the critical value $q_{n,\alpha}^{\rm HOM}$.

The following theorem characterize the power of $\Phi^{\rm HOM (adapt)}$ against an alternative with different levels of smoothness
$$
H_1^{\rm HOM(adapt)}(\Delta_{n,s}: s\ge d/4): (p,q)\in \bigcup_{s\ge d/4} \{(p,q): p,q\in \Wcal^{s,2}(M), \|p-q\|_{L_2}\ge \Delta_{n,s}\}.
$$

\begin{theorem}
\label{th:homadapt}
There exists a constant $c>0$ such that if
$$\liminf_{n\to\infty} \Delta_{n,s}(n/\log\log n)^{2s/(d+4s)}>c,$$
then
$$
{\rm power}\{\Phi^{\rm HOM (adapt)}; H_1^{\rm HOM(adapt)}(\Delta_{n,s}: s\ge d/4)\}\to 1.
$$
\end{theorem}
Similar to the case of goodness-of-fit test, Theorem \ref{th:homadapt} shows that $\Phi^{\rm HOM (adapt)}$ has a detection boundary of the order $O((n/\log\log n)^{-2s/(d+4s)})$ when $p\neq q\in \Wcal^{s,2}$ for any $s\ge d/4$. In light of the results from Section \ref{sec:hom}, this is optimal up to an extra polynomial of iterated logarithmic factor. The main advantage is that $\Phi^{\rm HOM (adapt)}$ is powerful against smooth alternatives simultaneously for all $s\ge d/4$.

\subsection{Test for Independence}
Similarly, for independence test, we shall adopt the following test statistic
$$
T_n^{\rm IND (adapt)}=\max_{1\le \nu_n\le n^{2/d}} T_{n,\nu_n}^{\rm IND}.
$$
and reject $H_0^{\rm IND}$ if and only $T_n^{\rm IND (adapt)}$ exceeds the upper $\alpha$ quantile, denoted by $q_{n,\alpha}^{\rm IND}$, of its null distribution. In what follows, we shall refer to this test as $\Phi^{\rm HOM (adapt)}$. The critical value, $q_{n,\alpha}^{\rm HOM}$, can also be evaluated via permutation test. See, e.g., \cite{pfister2018kernel} for detailed discussions.

We now show that $\Phi^{\rm IND (adapt)}$ is powerful in testing against the alternative with different levels of smoothness
\begin{align*}
H_1^{\rm IND(adapt)}(\Delta_{n,s}: s\ge d/4): p\in \bigcup_{s\ge d/4} \Big\{p\in \Wcal^{s,2}(M), p_j\in\Wcal^{s,2}(M_j),1\leq j\leq k,\\ \|p-p_1\otimes\cdots\otimes p_k\|_{L_2}\ge \Delta_{n,s}\Big\}.
\end{align*}
More specifically, we have

\begin{theorem}
\label{th:indadapt}
There exists a constant $c>0$ such that if
$$\liminf_{n\to\infty} \Delta_{n,s}(n/\log\log n)^{2s/(d+4s)}>c,$$
then
$$
{\rm power}\{\Phi^{\rm IND (adapt)}; H_1^{\rm IND(adapt)}(\Delta_{n,s}: s\ge d/4)\}\to 1.
$$
\end{theorem}
Similar to before, Theorem \ref{th:indadapt} shows that $\Phi^{\rm IND (adapt)}$ is optimal up to an extra polynomial of iterated logarithmic factor for detecting smooth departure from independence simultaneously for all $s\ge d/4$.

\section{Numerical Experiments}
\label{sec:sim}
To further complement our theoretical development and demonstrate the practical merits of the proposed methodology, we conducted several sets of numerical experiments.

\subsection{Effect of Scaling Parameter}
Our first set of experiments were designed to illustrate the importance of the scaling parameter and highlight the potential room for improvement over the ``median'' heuristic -- one of the most common data-driven choice of the scaling parameter in practice \citep[see, \eg,][]{gretton2008kernel,pfister2018kernel}.

\begin{itemize}
	\item \textit{Experiment \Ro}: the homogeneity test with underlying distributions being the normal distribution and the mixture of several normal distributions. Specifically, $$
	p(x)=f(x;0,1),\quad q(x)=0.5\times f(x;0,1)+0.1\times\sum_{\mu\in\bm{\mu}}f(x;\mu,0.05)
	$$
	where $f(x;\mu,\sigma)$ denotes the density of $N(\mu,\sigma^2)$ and $\bm{\mu}=\{-1,-0.5,0,0.5,1\}$.
	\item \textit{Experiment \Rt:} the joint independence test of $X^1,\cdots,X^5$ where $$X^1,\cdots,X^{4},(X^5)'\sim_{\rm iid} N(0,1),\quad
	X^5=\left|(X^5)'\right|\times \mathrm{sign}\left(\prod\limits_{l=1}^{4}X^l\right).
	$$
	Clearly $X^1,\cdots,X^5$ are jointly dependent since $\prod_{l=1}^dX^l\geq 0$.
\end{itemize}

In both experiments, our primary goal is to investigate how the power of Gaussian MMD based test is influenced by a pre-fixed scaling parameter. These tests are also compared to the ones with scaling parameter selected via ``median'' heuristic. In order to evaluate tests with different scaling parameters under a unified framework, we determined the critical values for each test via permutation test.

For Experiment \Ro\  we fixed the sample size at $n=m=200$; and for Experiment \Rt\ at $n=400$. The number of permutations was set at $100$, and significance level at $\alpha=0.05$. We first repeated the experiments $100$ times under the null to verify that permutation tests indeed yield the correct size, up to Monte Carlo error. Each experiment was then repeated for 100 times and the observed power ($\pm$ one standard error) for different choices of the scaling parameter. The results are summarized in Figure \ref{Fg:single}. It is perhaps not surprising that the scaling parameter selected via ``median heuristic'' has little variation across each simulation run, and we represent its performance by a single value.

\begin{figure*}[!htbp]
	\begin{minipage}{0.5\textwidth}
		\centering
		\begin{tikzpicture}
		\begin{axis}[
		height=6.5cm,
		width=8cm,
		grid=major,
		xlabel = $\log(\nu)$,
		ylabel=Power,
		xmax = 4,
		xmin = -1,
		ymax = 1,
		ymin = 0,
		xtick={-1,0,1,2,3,4},
		legend style={at={(0.03,0.95)},anchor=north west,
		}
		]
		\addplot[mark=*,cyan,error bars/.cd,
		y dir=both,y explicit]
		coordinates {
			(-1,0.23)+-(0,0.03) (-0.75,0.23)+-(0,0.03) (-0.5,0.24)+-(0,0.03) (-0.25,0.22)+-(0,0.029) (0,0.22)+-(0,0.029) (0.25,0.2)+-(0,0.028) (0.5,0.2)+-(0,0.028) (0.75,0.2)+-(0,0.028) (1,0.22)+-(0,0.029) (1.25,0.22)+-(0,0.029) (1.5,0.23)+-(0,0.03) (1.75,0.23)+-(0,0.03) (2,0.23)+-(0,0.03) (2.125,0.25)+-(0,0.031) (2.25,0.27)+-(0,0.031) (2.375,0.33)+-(0,0.033) (2.5,0.43)+-(0,0.035) (2.625,0.54)+-(0,0.035) (2.75,0.67)+-(0,0.033) (2.875,0.75)+-(0,0.031) (3,0.82)+-(0,0.027) (3.125,0.91)+-(0,0.02) (3.25,0.94)+-(0,0.017) (3.375,0.98)+-(0,0.01) (3.5,0.99)+-(0,0.007) (3.625,0.99)+-(0,0.007) (3.75,0.99)+-(0,0.007) (3.875,0.99)+-(0,0.007) (4,1)+-(0,0)
		};  \addlegendentry{Single fixed $\nu$} ;
		\addplot[mark size=3,mark=*,red,only marks,error bars/.cd,
		y dir=both,y explicit]
		coordinates {
			(0.2,0.21)+-(0,0.04)
		};  \addlegendentry{Median} ;
		\end{axis}
		\end{tikzpicture}
	\end{minipage}
	\begin{minipage}{0.5\textwidth}
		\centering
		\begin{tikzpicture}
		\begin{axis}[
		height=6.5cm,
		width=8cm,
		grid=major,
		xlabel = $\log(\nu)$,
		xmax = 2,
		xmin = -3,
		ymax = 1,
		ymin = 0,
		xtick={-3,-2,-1,0,1,2},
		legend style={at={(0.03,0.95)},anchor=north west,
		}
		]
		\addplot[mark=*,cyan,error bars/.cd,
		y dir=both,y explicit]
		coordinates {
			(-3,0.08)+-(0,0.014) (-2.75,0.08)+-(0,0.014) (-2.5,0.08)+-(0,0.014) (-2.25,0.07)+-(0,0.013) (-2,0.05)+-(0,0.011) (-1.75,0.06)+-(0,0.012) (-1.5,0.09)+-(0,0.014) (-1.25,0.11)+-(0,0.016) (-1,0.21)+-(0,0.02) (-0.75,0.27)+-(0,0.022) (-0.625,0.34)+-(0,0.024) (-0.5,0.45)+-(0,0.025) (-0.375,0.55)+-(0,0.025) (-0.25,0.62)+-(0,0.024) (-0.125,0.71)+-(0,0.023) (0,0.75)+-(0,0.022) (0.125,0.82)+-(0,0.019) (0.25,0.84)+-(0,0.018) (0.375,0.86)+-(0,0.017) (0.5,0.92)+-(0,0.014) (0.625,0.92)+-(0,0.014) (0.75,0.93)+-(0,0.013) (0.875,0.93)+-(0,0.013) (1,0.93)+-(0,0.013) (1.25,0.9)+-(0,0.015) (1.375,0.87)+-(0,0.017) (1.5,0.86)+-(0,0.017) (1.625,0.85)+-(0,0.018) (1.75,0.82)+-(0,0.019) (1.875,0.81)+-(0,0.02) (2,0.8)+-(0,0.02) 
		};  
		\addplot[mark size=3,mark=*,red,only marks,error bars/.cd,
		y dir=both,y explicit]
		coordinates {
			(-2.15,0.07)+-(0,0.013)
		};  
		\end{axis}
		\end{tikzpicture}
	\end{minipage}
	\caption{Observed power against $\log(\nu)$ in Experiment \Ro\ (left) and Experiment \Rt (right).}\label{Fg:single}
\end{figure*}
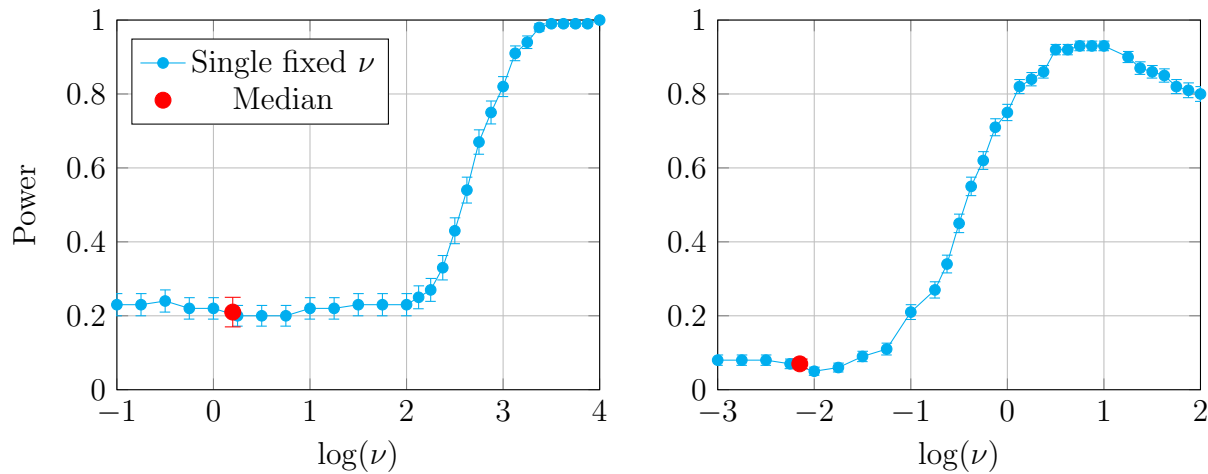

The importance of the scaling parameter is evident from Figure \ref{Fg:single} with the observed power varies quite significantly for different choices. It is also of interest to note that in these settings the ``median'' heuristic typically does not yield a scaling parameter with great power. More specifically, in Experiment \Ro, $\log(\nu_{\rm median})\approx 0.2$ and maximum power is attained at $\log(\nu)=4$; in Experiment \Rt,  $\log(\nu_{\rm median})\approx-2.15$ and maximum power is attained at $\log(\nu)=1$. This suggests that more appropriate choice of the scaling parameter may lead to much improved performance.

\subsection{Efficacy of Adaptation}\label{sec:sim_adapt}
Our second experiment aims to illustrate that the adaptive procedures we proposed in Section \ref{sec:adapt} indeed yield more powerful tests when compared with other alternatives that are commonly used in practice. In particular, we compare the proposed self-normalized adaptive test (\verb+S.A.+) with a couple of data-driven approaches, namely the ``median'' heuristic (\verb+Median+) and the unnormalized adaptive test (\verb+U.A.+) proposed in \cite{sriperumbudur2009kernel}. When computing both self-normalized and unnormalized test statistics, we first rescaled the squared distance $\|X_i-X_j\|^2$ by the dimensionality $d$ before taking maximum within a certain range of the scaling parameter. We considered two experiment setups:

%
%

\begin{itemize}
	\item \textit{Experiment \Rthree}: the homogeneity test with the underlying distributions being
	$$
	P\sim N(\mathbf{0},I_d),\quad Q\sim N\left(\mathbf{0},\left(1+2d^{-1/2}\right)I_d\right).
	$$
	As the `signal strength', the ratio between the variances of $Q$ and $P$ in each single direction is set to decrease to $1$ at the order $1/\sqrt{d}$ with $d$, which is the decreasing order of variance ratio that can be detected by the classical $F$-test.
	\item \textit{Experiment \Rfour}: the independence test of $X^1,X^2\in\RR^{d/2}$, where $X=(X^1,X^2)$ follows a mixture of $$N\left(\mathbf{0},I_d\right)\quad \text{and}\quad N\left(\mathbf{0},(1+6d^{-3/5})I_d\right)$$
	with mixture probability being $0.5$. Similarly, the ratio between the variances in each direction is set to decrease with $d$, but at a slightly higher rate.
\end{itemize}

To better compare different methods, we considered different combinations of sample size and dimensionality for each experiment. More specifically, for Experiment \Rthree, the sample sizes were set to be $m=n=25,50,75,\cdots,200$ and dimension $d=1,10,100,1000$; for Experiment \Rfour, the sample size were $n=100,200,\cdots,600$ and dimension $d=2,10,100,1000$. In both experiments, we fixed the significance level at $\alpha=0.05$, did $100$ permutations to calibrate the critical values as before. Again we simulated under $H_0$ to verify that the resulting tests have the targeted size, up to Monte Carlo error. The power of each method, estimated from $100$ such experiments, is reported in Figures \ref{Fg: adapt1} and \ref{Fg: adapt2}.

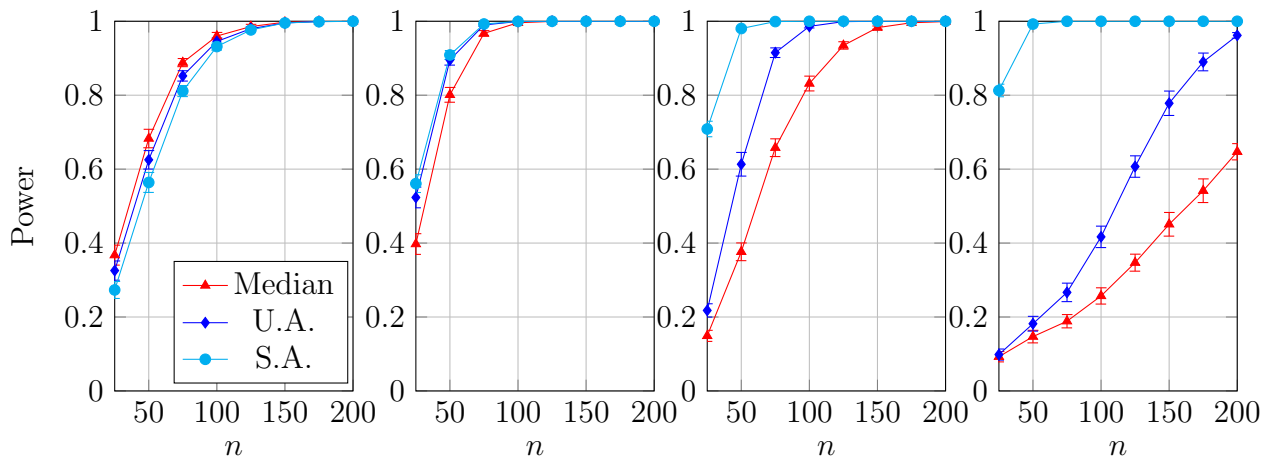
\begin{figure}[!htbp]
	\begin{minipage}{0.285\textwidth}
		\centering
		\begin{tikzpicture}
		\begin{axis}[
		height=6.5cm,
		width=4.75cm,
		grid=major,
		xlabel = $n$,
		ylabel=Power,
		xmax = 200,
		xmin = 25,
		ymax = 1,
		ymin = 0,
		xtick={50,100,150,200},
		legend style={at={(0.98,0.02)},anchor=south east
		}
		]
		\addplot[mark=triangle*,red,error bars/.cd,
		y dir=both,y explicit]
		coordinates {
			(25,0.3674)+-(0,0.027)
			(50,0.6828)+-(0,0.025)
			(75,0.8882)+-(0,0.011)
			(100,0.96)+-(0,0.0098)
			(125,0.9858)+-(0,0.0059)
			(150,0.9976)+-(0,0.0018)
			(175,0.9994)+-(0,0.00068)
			(200,1)+-(0,0)
		};  \addlegendentry{Median} ;
		\addplot[mark=diamond*,blue,error bars/.cd,
		y dir=both,y explicit]
		coordinates {
			(25,0.3256)+-(0,0.026)
			(50,0.6252)+-(0,0.025)
			(75,0.8522)+-(0,0.014)
			(100,0.9454)+-(0,0.0093)
			(125,0.98)+-(0,0.0055)
			(150,0.9954)+-(0,0.0023)
			(175,0.9994)+-(0,0.00068)
			(200,1)+-(0,0)
		};  \addlegendentry{U.A.} ;
		\addplot[mark=*,cyan,error bars/.cd,
		y dir=both,y explicit]
		coordinates {
			(25,0.2732)+-(0,0.023)
			(50,0.564)+-(0,0.027)
			(75,0.8112)+-(0,0.015)
			(100,0.9316)+-(0,0.0097)
			(125,0.9766)+-(0,0.0068)
			(150,0.9954)+-(0,0.0023)
			(175,0.9988)+-(0,0.00093)
			(200,1)+-(0,0)
			
		};  \addlegendentry{S.A.} ;
		\end{axis}
		\end{tikzpicture}
	\end{minipage}
	\begin{minipage}{0.235\textwidth}
		\centering
		\begin{tikzpicture}
		\begin{axis}[
		xlabel=$n$,
		height=6.5cm,
		width=4.75cm,
		grid=major,
		xmax = 200,
		xmin = 25,
		ymax = 1,
		ymin = 0,
		xtick={50,100,150,200},
		legend style={at={(1,0)},anchor=south east,
			nodes={scale=0.7, transform shape}
		}
		]
		\addplot[mark=triangle*,red,error bars/.cd,
		y dir=both,y explicit]
		coordinates {
			(25,0.3974)+-(0,0.028)
			(50,0.801)+-(0,0.02)
			(75,0.9672)+-(0,0.0075)
			(100,0.996)+-(0,0.0023)
			(125,0.9996)+-(0,0.00056)
			(150,1)+-(0,0)
			(175,1)+-(0,0)
			(200,1)+-(0,0)
		};  
		\addplot[mark=diamond*,blue,error bars/.cd,
		y dir=both,y explicit]
		coordinates {
			(25,0.5234)+-(0,0.028)
			(50,0.8954)+-(0,0.014)
			(75,0.99)+-(0,0.0038)
			(100,0.9994)+-(0,0.00068)
			(125,1)+-(0,0)
			(150,1)+-(0,0)
			(175,1)+-(0,0)
			(200,1)+-(0,0)
		};  
		\addplot[mark=*,cyan,error bars/.cd,
		y dir=both,y explicit]
		coordinates {
			(25,0.5608)+-(0,0.024)
			(50,0.9088)+-(0,0.012)
			(75,0.9926)+-(0,0.0029)
			(100,0.9996)+-(0,0.00056)
			(125,1)+-(0,0)
			(150,1)+-(0,0)
			(175,1)+-(0,0)
			(200,1)+-(0,0)
		};  
		\end{axis}
		\end{tikzpicture}
	\end{minipage}
	\begin{minipage}{0.235\textwidth}
		\centering
		\begin{tikzpicture}
		\begin{axis}[
		height=6.5cm,
		width=4.75cm,
		grid=major,
		xlabel=$n$,
		xmax = 200,
		xmin = 25,
		ymax = 1,
		ymin = 0,
		xtick={50,100,150,200},
		legend style={at={(1,0)},anchor=south east
		}
		]
		\addplot[mark=triangle*,red,error bars/.cd,
		y dir=both,y explicit]
		coordinates {
			((25,0.149)+-(0,0.015)
			(50,0.3764)+-(0,0.024)
			(75,0.658)+-(0,0.024)
			(100,0.8316)+-(0,0.02)
			(125,0.9346)+-(0,0.01)
			(150,0.9828)+-(0,0.0044)
			(175,0.996)+-(0,0.002)
			(200,0.9992)+-(0,0.00096)
		};  
		\addplot[mark=diamond*,blue,error bars/.cd,
		y dir=both,y explicit]
		coordinates {
			(25,0.2178)+-(0,0.018)
			(50,0.6132)+-(0,0.032)
			(75,0.915)+-(0,0.013)
			(100,0.9862)+-(0,0.0043)
			(125,0.999)+-(0,0.00086)
			(150,1)+-(0,0)
			(175,1)+-(0,0)
			(200,1)+-(0,0)
		};  
		\addplot[mark=*,cyan,error bars/.cd,
		y dir=both,y explicit]
		coordinates {
			(25,0.7086)+-(0,0.021)
			(50,0.9804)+-(0,0.0041)
			(75,0.9992)+-(0,0.00078)
			(100,1)+-(0,0)
			(125,1)+-(0,0)
			(150,1)+-(0,0)
			(175,1)+-(0,0)
			(200,1)+-(0,0)
		};  
		\end{axis}
		\end{tikzpicture}
	\end{minipage}
	\begin{minipage}{0.235\textwidth}
		\centering
		\begin{tikzpicture}
		\begin{axis}[
		height=6.5cm,
		width=4.75cm,
		grid=major,
		xlabel=$n$,
		xmax = 200,
		xmin = 25,
		ymax = 1,
		ymin = 0,
		xtick={50,100,150,200},
		legend style={at={(1,0)},anchor=south east,
			nodes={scale=0.7, transform shape}
		}
		]
		\addplot[mark=triangle*,red,error bars/.cd,
		y dir=both,y explicit]
		coordinates {
			(25,0.0926)+-(0,0.014)
			(50,0.147)+-(0,0.017)
			(75,0.1886)+-(0,0.018)
			(100,0.257)+-(0,0.022)
			(125,0.347)+-(0,0.023)
			(150,0.4508)+-(0,0.032)
			(175,0.5416)+-(0,0.032)
			(200,0.647)+-(0,0.022)
		};  
		\addplot[mark=diamond*,blue,error bars/.cd,
		y dir=both,y explicit]
		coordinates {
			(25,0.0984)+-(0,0.015)
			(50,0.1818)+-(0,0.02)
			(75,0.2666)+-(0,0.025)
			(100,0.4168)+-(0,0.029)
			(125,0.607)+-(0,0.029)
			(150,0.7782)+-(0,0.033)
			(175,0.89)+-(0,0.024)
			(200,0.9618)+-(0,0.0076)
		};  
		\addplot[mark=*,cyan,error bars/.cd,
		y dir=both,y explicit]
		coordinates {
			(25,0.8128)+-(0,0.017)
			(50,0.9924)+-(0,0.0027)
			(75,1)+-(0,0)
			(100,1)+-(0,0)
			(125,1)+-(0,0)
			(150,1)+-(0,0)
			(175,1)+-(0,0)
			(200,1)+-(0,0)
		};  
		\end{axis}
		\end{tikzpicture}
	\end{minipage}
	\caption{Observed power versus sample size in Experiment \Rthree\ for $d=1,10,100,1000$ from left to right.}\label{Fg: adapt1}
\end{figure}

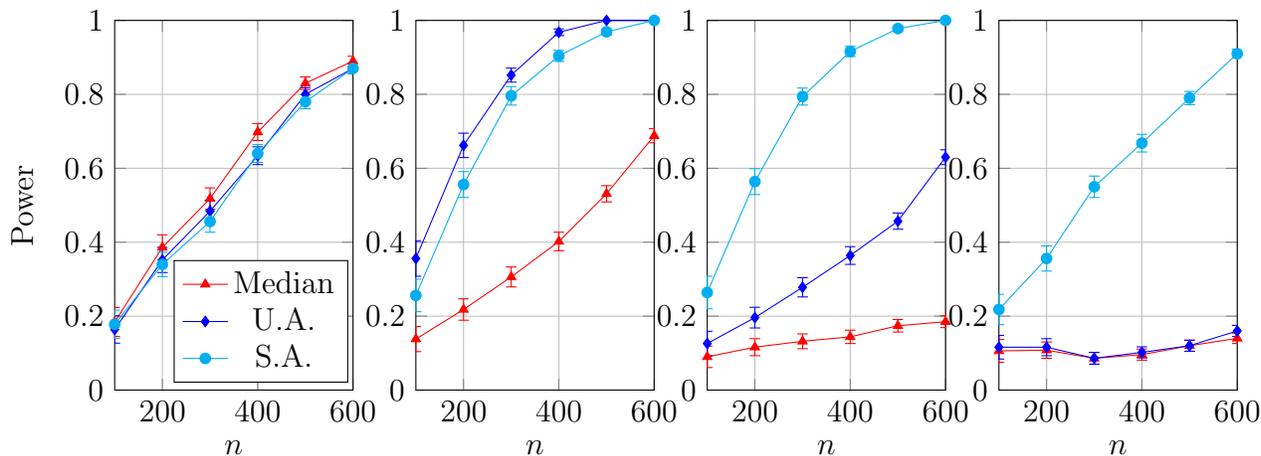
\begin{figure}[!htbp]
	\begin{minipage}{0.285\textwidth}
		\centering
		\begin{tikzpicture}
		\begin{axis}[
		height=6.5cm,
		width=4.75cm,
		grid=major,
		xlabel = $n$,
		ylabel=Power,
		xmax = 600,
		xmin = 100,
		ymax = 1,
		ymin = 0,
		xtick={200,400,600},
		legend style={at={(0.98,0.02)},anchor=south east
		}
		]
		\addplot[mark=triangle*,red,error bars/.cd,
		y dir=both,y explicit]
		coordinates {
			(100,0.184)+-(0,0.039)  (200,0.386)+-(0,0.034)  (300,0.518)+-(0,0.029) (400,0.698)+-(0,0.023) 
			 (500,0.83)+-(0,0.017)  (600,0.89)+-(0,0.013)
		};  \addlegendentry{Median} ;
		\addplot[mark=diamond*,blue,error bars/.cd,
		y dir=both,y explicit]
		coordinates {
			(100,0.164)+-(0,0.037)  (200,0.352)+-(0,0.034) (300,0.484)+-(0,0.029) (400,0.634)+-(0,0.024)
			(500,0.8)+-(0,0.018)  (600,0.87)+-(0,0.014)  
		};  \addlegendentry{U.A.} ;
		\addplot[mark=*,cyan,error bars/.cd,
		y dir=both,y explicit]
		coordinates {
			(100,0.178)+-(0,0.038)  (200,0.34)+-(0,0.033)  (300,0.456)+-(0,0.029) (400,0.64)+-(0,0.024) (500,0.78)+-(0,0.019) (600,0.87)+-(0,0.014) 
		};  \addlegendentry{S.A.} ;
		\end{axis}
		\end{tikzpicture}
	\end{minipage}
	\begin{minipage}{0.235\textwidth}
		\centering
		\begin{tikzpicture}
		\begin{axis}[
		xlabel=$n$,
		height=6.5cm,
		width=4.75cm,
		grid=major,
		xmax = 600,
		xmin = 100,
		ymax = 1,
		ymin = 0,
		xtick={200,400,600},
		legend style={at={(1,0)},anchor=south east,
			nodes={scale=0.7, transform shape}
		}
		]
		\addplot[mark=triangle*,red,error bars/.cd,
		y dir=both,y explicit]
		coordinates {
			(100,0.138)+-(0,0.034)  (200,0.218)+-(0,0.029) (300,0.306)+-(0,0.027)  (400,0.402)+-(0,0.025) 
			(500,0.531)+-(0,0.022)  (600,0.688)+-(0,0.019) 
		};  
		\addplot[mark=diamond*,blue,error bars/.cd,
		y dir=both,y explicit]
		coordinates {
			(100,0.356)+-(0,0.048)  (200,0.662)+-(0,0.033)  (300,0.852)+-(0,0.019)  (400,0.968)+-(0,0.009) (500,1)+-(0,0) (600,1)+-(0,0) 
		};  
		\addplot[mark=*,cyan,error bars/.cd,
		y dir=both,y explicit]
		coordinates {
			(100,0.256)+-(0,0.044)  (200,0.556)+-(0,0.035) (300,0.796)+-(0,0.025) (400,0.904)+-(0,0.015) 
			(500,0.969)+-(0,0.008) (600,1)+-(0,0) 
		};  
		\end{axis}
		\end{tikzpicture}
	\end{minipage}
	\begin{minipage}{0.235\textwidth}
		\centering
		\begin{tikzpicture}
		\begin{axis}[
		height=6.5cm,
		width=4.75cm,
		grid=major,
		xlabel=$n$,
		xmax = 600,
		xmin = 100,
		ymax = 1,
		ymin = 0,
		xtick={200,400,600},
		legend style={at={(1,0)},anchor=south east
		}
		]
		\addplot[mark=triangle*,red,error bars/.cd,
		y dir=both,y explicit]
		coordinates {
			(100,0.09)+-(0,0.029)  (200,0.116)+-(0,0.023) (300,0.132)+-(0,0.02)  (400,0.144)+-(0,0.018)  (500,0.174)+-(0,0.017)  (600,0.185)+-(0,0.016)
		};  
		\addplot[mark=diamond*,blue,error bars/.cd,
		y dir=both,y explicit]
		coordinates {
			(100,0.126)+-(0,0.033)  (200,0.196)+-(0,0.028)  (300,0.278)+-(0,0.026)  (400,0.364)+-(0,0.024) 
			(500,0.457)+-(0,0.022)  (600,0.63)+-(0,0.02) 
		};  
		\addplot[mark=*,cyan,error bars/.cd,
		y dir=both,y explicit]
		coordinates {
			(100,0.264)+-(0,0.044)  (200,0.564)+-(0,0.035) (300,0.794)+-(0,0.023) (400,0.916)+-(0,0.014) (500,0.978)+-(0,0.007) (600,1)+-(0,0) 
		};  
		\end{axis}
		\end{tikzpicture}
	\end{minipage}
	\begin{minipage}{0.235\textwidth}
		\centering
		\begin{tikzpicture}
		\begin{axis}[
		height=6.5cm,
		width=4.75cm,
		grid=major,
		xlabel=$n$,
		xmax = 600,
		xmin = 100,
		ymax = 1,
		ymin = 0,
		xtick={200,400,600},
		legend style={at={(1,0)},anchor=south east,
			nodes={scale=0.7, transform shape}
		}
		]
		\addplot[mark=triangle*,red,error bars/.cd,
		y dir=both,y explicit]
		coordinates {
			(100,0.106)+-(0,0.031)  (200,0.108)+-(0,0.022)  (300,0.086)+-(0,0.016)  (400,0.096)+-(0,0.015)  (500,0.12)+-(0,0.015) (600,0.14)+-(0,0.014)
			
		};  
		\addplot[mark=diamond*,blue,error bars/.cd,
		y dir=both,y explicit]
		coordinates {
			(100,0.116)+-(0,0.032) (200,0.116)+-(0,0.023) (300,0.086)+-(0,0.016)  (400,0.102)+-(0,0.015) (500,0.12)+-(0,0.015)  (600,0.16)+-(0,0.015) 
		};  
		\addplot[mark=*,cyan,error bars/.cd,
		y dir=both,y explicit]
		coordinates {
			(100,0.218)+-(0,0.041)  (200,0.356)+-(0,0.034) (300,0.55)+-(0,0.029)  (400,0.668)+-(0,0.024) (500,0.79)+-(0,0.018) (600,0.91)+-(0,0.012) 
		};  
		\end{axis}
		\end{tikzpicture}
	\end{minipage}
	\caption{Observed power versus sample size in Experiment \Rfour\ for $d=2,10,100,1000$ from left to right.}\label{Fg: adapt2}
\end{figure}

As Figures \ref{Fg: adapt1} and \ref{Fg: adapt2} show, for both experiments, these tests are comparable in low-dimensional settings. But as $d$ increases, the proposed self-normalized adaptive test becomes more and more preferable to the two alternatives. For example, for Experiment \Rfour, when $d=1000$, the observed power of the proposed self-normalized adaptive test is about $90\%$ when $n=600$, while the other two tests have power around only $15\%$.

\subsection{Data Example}
Finally, we considered applying the proposed self-normalized adaptive test in a data example from \cite{mooij2016distinguishing}. The dataset consists of three variables, altitude (Alt), average temperature (Temp) and average duration of sunshine (Sun) from different weather stations. One goal of interest is to figure out the causal relationship among the three variables by figuring out a suitable directed acyclic graph (DAG) among them. Following \cite{peters2014causal}, if a set of random variables $X^1,\cdots$,$X^d$ follow a DAG $\Gcal_0$, then we assume that they follow a sequence of additive models:
$$
X^l=\sum\limits_{r\in \mathrm{PA}^l}f_{l,r}(X^{r})+N^l,\quad\forall\ 1\leq l\leq d,
$$
where $N^l$'s are independent Gaussian noises and $\mathrm{PA}^l$ denotes the collection of parent nodes of node $l$ specified  by $\Gcal_0$. As shown by \citep{peters2014causal}, $\Gcal_0$ is identifiable from the joint distribution of $X^1,\cdots,X^d$ under the assumption of $f_{l,r}$'s being non-linear. Therefore a natural method of deciding a specific DAG underlying a set of random variables is by testing the independence of the regression residuals after fitting the DAG induced additive models. In our case, there are totally $25$ possible DAGs for the three variables. We can apply independence tests for the residuals for each of the 25 DAGs and choose the one with the largest $p$-value as the most plausible underlying DAG. See \cite{peters2014causal} for more details. 

As before, we considered three different ways for independence tests: the proposed self-normalized adaptive test (\verb+S.A.+), Gaussian kernel embedding based independent test with the scaling parameter determined by the ``median'' heuristic (\verb+Median+), and the unnormalized adaptive test from \cite{sriperumbudur2009kernel} (\verb+U.A.+). Note that the three variables have different scales and we standardize them before applying the tests of independence.

The overall sample size of the dataset is $349$. Each time we randomly select $150$ samples and compute the $p$-value associated with each DAG. The $p$-value is again computed based on $100$ permutations. We repeated the experiment for $1000$ times and recorded for each test the DAG with the largest $p$-value. All three tests agree on the top three most selected DAGs and they are shown in Figure \ref{fig:DAG}.

\begin{figure}[!htbp]
	\begin{subfigure}[b]{.32\linewidth}
		\resizebox{\linewidth}{!}{
			\begin{tikzpicture}
			\node[draw,circle,fill=blue!40,minimum size=1.5cm] (A) at (0,0) {Alt};
			\node[draw,circle,fill=blue!40,minimum size=1.5cm] (T) at (-2,-4) {Temp};
			\node[draw,circle,fill=blue!40,minimum size=1.5cm] (S) at (2,-4) {Sun};
			\draw [-{Stealth[length=2mm]}]
			(A) edge (T) (A) edge (S) (T) edge (S);
			\end{tikzpicture}
		}
		\caption*{\textbf{DAG \Ro}}
	\end{subfigure}\hfill%
    \begin{subfigure}[b]{.32\linewidth}
    	\resizebox{\linewidth}{!}{
    		\begin{tikzpicture}
    		\node[draw,circle,fill=blue!40,minimum size=1.5cm] (A) at (0,0) {Alt};
    		\node[draw,circle,fill=blue!40,minimum size=1.5cm] (T) at (-2,-4) {Temp};
    		\node[draw,circle,fill=blue!40,minimum size=1.5cm] (S) at (2,-4) {Sun};
    		\draw [-{Stealth[length=2mm]}]
    		(A) edge (T) (A) edge (S) (S) edge (T);
    		\end{tikzpicture}
    	}
    	\caption*{DAG \Rt}
    \end{subfigure}\hfill%
	\begin{subfigure}[b]{.32\linewidth}
		\resizebox{\linewidth}{!}{
			\begin{tikzpicture}
			\node[draw,circle,fill=blue!40,minimum size=1.5cm] (A) at (0,0) {Alt};
			\node[draw,circle,fill=blue!40,minimum size=1.5cm] (T) at (-2,-4) {Temp};
			\node[draw,circle,fill=blue!40,minimum size=1.5cm] (S) at (2,-4) {Sun};
			\draw [-{Stealth[length=2mm]}]
			(A) edge (T) (S) edge (A) (S) edge (T);
			\end{tikzpicture}
		}
		\caption*{DAG \Rthree}
	\end{subfigure}
	\caption{DAGs with the top 3 highest probabilities of being selected.}
	\label{fig:DAG}
\end{figure}
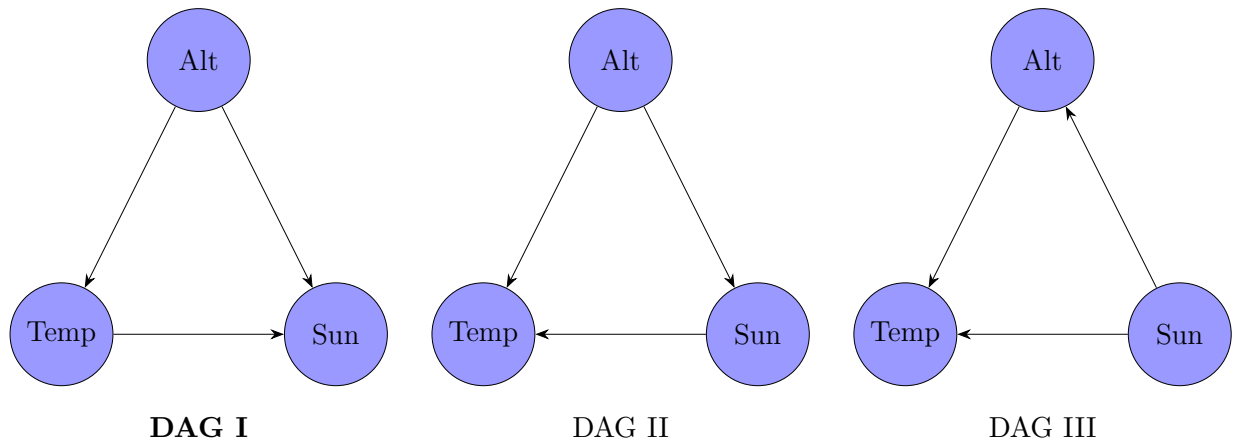

In addition, we report in Table \ref{tab:DAG} the frequencies that these three DAGs were selected by each of the tests. They are generally comparable with the proposed method more consistently selecting DAG I, the one heavily favored by all three methods.

\begin{table}[!htbp]
	\centering
	\begin{tabular}{?c?C{3cm}C{3cm}C{3cm}?}
		\thickhline
		\diagbox{Test}{Prob($\%$)}{DAG} & I & II & III \\
		\thickhline
		Median & 78.5 &4.7& 14.5  \\
		U.A. & 81.4&  8.1&8.5 \\
		S.A. &83.4 &9.8&4.7\\
		\thickhline
	\end{tabular}
	\caption{Frequency that each DAG in Figure \ref{fig:DAG} was selected by three tests.}\label{tab:DAG}
\end{table}


\section{Concluding Remarks}
\label{sec:disc}
In this paper, we provide a systematic investigation of the statistical properties of Gaussian kernel embedding based nonparametric tests. Our contribution is twofold. 

First of all, we provide theoretical justifications for this popular class of methods by showing that they are capable of detecting the smallest possible deviation from the null hypotheses in the context of goodness-of-fit, homogeneity, and independence test. Our analyses also suggest that the existing theoretical studies do not fully explain the practical success of these methods because they assume a fixed kernel or scaling parameter for Gaussian kernel and these methods, as we argue, are most powerful with a varying scaling parameter.

From a more practical viewpoint, we offer general guidelines on choosing the scaling parameter for Gaussian kernels: our results highlight the importance of using larger scaling parameter for larger sample size and establish the relationship between the smoothness of the underlying densities and the appropriate scaling parameter. Furthermore, we introduce new adaptive testing procedures for goodness-of-fit, homogeneity, and independence respectively that are optimal, up to a polynomial of iterated logarithmic factor, for a wide range of smooth densities while not needing to know the level of smoothness.

RKHS embedding has emerged as a powerful tool for nonparametric inferences and has found success in numerous applications. Our work here provides insights into their operating characteristics and leads to improved testing procedures within the framework.

\section{Proofs}
\label{sec:proof}
Throughout this section, we shall write $a_n\lesssim b_n$ if there exists a universal constant $C>0$ such that $a_n\leq Cb_n$. Similarly, we write $a_n\gtrsim b_n$ if $b_n\lesssim a_n$, and $a_n\asymp b_n$ if $a_n\lesssim b_n$ and $a_n\gtrsim b_n$. When the the constant depends on another quantity $D$, we shall write $a_n\lesssim_D b_n$. Relations $\gtrsim_D$ and $\asymp_D$ are defined accordingly.

\begin{proof}[Proof of Theorem \ref{th:gofnull}] 
We begin with \eqref{eq:gofnull1}. Note that $\hat{\gamma_{\nu_n}^2}(\PP,\PP_0)$ is a U-statistic. We can apply the general techniques for U-statistics to establish its asymptotic normality. In particular, as shown in \cite{hall1984central}, it suffices to verify the following four conditions:
\begin{align}
&\left(2\nu_n\over \pi\right)^{d/2}\EE \bar{G}_{\nu_n}^2(X_1,X_2)\to\|p_0\|_{L_2}^2\label{lc0},\\
&{\EE \bar{G}_{\nu_n}^4(X_1,X_2)\over n^2[\EE\bar{G}_{\nu_n}^2(X_1,X_2)]^2}\rightarrow 0\label{lc1},\\
&{\EE [\bar{G}_{\nu_n}^2(X_1,X_2)\bar{G}_{\nu_n}^2(X_1,X_3)]\over n[\EE\bar{G}_{\nu_n}^2(X_1,X_2)]^2}\rightarrow 0\label{lc2},\\
&{\EE H_{\nu_n}^2(X_1,X_2)\over [\EE\bar{G}_{\nu_n}^2(X_1,X_2)]^2}\rightarrow 0\label{lc3},
\end{align}
as $n\rightarrow \infty$, where
$$
H_{\nu_n}(x,y)=\EE\bar{G}_{\nu_n}(x,X_3)\bar{G}_{\nu_n}(y,X_3),\quad\forall\ x,y\in \RR^d.
$$

\paragraph{Verifying Condition \eqref{lc0}.} Note that
\begin{align*}
\EE\bar{G}_{\nu_n}^2(X_1,X_2)=\EE G_{\nu_n}^2(X_1,X_2)-2\EE \{\EE [G_{\nu_n}(X_1,X_2)|X_1]\}^2+[\EE G_{\nu_n}(X_1,X_2)]^2.
\end{align*}
By Lemma \ref{le:gausskernel}, 
\begin{align*}
\EE G_{\nu_n}(X_1,X_2)=\left(\frac{\pi}{\nu_n}\right)^{\frac{d}{2}}\int\exp\left(-\frac{\|\omega\|^2}{4\nu_n}\right)\left\|\Fcal{p_0}(\omega)\right\|^2d \omega,
\end{align*}
which immediately yields
$$
\left(\frac{\nu_n}{\pi}\right)^{\frac{d}{2}}\EE G_{\nu_n}(X_1,X_2)\to \|p_0\|_{L_2}^2
$$
and
$$
\left(\frac{2\nu_n}{\pi}\right)^{\frac{d}{2}}\EE G^2_{\nu_n}(X_1,X_2)=\left(\frac{2\nu_n}{\pi}\right)^{\frac{d}{2}}\EE G_{2\nu_n}(X_1,X_2)\to \|p_0\|_{L_2}^2,
$$
as $\nu_n\to\infty$.

On the other hand,
\begin{align*}
&\EE \{\EE [G_{\nu_n}(X_1,X_2)|X_1]\}^2\\=&\int \left(\int G_{\nu_n}(x,x')G_{\nu_n}(x,x'')p_0(x)d x\right)p_0(x')p_0(x'')d x'd x''\\
=&\int \left(\int G_{2\nu_n}(x,(x'+x'')/2)p_0(x)d x\right)G_{\nu_n/2}(x',x'')p_0(x')p_0(x'')d x'd x''.
\end{align*}
Let $Z\sim N(0,4\nu_nI_d)$. Then
\begin{align*}
\int G_{2\nu_n}(x,(x'+x'')/2)p_0(x)d x&=(2\pi)^{d/2}\EE\left[\Fcal{p}_0(Z)\exp\left(\frac{x'+x''}{2}iZ\right)\right]\\
&\leq (2\pi)^{d/2}\sqrt{\EE\left\|\Fcal{p}_0(Z)\right\|^2}\\
&\lesssim_d \|p_0\|_{L_2}/\nu_n^{d/4}.
\end{align*} 
Thus
$$
\EE \{\EE [G_{\nu_n}(X_1,X_2)|X_1]\}^2\lesssim_d \|p_0\|_{L_2}^3/\nu_n^{3d/4}.
$$
Condition \eqref{lc0} then follows.

\paragraph{Verifying Conditions \eqref{lc1} and \eqref{lc2}.} Since
$$
\EE \bar{G}_{\nu_n}^2(X_1,X_2)\asymp_{d,p_0} \nu_n^{-d/2}.
$$
and
$$
\EE\bar{G}_{\nu_n}^4(X_1,X_2)\lesssim \EE G_{\nu_n}^4(X_1,X_2)\lesssim_d \nu_n^{-d/2},
$$
we obtain
$$
n^{-2}\EE \bar{G}_{\nu_n}^4(X_1,X_2)/(\EE\bar{G}_{\nu_n}^2(X_1,X_2))^2\lesssim_{d,p_0} \nu_n^{d/2}/n^{2}\rightarrow 0.
$$

Similarly,
\begin{align*}
\EE \bar{G}_{\nu_n}^2(X_1,X_2)\bar{G}_{\nu_n}^2(X_1,X_3)&\lesssim \EE G_{\nu_n}^2(X_1,X_2)G_{\nu_n}^2(X_1,X_3)\\
&=\EE G_{2\nu_n}(X_1,X_2)G_{2\nu_n}(X_1,X_3)\\
&\lesssim_{d,p_0} \nu_n^{-3d/4}.
\end{align*}
This implies
$$
n^{-1}\EE\bar{G}_{\nu_n}^2(X_1,X_2)\bar{G}_{\nu_n}^2(X_1,X_3)/(\EE\bar{G}_{\nu_n}^2(X_1,X_2))^2\lesssim_{d,p_0} \nu_n^{d/4}/n\rightarrow 0,
$$
which verifies \eqref{lc2}.

\paragraph{Verifying Condition \eqref{lc3}.} We now prove (\ref{lc3}). It suffices to show
$$
\nu_n^{d}\EE(\EE(\bar{G}_{\nu_n}(X_1,X_2)\bar{G}_{\nu_n}(X_1,X_3)|X_2,X_3))^2\rightarrow 0
$$
as $n\rightarrow \infty$.
Note that
\begin{align*}
&\EE(\EE(\bar{G}_{\nu_n}(X_1,X_2)\bar{G}_{\nu_n}(X_1,X_3)|X_2,X_3))^2\\\lesssim &\EE(\EE(G_{\nu_n}(X_1,X_2)G_{\nu_n}(X_1,X_3)|X_2,X_3))^2\\=& \EE G_{\nu_n}(X_1,X_2)G_{\nu_n}(X_1,X_3)G_{\nu_n}(X_4,X_2)G_{\nu_n}(X_4,X_3)\\
=&\EE(G_{\nu_n}(X_1,X_4)G_{\nu_n}(X_2,X_3)\EE(G_{\nu_n}(X_1+X_4,X_2+X_3)|X_1-X_4,X_2-X_3)).
\end{align*}
Since for any $\delta>0$,
\begin{eqnarray*}
\nu_n^d\EE(G_{\nu_n}(X_1,X_4)G_{\nu_n}(X_2,X_3)\EE(G_{\nu_n}(X_1+X_4,X_2+X_3)|X_1-X_4,X_2-X_3)\\
(\mathds{1}_{\{\|X_1-X_4\|>\delta\}}+\mathds{1}_{\|X_2-X_3\|>\delta\}}))\rightarrow 0,
\end{eqnarray*}
it remains to show that
\begin{eqnarray*}
\nu_n^d\EE(G_{\nu_n}(X_1,X_4)G_{\nu_n}(X_2,X_3)\EE(G_{\nu_n}(X_1+X_4,X_2+X_3)|X_1-X_4,X_2-X_3)\\
\mathds{1}_{\{\|X_1-X_4\|\leq\delta,\|X_2-X_3\|\leq \delta\}}))\rightarrow 0
\end{eqnarray*}
for some $\delta>0$, which holds as long as
\begin{align}
\EE(G_{\nu_n}(X_1+X_4,X_2+X_3)|X_1-X_4,X_2-X_3)\rightarrow 0\label{uc}
\end{align}
uniformly on $\{\|X_1-X_4\|\leq \delta,\|X_2-X_3\|\leq \delta\}$.

Let
$$
Y_1=X_1-X_4,\quad Y_2=X_2-X_3,\quad Y_3=X_1+X_4,\quad Y_4=X_2+X_3.
$$
Then
\begin{align*}
&\EE(G_{\nu_n}(X_1+X_4,X_2+X_3)|X_1-X_4,X_2-X_3)\\= &\left(\frac{\pi}{\nu_n}\right)^{\frac{d}{2}}\int\exp\left(-\frac{\|\omega\|^2}{4\nu_n}\right)\Fcal{p_{Y_1}}(\omega)\overline{\Fcal{p_{Y_2}}}(\omega)d \omega\\
\leq&\sqrt{\left(\frac{\pi}{\nu_n}\right)^{\frac{d}{2}}\int\exp\left(-\frac{\|\omega\|^2}{4\nu_n}\right)\left\|\Fcal{p_{Y_1}}(\omega)\right\|^2d \omega}\sqrt{\left(\frac{\pi}{\nu_n}\right)^{\frac{d}{2}}\int\exp\left(-\frac{\|\omega\|^2}{4\nu_n}\right)\left\|\Fcal{p_{Y_2}}(\omega)\right\|^2d \omega}
\end{align*}
where 
\begin{align*}
p_{y}(y')=\frac{p(Y_1=y,Y_3=y')}{p(Y_1=y)}=\frac{p_0\left(\frac{y+y'}{2}\right)p_0\left(\frac{y'-y}{2}\right)}{\int p_0\left(\frac{y+y'}{2}\right)p_0\left(\frac{y'-y}{2}\right)d y'}
\end{align*}
is the conditional density of $Y_3$ given $Y_1=y$. Thus to prove (\ref{uc}), it suffices to show
\begin{align*}
h_n(y)&:=\left(\frac{\pi}{\nu_n}\right)^{\frac{d}{2}}\int\exp\left(-\frac{\|\omega\|^2}{4\nu_n}\right)\left\|\Fcal{p_y}(\omega)\right\|^2d \omega\\
&=\pi^{\frac{d}{2}}\int\exp\left(-\frac{\|\omega\|^2}{4}\right)\left\|\Fcal{p_y}(\sqrt{\nu_n}\omega)\right\|^2d \omega\\
&\rightarrow 0
\end{align*}
uniformly over $\{y:\ \|y\|\leq \delta\}$.

Note that
$$
h_n(y)=\EE G_{\nu_n}(X,X')
$$
where $X,X'\sim_{\rm iid} p_y$, which suggests
$
h_n(y)\rightarrow 0
$
pointwisely. To prove the uniform convergence of $h_n(y)$, we only need to show
$$
\lim\limits_{y_1\rightarrow y}\sup\limits_{n}|h_n(y_1)-h_n(y)|=0
$$
for any $y$.

Since $p_0\in L_2$, $P(Y_1=y)$ is continuous. Therefore, the almost surely continuity of $p_0$ immediately suggests that for every $y$,
$
p_{y_1}(\cdot)\rightarrow p_y(\cdot)
$
almost surely as $y_1\rightarrow y$. Considering that $p_{y_1}$ and $p_y$ are both densities, it follows that
$$
|\Fcal{p_{y_1}}(\omega)-\Fcal{p_y}(\omega)|\leq (2\pi)^{-d/2}\int |p_{y_1}(y')-p_y(y')|d y'\rightarrow 0,
$$
\ie, $\Fcal{p_{y_1}}\rightarrow \Fcal{p_y}$ uniformly as $y_1\rightarrow y$. Therefore we have
$$
\sup_{n\to\infty}|h_n(y_1)-h_n(y)|\lesssim\left\|\Fcal{p_{y_1}}-\Fcal{p_y}\right\|_{L_\infty}\rightarrow 0,
$$
which ensures the uniform convergence of $h_n(y)$ to $h(y)$ over $\{y:\ \|y\|\leq \delta\}$, and hence (\ref{lc3}).

Indeed, we have shown that
$$
\frac{n\hat{\gamma_{\nu_n}^2}(\PP,\PP_0)}{\sqrt{2\EE[\bGnn(X_1,X_2)]^2}}\to_d N(0,1).
$$
By Slutsky Theorem, in order to prove (\ref{eq:gofnull2}), it sufficies to show
$$
\hat{s}_{n,\nu_n}^2/\EE[\bGnn(X_1,X_2)]^2\to_p1,
$$
which is equivalent to 
\begin{align}\label{consistent-est}
\tilde{s}_{n,\nu_n}^2/\EE[\bGnn(X_1,X_2)]^2\to_p 1
\end{align}
since $1/n^2=o(\EE[\bGnn(X_1,X_2)]^2)$.

It follows from
$$
\EE \left(\tilde{s}_{n,\nu_n}^2\right)=\EE[\bGnn(X_1,X_2)]^2
$$
and
\begin{align*}
&\var\left(\tilde{s}_{n,\nu_n}^2\right)\\\lesssim &n^{-4}\var\left(\sum\limits_{1\leq i\neq j\leq n}G_{2\nu_n}(X_i,X_j)\right)+n^{-6}\var\left(\sum\limits_{\substack{1\le i,j_1,j_2\le n\\ |\{i,j_1,j_2\}|=3}}G_{\nu_n}(X_i,X_{j_1})G_{\nu_n}(X_i,X_{j_2})\right)\\&+n^{-8}\var\left(\sum\limits_{\substack{1\le i_1,i_2,j_1,j_2\le n\\ |\{i_1,i_2,j_1,j_2\}|=4}}G_{\nu_n}(X_{i_1},X_{j_1})G_{\nu_n}(X_{i_2},X_{j_2})\right)\\
\lesssim &n^{-2}\EE G_{4\nu_n}(X_1,X_2)+n^{-1}\EE G_{2\nu_n}(X_1,X_2)G_{2\nu_n}(X_1,X_3)+n^{-1}(\EE G_{2\nu_n}(X_1,X_2))^2\\
=\ &o((\EE[\bGnn(X_1,X_2)]^2)^2).
\end{align*}
that (\ref{consistent-est}) holds.
\end{proof}
\vskip 25pt

\begin{proof}[Proof of Theorem \ref{th:gofpower}] Recall that
\begin{align*}
\hat{\gamma_{\nu_n}^2}(\PP,\PP_0)=&\frac{1}{n(n-1)}\sum_{i\neq j}\bar{G}_{\nu_n}(X_i,X_j;\PP_0)\\
=&\gamma_{\nu_n}^2(\PP,\PP_0)+\frac{1}{n(n-1)}\sum_{i\neq j}\bar{G}_{\nu_n}(X_i,X_j;\PP)\\
&+{2\over n}\sum_{i=1}^n\biggl(\EE_{X\sim \PP} [G_{\nu_n}(X_i,X)|X_i]-\EE_{X\sim \PP_0} [G_{\nu_n}(X_i,X)|X_i]\\
&-\EE_{X,X'\sim_{\rm iid} \PP} G_{\nu_n}(X,X')+\EE_{(X,Y)\sim \PP\otimes\PP_0} G_{\nu_n}(X,Y)\biggr).
%
\end{align*}
Denote by the last two terms on the rightmost hand side by $V_{\nu_n}^{(1)}$ and $V_{\nu_n}^{(2)}$ respectively. It is clear that $\EE V_{\nu_n}^{(1)}=\EE V_{\nu_n}^{(2)}=0$. Then it suffices to show that
\begin{equation}
\sup_{\substack{p\in \Wcal^{s,2}(M)\\ \|p-p_0\|\ge \Delta_n}}\frac{\EE\left( V_{\nu_n}^{(1)}\right)^2+\EE \left(V_{\nu_n}^{(2)}\right)^2}{\gamma_{\nu_n}^4(\PP,\PP_0)}\to 0\label{lc4}
\end{equation}
and
\begin{equation}
\inf_{\substack{p\in \Wcal^{s,2}(M)\\ \|p-p_0\|\ge \Delta_n}}\frac{n\gamma^2_{G_{\nu_n}}(\PP,\PP_0)}{\sqrt{
\EE\left(\hat{s}_{n,\nu_n}^2\right)
}}\to\infty\label{lc5}
\end{equation}
as $n\to\infty$.

We first prove (\ref{lc4}). Note that $\|p\|_{L_2}\le \|p\|_{\Wcal^{s,2}(M)}\le M$. Following arguments similar to those in the proof of Theorem \ref{th:gofnull}, we get
$$
\EE\left( V_{\nu_n}^{(1)}\right)^2\lesssim n^{-2}\EE G_{\nu_n}^2(X_1,X_2)\lesssim_d M^2 n^{-2}\nu_n^{-d/2},
$$
and
\begin{align*}
\EE\left( V_{\nu_n}^{(2)}\right)^2&\leq {4\over n}\EE \left[\EE_{X\sim \PP} [G_{\nu_n}(X_i,X)|X_i]-\EE_{X\sim \PP_0} [G_{\nu_n}(X_i,X)|X_i]\right]^2\\
&={4\over n}\int \left(\int G_{2\nu_n}(x,(x'+x'')/2)p(x)d x\right)G_{\nu_n/2}(x',x'')f(x')f(x'')d x'd x''\\
&\lesssim_d {4M\over n\nu^{d/4}} \int G_{\nu_n/2}(x',x'')|f(x')||f(x'')|d x'd x''\\
&\lesssim_d {4M\over n\nu^{3d/4}}\|f\|_{L_2}^2.
\end{align*}

By Lemma \ref{le:gaussmmd}, there exists a constant $C>0$ depending on $s$ and $M$ only such that for $f\in \Wcal^{s,2}(M)$,
\begin{align*}
\int\exp\left(-\frac{\|\omega\|^2}{4\nu_n}\right)\left\|\Fcal{f}(\omega)\right\|^2 d \omega\geq \frac{1}{4}\|f\|_{L_2}^2
\end{align*}
given that $\nu_n\geq C\|f\|_{L_2}^{-2/s}$. Because $\nu_n\Delta_n^{s/2}\rightarrow \infty$, we obtain
$$
\gamma_{\nu_n}^2(\PP,\PP_0)\gtrsim_d \nu_n^{-d/2}\|f\|_{L_2}^2,
$$
for sufficiently large $n$. Thus
$$
\sup_{\substack{p\in \Wcal^{s,2}(M)\\ \|p-p_0\|\ge \Delta_n}}\frac{\EE\left( V_{\nu_n}^{(1)}\right)^2}{\gamma_{\nu_n}^4(\PP,\PP_0)}\lesssim_d M^2(n^2\nu_n^{-d/2}\Delta_n^4)^{-1}\rightarrow 0
$$
and
$$
\sup_{\substack{p\in \Wcal^{s,2}(M)\\ \|p-p_0\|\ge \Delta_n}}\frac{\EE \left(V_{\nu_n}^{(2)}\right)^2}{\gamma_{G_{\nu_n}}^4(\PP,\PP_0)}\lesssim_d M(n\nu_n^{-d/4}\Delta_n^2)^{-1}\rightarrow 0,
$$
as $n\rightarrow \infty$.

Next we prove (\ref{lc5}). It follows from
$$
\EE\left(\hat{s}_{n,\nu_n}^2\right)\leq \EE\max\left\{\left|\tilde{s}_{n,\nu_n}^2\right|,1/n^2\right\}\lesssim\EE G_{2\nu_n}(X_1,X_2)+1/n^2\lesssim_d M^2\nu_n^{-d/2}+1/n^2
$$
that (\ref{lc5}) holds.
\end{proof}
\vskip 25pt

\begin{proof}[Proof of Theorem \ref{th:goflower}]	
This, in a certain sense, can be viewed as an extension of results from \cite{ingster1987minimax}, and the proof proceeds in a similar fashion. While \cite{ingster1987minimax} considered the case when $p_0$ is the uniform distribution on $[0,1]$, we shall show that similar bounds hold for a wider class of $p_0$.

For any $M>0$ and $p_0$ such that $\|p_0\|_{\Wcal^{s,2}}<M$, let
\begin{align*}
H_1^{\rm GOF}(\Delta_n;s,M-\|p_0\|_{\Wcal^{s,2}})^*&\\
:=\{p\in\Wcal^{s,2}:&\ \|p-p_0\|_{\Wcal^{s,2}}\leq M-\|p_0\|_{\Wcal^{s,2}},\ \|p-p_0\|_{L_2}\geq \Delta_n\}.
\end{align*}
It is clear that $H_1^{\rm GOF}(\Delta_n;s)\supset H_1^{\rm GOF}(\Delta_n;s,M-\|p_0\|_{\Wcal^{s,2}})^*$. Hence it suffices to prove Theorem \ref{th:goflower} with $H_1^{\rm GOF}(\Delta_n;s)$ replaced by $H_1^{\rm GOF}(\Delta_n;s,M)^*$ for an arbitrary $M>0$. We shall  abbreviate $H_1^{\rm GOF}(\Delta_n;s,M)^*$ as $H_1^{\rm GOF}(\Delta_n;s)^*$ in the rest of the proof.

Since $p_0$ is almost surely continuous, there exists $x_0\in\RR^d$ and $\delta,c>0$ such that
$$
p_0(x)\geq c>0,\quad\forall\ \|x-x_0\|\leq \delta.
$$
In light of this, we shall assume $p_0(x)\geq c>0$, for all $x\in[0,1]^d$ without loss of generality.

Let $\bm{a}_n$ be a multivariate random index. As proved in \cite{ingster1987minimax}, in order to prove the existence of $\alpha\in(0,1)$ such that no asymptotic $\alpha$-level test can be consistent, it suffices to identify $p_{n,\bm{a}_n} \in H_1^{\rm GOF}(\Delta_n;s)^*$ for all possible values of $\bm{a}_n$ 
such that \begin{align}\label{l2b}
\EE_{p_0}\left(\frac{p_n(X_1,\cdots,X_n)}{\prod_{i=1}^n p_0(X_i)}\right)^2=O(1),
\end{align}
where
$$
p_n(x_1,\cdots,x_n)=\EE_{\bm{a}_n}\left(\prod\limits_{i=1}^np_{n,\bm{a}_n}(x_i)\right),\ \forall\ x_1,\cdots,x_n,
$$
\ie, $p$ is the mixture of all $p_{n,\bm{a}_n}$'s.

Let $\mathds{1}_{\{x\in [0,1]^d\}},\phi_{n,1},\cdots,\phi_{n,B_n}$ be an orthonormal sets of functions in $L^2(\Rd)$ such that the supports of $\phi_{n,1},\cdots,\phi_{n,B_n}$ are disjoint and all included in $[0,1]^d$. Let $\bm{a}_n=(a_{n,1},\cdots,a_{n,B_n})$ satisfy that $a_{n,1},\cdots,a_{n,B_n}$ are independent and that
$$
p(a_{n,k}=1)=p(a_{n,k}=-1)=\frac{1}{2},\quad \forall\ 1\leq k\leq B_n.
$$
Define
$$
p_{n,\bm{a}_n}=p_0+r_n\sum\limits_{k=1}^{B_n}a_{n,k}\phi_{n,k}.
$$
Then
$$
\frac{p_{n,\bm{a}_n}}{p_0}=1+r_n\sum\limits_{k=1}^{B_n}a_{n,k}\frac{\phi_{n,k}}{p_0},
$$
where $1,\frac{\phi_{n,1}}{p_0},\cdots,\frac{\phi_{n,B_n}}{p_0}$ are orthogonal in $L_2(P_0)$.

By arguments similar to those in \cite{ingster1987minimax}, we find
\begin{align*}
\EE_{p_0}\left(\frac{p_n(X_1,\cdots,X_n)}{\prod_{i=1}^n p_0(X_i)}\right)^2&\leq \exp\left(\frac{1}{2}B_nn^2r_n^4\max_{1\leq k\leq B_n}\left(\int \phi_{n,k}^2/p_0d x\right)^2\right)\\
&\leq \exp\left(\frac{1}{2c^2}B_nn^2r_n^4\right).
\end{align*}
In order to ensure (\ref{l2b}), it suffices to have
\begin{equation}\label{condition}
B_n^{1/2}nr_n^2=O(1).
\end{equation}
Therefore, given $\Delta_n=O\left(n^{-\frac{2s}{4s+d}}\right)$, once we can find proper $r_n$, $B_n$ and $\phi_{n,1},\cdots,\phi_{n,B_n}$ such that $p_{n,\bm{a}_n}\in H_1^{\rm GOF}(\Delta_n;s)^*$ for all $\bm{a}_n$ and (\ref{condition}) holds, the proof is finished.

Let $b_n=B_n^{1/d}$, $\phi$ be an infinitely differentiable function supported on $[0,1]^d$ that is orthogonal to $\mathds{1}_{\{x\in [0,1]^d\}}$ in $L_2$, and for each $x_{n,k}\in\{0,1,\cdots,b_n-1\}^{\otimes d}$, let
$$
\phi_{n,k}(x)=\frac{b_n^{d/2}}{\|\phi\|_{L_2}}\phi(b_nx-x_{n,k}),\quad \forall\ x\in\mathbb{R}^d.
$$
Then all $\phi_{n,k}$'s are supported on $[0,1]^d$ and
\begin{align*}
&\langle \phi_{n,k},1\rangle_{L_2}=\frac{b_n^{d/2}}{\|\phi\|_{L_2}}\int_{\mathbb{R}^d}\phi(b_nx-x_{n,k})dx=\frac{1}{b_n^{d/2}\|\phi\|_{L_2}}\int_{\mathbb{R}^d}\phi(x)dx=0,\\
&\|\phi_{n,k}\|_{L_2}^2=\frac{b_n^d}{\|\phi\|_{L_2}^2}\int_{[0,1/b_n]^d}\phi^2(b_nx)d x=1,\\
&\|\phi_{n,k}\|_{\Wcal^{s,2}}^2\leq b_n^{2s}\frac{\|\phi\|_{\Wcal^{s,2}}^2}{\|\phi\|_{L_2}^2}.
\end{align*}

Since for $k\neq k'$, the supports of $\phi_{n,k}$ and $\phi_{n,k'}$ are disjoint,
$$
\|p_{n,\bm{a}_n}-p_0\|_{\infty}=r_nb_n^{d/2}\frac{\|\phi\|_{\infty}}{\|\phi\|_{L_2}},
$$
and
$$
\langle \phi_{n,k},\phi_{n,k'}\rangle_{L_2}=0,\qquad \langle \phi_{n,k},\phi_{n,k'}\rangle_{\Wcal^{s,2}}=0,
$$
from which we immediately obtain 
\begin{align*}
&\|p_{n,\bm{a}_n}-p_0\|_{L_2}^2=r_n^2b_n^d\\
&\|p_{n,\bm{a}_n}-p_0\|_{\Wcal^{s,2}}^2\leq r_n^2b_n^{d+2s}\frac{\|\phi\|_{\Wcal^{s,2}}^2}{\|\phi\|_{L_2}^2}.
\end{align*}
To ensure $p_{n,\bm{a}_n}\in H_1^{\rm GOF}(\Delta_n;s)^*$, it suffices to make
\begin{align}
&r_nb_n^{d/2}\frac{\|\phi\|_{\infty}}{\|\phi\|_{L_2}}\rightarrow 0\ \text{as}\ n\rightarrow \infty,\label{condition2}\\& r_n^2b_n^d=\Delta_n^2,\label{condition3}\\ &r_n^2b_n^{d+2s}\frac{\|\phi\|_{\Wcal^{s,2}}^2}{\|\phi\|_{L_2}^2}\leq M^2.\label{condition4}
\end{align}

Let
$$
b_n=\left\lfloor\left(\frac{M\|\phi\|_{L_2}^2}{\|\phi\|_{\Wcal^{s,2}}}\right)^{1/s}\Delta_n^{-1/s}\right\rfloor,\quad r_n=\frac{\Delta_n}{b_n^{d/2}}.
$$
Then (\ref{condition3}) and (\ref{condition4}) are satisfied. Moreover, given $\Delta_n=O\left(n^{-\frac{2s}{4s+d}}\right)$,
$$
B_n^{1/2}nr_n^2=b_n^{-d/2}n\Delta_n^{2}\lesssim_{d,\phi,M}n\Delta_n^{\frac{4s+d}{2s}}=O(1),
$$
and
$$
r_nb_n^{d/2}\frac{\|\phi\|_{\infty}}{\|\phi\|_{L_2}}\lesssim_{\phi}\Delta_n=o(1)
$$
ensuring both (\ref{condition}) and (\ref{condition2}).

Finally, we show the existence of such $\phi$. Let
$$
\phi_0(x_1)=\begin{cases}
\exp\left(-\frac{1}{1-(4x_1-1)^2}\right) &0<x_1<\frac{1}{2}\\
-\exp\left(-\frac{1}{1-(4x_1-3)^2}\right) &\frac{1}{2}<x_1<1\\
0&\text{otherwise}
\end{cases}.
$$
Then $\phi_0$ is supported on $[0,1]$, infinitely differentiable and orthogonal to the indicator function of $[0,1]$.

Let 
$$
\phi(x)=\prod\limits_{l=1}^{d}\phi_0(x_l),\quad \forall\ x=(x_1,\cdots,x_d)\in\mathbb{R}^d.
$$
Then $\phi$ is supported on $[0,1]^d$, infinitely differentiable and 
$
\langle \phi, 1\rangle_{L_2}=\langle \phi_0,1\rangle_{L_2[0,1]}^d=0.
$	
\end{proof}
\vskip 25pt

\begin{proof}[Proof of Theorem \ref{th:homnull}]
Let $N=m+n$ denote the total sample size. It suffices to prove the result under the assumption that $n/N\rightarrow r\in(0,1)$. 

Note that under $H_0$,
\begin{align*}
\hat{\gamma_{\nu_n}^2}(\PP,\QQ)=&{1\over n(n-1)}\sum_{1\leq i\neq j\leq n} \bar{G}_{\nu_n}(X_i,X_j)+{1\over m(m-1)}\sum_{1\leq i\neq j\leq m} \bar{G}_{\nu_n}(Y_i,Y_j)\\&-{2\over nm}\sum_{1\leq i\leq n}\sum\limits_{1\leq j\leq m} \bar{G}_{\nu_n}(X_i,Y_j).
\end{align*}
Let $n/N=r_n$. Then we have
\begin{align*}
&\hat{\gamma_{\nu_n}^2}(\PP,\QQ)\\=&N^{-2}\left({1\over r_n(r_n-N^{-1})}\sum_{1\leq i\neq j\leq n} \bar{G}_{\nu_n}(X_i,X_j)\right.+\\&\left.{1\over (1-r_n)(1-r_n-N^{-1})}\sum_{1\leq i\neq j\leq m} \bar{G}_{\nu_n}(Y_i,Y_j)-{2\over r_n(1-r_n)}\sum\limits_{1\leq i\leq n}\sum\limits_{1\leq j\leq m} \bar{G}_{\nu_n}(X_i,Y_j)\right).
\end{align*}

Let 
\begin{align*}
\hat{\gamma_{\nu_n}^2}(\PP,\QQ)'=&N^{-2}\left({1\over r^2}\sum_{1\leq i\neq j\leq n} \bar{G}_{\nu_n}(X_i,X_j)+{1\over (1-r)^2}\sum_{1\leq i\neq j\leq m} \bar{G}_{\nu_n}(Y_i,Y_j)\right.\\
&\left.-{2\over r(1-r)}\sum_{1\leq i\leq n}\sum\limits_{1\leq j\leq m} \bar{G}_{\nu_n}(X_i,Y_j)\right).
\end{align*}
As we assume $r_n\rightarrow r$ as $n\rightarrow \infty$, Theorem \ref{th:gofnull} ensures that 
$$
\frac{nm}{\sqrt{2}(n+m)}\left[\EE \bar{G}_{\nu_n}^2(X_1,X_2)\right]^{-\frac{1}{2}}\left(\hat{\gamma_{\nu_n}^2}(\PP,\QQ)-\hat{\gamma_{\nu_n}^2}(\PP,\QQ)'\right)=o_p(1)
$$

A slight adaption of arguments in \cite{hall1984central} suggests that
\begin{align}\label{lc6}
\frac{\EE \bar{G}_{\nu_n}^4(X_1,X_2)}{N^2\EE\bar{G}_{\nu_n}^2(X_1,X_2)}+\frac{\EE \bar{G}_{\nu_n}^2(X_1,X_2)\bar{G}_{\nu_n}^2(X_1,X_3)}{N\EE\bar{G}_{\nu_n}^2(X_1,X_2)}+\frac{\EE H_{\nu_n}^2(X_1,X_2)}{\EE\bar{G}_{\nu_n}^2(X_1,X_2)}\to 0
\end{align}
ensures that
$$
\frac{nm}{\sqrt{2}(n+m)}\left[\EE \bar{G}_{\nu_n}^2(X_1,X_2)\right]^{-\frac{1}{2}}\hat{\gamma_{\nu_n}^2}(\PP,\QQ)'\to_d N(0,1).
$$
Following arguments similar to those in the proof of Theorem \ref{th:gofnull}, given $\nu_n\rightarrow \infty$ and $\nu_n/n^{4/d}\rightarrow 0$, (\ref{lc6}) holds and therefore
$$
\frac{nm}{\sqrt{2}(n+m)}\left[\EE \bar{G}_{\nu_n}^2(X_1,X_2)\right]^{-\frac{1}{2}}\hat{\gamma_{\nu_n}^2}(\PP,\QQ)\to_d N(0,1).
$$

Additionally, based on the same arguments as in the proof of Theorem \ref{th:gofnull},
$$
\hat{s}_{n,m,\nu_n}^2/\EE[\bGnn(X_1,X_2)]^2\to_p 1.
$$
The proof is therefore concluded.
\end{proof}
\vskip 25pt

\begin{proof}[Proof of Theorem \ref{th:hompower}] With slight abuse of notation, we shall write
$$
\bar{G}_{\nu_n}(x,y;\PP,\QQ)=G_{\nu_n}(x,y)-\EE_{Y\sim\QQ} G_{\nu_n}(x,Y)-\EE_{X\sim \PP} G_{\nu_n}(X,y)+\EE_{(X,Y)\sim \PP\otimes\QQ} G_{\nu_n}(X,Y),
$$
We consider the two parts separately.

\paragraph{Part (\ro).} We first verify the consistency of $\Phi_{n,\nu_n,\alpha}^{\mathrm{HOM}}$ with $\nu_n\asymp n^{4/(d+4s)}$ given $\Delta_n\gg n^{-2s/(d+4s)}$.


Observe the following decomposition of $\hat{\gamma_{\nu_n}^2}(\PP,\QQ)$,
$$
\hat{\gamma_{\nu_n}^2}(\PP,\QQ)=\gamma_{\nu_n}^2(\PP,\QQ)+L_{n,\nu_n}^{(1)}+L_{n,\nu_n}^{(2)},
$$
where
\begin{align*}
L_{n,\nu_n}^{(1)}\notag=&\frac{1}{n(n-1)}\sum\limits_{1\leq i\neq j\leq n}\bar{G}_{\nu_n}(X_i,X_j;\PP)-\frac{2}{mn}\sum\limits_{1\leq i\leq n}\sum\limits_{1\leq j\leq m}\bar{G}_{\nu_n}(X_i,Y_j;\PP,\QQ)\\
&+\frac{1}{m(m-1)}\sum\limits_{1\leq i\neq j\leq m}\bar{G}_{\nu_n}(Y_i,Y_j;\QQ)
\end{align*}
and
\begin{align*}
L_{n,\nu_n}^{(2)}=&\frac{2}{n}\sum\limits_{i=1}^n\left(\EE[G_{\nu_n}(X_i,X)|X_i]-\EE G_{\nu_n}(X,X')-\EE[G_{\nu_n}(X_i,Y)|X_i]+\EE G_{\nu_n}(X,Y)\right)\\
&+\frac{2}{m}\sum\limits_{j=1}^m\left(\EE [G_{\nu_n}(Y_j,Y)|Y_j]-\EE G_{\nu_n}(Y,Y')-\EE [G_{\nu_n}(X,Y_j)|Y_j]+\EE G_{\nu_n}(X,Y)\right).\notag
\end{align*}
In order to prove the consistency of $\Phi_{n,\nu_n,\alpha}^{\mathrm{HOM}}$, it suffices to show 
\begin{align}
&\sup\limits_{\substack{p,q\in \Wcal^{s,2}(M)\\ \|p-q\|_{L_2}\geq \Delta_n}}\frac{\EE\left(L_{n,\nu_n}^{(1)}\right)^2+\EE\left(L_{n,\nu_n}^{(2)}\right)^2}{\gamma_{G_{\nu_n}}^4(\PP,\QQ)}\rightarrow 0,\label{lc7}\\
&\inf\limits_{\substack{p,q\in \Wcal^{s,2}(M)\\ \|p-q\|_{L_2}\geq \Delta_n}}\frac{\gamma_{G_{\nu_n}}^2(\PP,\QQ)}{\left(1/n+1/m\right)\sqrt{\EE\left(\hat{s}_{n,m,\nu_n}^2\right)}}\rightarrow \infty,\label{lc8}
\end{align}
as $n\rightarrow \infty$. We now prove (\ref{lc7}) and (\ref{lc8}) with arguments similar to those obtained in the proof of Theorem \ref{th:gofpower}. 

Note that
\begin{align*}
\EE(L_{n,\nu_n}^{(1)})^2\lesssim&\EE\left(\frac{1}{n(n-1)}\sum\limits_{1\leq i\neq j\leq n}\bar{G}_{\nu_n}(X_i,X_j;\PP)\right)^2+\EE\left(\frac{2}{mn}\sum_{1\leq i\leq n}\sum_{1\leq j\leq m}\bar{G}_{\nu_n}(X_i,Y_j;\PP,\QQ)\right)^2\\&+\EE\left(\frac{1}{m(m-1)}\sum\limits_{1\leq i\neq j\leq m}\bar{G}_{\nu_n}(Y_i,Y_j;\QQ)\right)^2\\
\lesssim &\frac{1}{n^2}\EE G_{\nu_n}^2(X_1,X_2)+\frac{1}{m^2}\EE G_{\nu_n}^2(Y_1,Y_2).
\end{align*}
Given $p,q\in \Wcal^{s,2}(M)$,
\begin{align*}
\EE G_{\nu_n}^2(X_1,X_2)
\lesssim_d M^2\nu_n^{-d/2},\quad \EE G_{\nu_n}^2(Y_1,Y_2)\lesssim_d M^2\nu_n^{-d/2}.
\end{align*}
Hence
\begin{align}\label{var1}
\EE(L_{n,\nu_n}^{(1)})^2\lesssim_d M^2\nu_n^{-d/2}\left(\frac{1}{n^2}+\frac{1}{m^2}\right).
\end{align}

Now consider bounding $L_{n,\nu_n}^{(2)}$. Let $f=p-q$. Then we have
\begin{align}\label{var2}
\EE(L_{n,\nu_n}^{(2)})^2\lesssim_d \nu_n^{-\frac{3d}{4}}M\|f\|_{L_2}^2\left(\frac{1}{n}+\frac{1}{m}\right).
\end{align}
Since $\nu_n\asymp n^{4/(4s+d)}\gg \Delta_n^{-2/s}$, Lemma \ref{le:gaussmmd} ensures that for sufficiently large $n$,
$$
\gamma_{G_{\nu_n}}^2(\PP,\QQ)\gtrsim_d \nu_n^{-d/2}\|f\|_{L_2}^2,\quad\forall\ p,q\in \Wcal^{s,2}(M).
$$
This together with (\ref{var1}) and (\ref{var2}) gives
$$
\sup\limits_{\substack{p,q\in \Wcal^{s,2}(M)\\ \|p-q\|_{L_2}\geq \Delta_n}}\frac{\EE\left(L_{n,\nu_n}^{(1)}\right)^2+\EE\left(L_{n,\nu_n}^{(2)}\right)^2}{\gamma_{G_{\nu_n}}^4(\PP,\QQ)}\lesssim_d \frac{M^2\nu_n^{d/2}}{n^2\Delta_n^4}+\frac{M\nu_n^{d/4}}{n\Delta_n^2}\rightarrow 0
$$
as $n\rightarrow \infty$, which proves (\ref{lc7}).

Finally, consider (\ref{lc8}). It follows from
\begin{align*}
\EE\left(\hat{s}_{n,m,\nu_n}^2\right)\leq\ & \EE\max\left\{\left|\tilde{s}_{n,m,\nu_n}^2\right|,1/n^2\right\}\\
\lesssim\ &\max\{\EE G_{\nu_n}^2(X_1,X_2),\EE G_{\nu_n}^2(Y_1,Y_2)\}+1/n^2\\
\lesssim_d &M^2\nu_n^{-d/2}+1/n^2
\end{align*}
that (\ref{lc8}) holds. 
\paragraph{Part (\rt).} Next, we prove that if $\liminf_{n\to\infty}\Delta_nn^{2s/(d+4s)}<\infty$, then there exists some $\alpha\in(0,1)$ such that no asymptotic $\alpha$-level test can be consistent. To prove this, we shall verify that consistency of homogeneity test is harder to achieve than that of goodness-of-fit test. 

Consider an arbitrary $p_0\in \Wcal^{s,2}(M/2)$. It immediately follows 
$$
H_1^{\rm HOM}(\Delta_n; s)\supset \{(p,p_0):\ p\in H_1^{\rm GOF}(\Delta_n;s)\}.
$$
Let $\{\Phi_{n}\}_{n\geq 1}$ be any sequence of asymptotic $\alpha$-level homogeneity tests, where 
$$
\Phi_{n}=\Phi_{n}(X_1,\cdots,X_n,Y_1,\cdots,Y_m).
$$ 
Then if $Y_1,\cdots,Y_m\sim_{\rm iid} P_0$, $\{\Phi_{n}\}_{n\geq 1}$ can also be treated as a sequence of (random) goodness-of-fit tests
$$
\Phi_{n}(X_1,\cdots,X_n,Y_1,\cdots,Y_m)=\tilde{\Phi}_n(X_1,\cdots,X_n)
$$
whose probabilities of type \Ro\ error with respect to $P_0$ are controlled at $\alpha$ asymptotically. Moreover,
$$
{\rm power}\{\Phi_n; H_1^{\rm HOM}(\Delta_n; s)\}\leq {\rm power}\{\tilde{\Phi}_n; H_1^{\rm GOF}(\Delta_n; s)\}
$$

Since $0<c\leq m/n\leq C<\infty$,
Theorem \ref{th:goflower} ensures that there exists some $\alpha\in(0,1)$ such that for any sequence of asymptotic $\alpha$-level tests $\{\Phi_n\}_{n\geq 1}$,
$$
\liminf_{n\to\infty}{\rm power}\{\Phi_n; H_1^{\rm HOM}(\Delta_n; s)\}\leq \liminf_{n\to\infty}{\rm power}\{\tilde{\Phi}_n; H_1^{\rm GOF}(\Delta_n; s)\}<1
$$
given $\liminf_{n\to\infty}\Delta_nn^{2s/(d+4s)}<\infty$.		
\end{proof}
\vskip 25pt


\begin{proof}[Proof of Theorem \ref{th:indnull}]
For brevity, we shall focus on the case when $k=2$ in the rest of the proof. Our argument, however, can be straightforwardly extended to the more general cases. The proof relies on the following decomposition of $\hat{\gamma_{\nu_n}^2}(\PP,\PP^{X^1}\otimes\PP^{X^2})$ under $H_0^{\rm IND}$:
\begin{align*}
\hat{\gamma^2_{\nu_n}}(\PP,\PP^{X^1}\otimes\PP^{X^2})=\frac{1}{n(n-1)}\sum_{1\leq i\neq j\leq n}G_{\nu_n}^*(X_i,X_j)+R_n,
\end{align*}
where
\begin{align*}
G_{\nu_n}^*(x,y)=\bar{G}_{\nu_n}(x,y)-\sum\limits_{\substack{1\leq j\leq 2}}g_j(x^j,y)-\sum\limits_{\substack{1\leq j\leq 2}}g_j(y^j,x)+\sum\limits_{\substack{1\leq j_1,j_2\leq 2}}g_{j_1,j_2}(x^{j_1},y^{j_2})
\end{align*}
and the remainder $R_n$ satisfies
$$\EE(R_n)^2\lesssim \EE G_{2\nu}(X_1,X_2)/n^3\lesssim_d \|p\|_{L_2}^2\nu_n^{-d/2}/n^3.$$ See Appendix \ref{sec:HSIC_decomp} for more details.


Moreover, borrowing arguments in the proof of Lemma \ref{le:var}, we obtain
\begin{align*}
&\EE(G_{\nu_n}^*(X_1,X_2)-\bar{G}_{\nu_n}(X_1,X_2))^2\\
\lesssim &\sum\limits_{1\leq j\leq 2}\EE\Big(g_j(X_1^j,X_2)\Big)^2+\sum\limits_{\substack{1\leq j_1,j_2\leq 2}}\EE \Big(g_{j_1,j_2}(X_1^{j_1},X_2^{j_2})\Big)^2\\
\leq &\sum\limits_{1\leq j_1\neq j_2\leq 2}\EE G_{2\nu_n}(X_1^{j_1},X_2^{j_1})\cdot\EE\left\{\EE\left[ G_{\nu_n}(X_1^{j_2},X_2^{j_2})\Big|X_1^{j_2}\right]\right\}^2+\\
&\sum\limits_{1\leq j_1\neq j_2\leq 2}\EE G_{2\nu_n}(X_1^{j_1},X_2^{j_1})[\EE G_{\nu_n}(X_1^{j_2},X_2^{j_2})]^2+\\
&\ 2\EE\left\{\EE\left[ G_{\nu_n}(X_1^{1},X_2^{1})\Big|X_1^{1}\right]\right\}^2\EE\left\{\EE\left[ G_{\nu_n}(X_1^{2},X_2^2)\Big|X_1^{2}\right]\right\}^2\\
\lesssim_d &\ \nu_n^{-d_1/2-3d_2/4}\|p_1\|_{L_2}^2\|p_2\|_{L_2}^3+\nu_n^{-3d_1/4-d_2/2}\|p_1\|_{L_2}^3\|p_2\|_{L_2}^2
\end{align*}
Together with the fact that
$$
(2\nu_n/\pi)^{d/2}\EE\bar{G}_{\nu_n}^2(X_1,X_2)\to \|p\|_{L_2}^2
$$
as $\nu_n\to \infty$, we conclude that
$$
\hat{\gamma_{\nu_n}^2}(\PP,\PP^{X^1}\otimes\PP^{X^2})=D(\nu_n)+o_p\left(\sqrt{\EE D^2(\nu_n)}\right),
$$
where
$$
D(\nu_n)=\frac{1}{n(n-1)}\sum\limits_{1\leq i\neq j\leq n}\bar{G}_{\nu_n}(X_i,X_j).
$$

Applying arguments similar to those in the proofs of Theorem \ref{th:gofnull} and \ref{th:homnull}, we have
$$
\frac{D(\nu_n)}{\sqrt{\EE D^2(\nu_n)}}\to_d N(0,1).
$$

Since 
$$
\EE D^2(\nu_n)=\frac{2}{n(n-1)}\EE [\bGnn(X_1,X_2)]^2\quad \text{and}\quad
\EE[\bGnn(X_1,X_2)]^2/\EE[G_{\nu_n}^*(X_1,X_2)]^2\to 1,
$$
it remains to prove
$$
\hat{s}_{n,\nu_n}^2/\EE[G_{\nu_n}^*(X_1,X_2)]^2\to_p 1,
$$
which immediately follows by observing
$$
\tilde{s}_{n,\nu_n}^2/\EE[G_{\nu_n}^*(X_1,X_2)]^2=\prod\limits_{j=1}^2\tilde{s}_{n,j,\nu_n}^2/\EE[\bGnn(X_1^j,X_2^j)]^2\to_p 1
$$
and $1/n^2=o(\EE[G_{\nu_n}^*(X_1,X_2)]^2)$.
%
The proof is therefore concluded.
\end{proof}
\vskip 25pt

\begin{proof}[Proof of Theorem \ref{th:indpower}] We prove the two parts separately.
\paragraph{Part (\ro).} The proof of consistency of $\Phi^{\rm IND}_{n,\nu_n,\alpha}$ is very similar to its counterpart in the proof of Theorem \ref{th:hompower}. It sufficies to show
\begin{align}
&\sup\limits_{p\in H_1^{\rm IND}(\Delta_n,s)}\frac{\var(\hat{\gamma_{\nu_n}^2}(\PP,\PP^{X^1}\otimes \PP^{X^2}))}{\gamma_{\nu_n}^4(\PP,\PP^{X^1}\otimes \PP^{X^2})}\rightarrow 0,\label{lc12}\\
&\inf\limits_{p\in H_1^{\rm IND}(\Delta_n,s)}\frac{n\gamma_{\nu_n}^2(\PP,\PP^{X^1}\otimes \PP^{X^2})}{
    \EE\left(\hat{s}_{n,\nu_n}\right)}\rightarrow \infty,\label{lc13}
\end{align}
as $n\rightarrow \infty$.

We begin with (\ref{lc12}). 
Let $f=p-p_1\otimes p_2$. Lemma \ref{le:gaussmmd} then implies that there exists $C=C(s,M)>0$ such that
$$
\gamma_\nu^2(\PP,\PP^{X^1}\otimes\PP^{X^2})\asymp_d \nu^{-d/2}\|f\|_{L_2}^2
$$
for $\nu\geq C\|f\|_{L_2}^{-2/s}$, which is satisfied by all $p\in H_1^{\rm IND}(\Delta_n,s)$ given $\nu=\nu_n$ and $\lim\limits_{n\rightarrow \infty}\Delta_nn^{2s\over 4s+d}=\infty$. On the other hand, we can still do the decomposition of $\hat{\gamma_{\nu_n}^2}(\PP,\PP^{X^1}\otimes \PP^{X^2})$ as in Appendix \ref{sec:HSIC_decomp}. We follow the same notations here.

Under the alternative hypothesis, the ``first order'' term
\begin{align*}
&D_1(\nu_n)\\=&\frac{2}{n}\sum\limits_{1\leq i\leq n}\Big(\EE_{X_i,X\sim_{\rm  iid}\PP}[G_{\nu_n}(X_i,X)|X_i]-\EE_{X,X'\sim_{\rm iid}\PP} G_{\nu_n}(X,X')\Big) \\&-\frac{2}{n}\sum\limits_{1\leq i\leq n}\Big(\EE_{X_i\sim \PP,Y\sim \PP^{X^1}\otimes \PP^{X^2}}[G_{\nu_n}(X_i,Y)|X_i]-\EE_{X\sim\PP,Y\sim \PP^{X^1}\otimes \PP^{X^2}} G_{\nu_n}(X,Y)\Big)\\
&-\sum\limits_{1\leq j\leq 2}\left(\frac{2}{n}\sum\limits_{1\leq i\leq n}\left(\EE_{X_i\sim \PP^{X^1}\otimes\PP^{X^2},X\sim \PP} [G_{\nu_n}(X_i,X)|X_i^j]-\EE_{X\sim \PP,Y\sim \PP^{X^1}\otimes\PP^{X^2}} G_{\nu_n}(X,Y)\right)\right)\\&+\sum\limits_{1\leq j\leq 2}\left(\frac{2}{n}\sum\limits_{1\leq i\leq n}\left(\EE_{X_i,Y\sim_{\rm iid} \PP^{X^1}\otimes\PP^{X^2}} [G_{\nu_n}(X_i,Y)|X_i^j]-\EE_{Y,Y'\sim_{\rm iid} \PP^{X^1}\otimes\PP^{X^2}} G_{\nu_n}(Y,Y')\right)\right)
\end{align*}
no longer vanish, but based on arguments similar to those in the proof of Theorem \ref{th:gofpower},
$$
\EE D_1^2(\nu_n)\lesssim_d Mn^{-1}\nu_n^{-3d/4}\|f\|_{L_2}^2.
$$
Moreover, the ``second order'' term $D_2(\nu_n)$ is not solely $\sum\limits_{1\leq i\neq j\leq n}G_{\nu_n}^*(X_i,X_j)/(n(n-1))$, but
we still have
$$
\EE D_2^2(\nu_n)\lesssim n^{-2}\max\{\EE G_{2\nu_n}(X_1,X_2),\EE G_{2\nu_n}(X_1^1,X_2^1)\EE G_{2\nu_n}(X_1^2,X_2^2)\}\lesssim_d M^2n^{-2}\nu_n^{-d/2}. 
$$
Similarly, define the third order term $D_3(\nu_n)$ and the fourth order term $D_4(\nu_n)$ as the aggregation of all $3$-variate centered components
and the aggregation of all $4$-variate 
centered components in $\hat{\gamma_{\nu_n}^2}(\PP,\PP^{X^1}\otimes \PP^{X^2})$ respectively, which together constitue $R_n$. Then we have
$$
\EE D_3^2(\nu_n)\lesssim_d M^2n^{-3}\nu_n^{-d/2},\quad \EE D_4^2(\nu_n)\lesssim_d M^2n^{-4}\nu_n^{-d/2}.
$$

Hence we finally obtain
$$
\hat{\gamma_{\nu_n}^2}(\PP,\PP^{X^1}\otimes \PP^{X^2})=\gamma_{\nu_n}^2(\PP,\PP^{X^1}\otimes \PP^{X^2})+\sum\limits_{l=1}^4D_l(\nu_n)
$$
and
$$
\var\Big(\hat{\gamma_{\nu_n}^2}(\PP,\PP^{X^1}\otimes \PP^{X^2})\Big)=\sum\limits_{l=1}^4\EE D_l^2(\nu_n)\lesssim _d Mn^{-1}\nu_n^{-3d/4}\|f\|_{L_2}^2+M^2n^{-2}\nu_n^{-d/2}
$$
which proves (\ref{lc12}).

Now consider (\ref{lc13}). Since
$$
\hat{s}_{n,\nu_n}\leq \max\left\{\prod\limits_{j=1}^2\sqrt{\left|\tilde{s}_{n,j,\nu_n}^2\right|},1/n\right\},
$$
we have
$$
\EE\left(\hat{s}_{n,\nu_n}\right)\leq \prod\limits_{j=1}^2\sqrt{\EE\left|\tilde{s}_{n,j,\nu_n}^2\right|}+1/n,
$$
where
$$
\prod\limits_{j=1}^2\EE\left|\tilde{s}_{n,j,\nu_n}^2\right|\lesssim \prod\limits_{j=1}^2\EE G_{2\nu_n}(X_1^j,X_2^j)=\EE_{Y_1,Y_2\sim_{\rm iid} \PP^{X^1}\otimes\PP^{X^2}} G_{2\nu_n}(Y_1,Y_2)\lesssim_d M^2\nu_n^{-d/2}.
$$
Therefore (\ref{lc13}) holds.

\paragraph{Part (\rt).} Then we verify that $n^{2s/(d+4s)}\Delta_n\to \infty$ is also the necessary condition for the existence of consistent asymptotic $\alpha$-level tests for any $\alpha\in(0,1)$. Similarly to the proof of Theorem \ref{th:hompower}, the idea is to relate the existence of consistent independence test to the existence of consistent goodness-of-fit test.


Let $p_{j,0}\in \Wcal^{s,2}\left(M_j/\sqrt{2}\right)$ be density on $\RR^{d_j}$ for $j=1,2$ and $p_0$ be the product of $p_{1,0}$ and $p_{2,0}$, \ie,
$$
p_0(x^1,x^2)=p_{1,0}(x^1)p_{2,0}(x^2),\quad\forall\ x^1\in\RR^{d_1},x^2\in\RR^{d_2}.
$$
Hence $p_0\in \Wcal^{s,2}(M/2)$.

Let 
$$
H_1^{\rm GOF}(\Delta_n;s)':=\{p:\ p\in \Wcal^{s,2}(M), \ p_1=p_{1,0},\ p_2=p_{2,0}, \|p-p_0\|_{L_2}\geq \Delta_n\}.
$$
We immediately have
$$
H_1^{\rm IND}(\Delta_n; s)\supset H_1^{\rm GOF}(\Delta_n;s)'
$$

Let $\{\Phi_{n}\}_{n\geq 1}$ be any sequence of asymptotic $\alpha$-level independence tests, where 
$$
\Phi_{n}=\Phi_{n}(X_1,\cdots,X_n).
$$ 
Then $\{\Phi_{n}\}_{n\geq 1}$ can also be treated as a sequence of asymptotic $\alpha$-level goodness-of-fit tests with the null density being $p_0$. Moreover,
$$
{\rm power}\{\Phi_n; H_1^{\rm IND}(\Delta_n; s)\}\leq {\rm power}\{\Phi_n; H_1^{\rm GOF}(\Delta_n;s)'\}.
$$

It remains to show that given $\liminf_{n\to\infty}n^{2s/(d+4s)}\Delta_n< \infty$, there exists some $\alpha\in(0,1)$ such that
\begin{align*}
\liminf_{n\to\infty}{\rm power}\{\Phi_n; H_1^{\rm GOF}(\Delta_n;s)'\}<1,
\end{align*}
which cannot be directly obtained from Theorem \ref{th:goflower} because of the additional constraints 
\begin{align}\label{constraint}
p_1=p_{1,0},\quad p_2=p_{2,0}
\end{align}
in $H_1^{\rm GOF}(\Delta_n;s)'$.

However, by modifying the proof of Theorem \ref{th:goflower}, we only need to further require each $p_{n,\bm{a}_n}$ in the proof of Theorem \ref{th:goflower} satisfying (\ref{constraint}), or equivalently,
$$
\int_{\RR^{d_2}} (p-p_0)(x^1,x^2)d x^2=0,\quad \int_{\RR^{d_1}} (p-p_0)(x^1,x^2)d x^1=0.
$$

Recall that each $p_{n,\bm{a}_n}=p_0+r_n\sum\limits_{k=1}^{B_n}a_{n,k}\phi_{n,k}$, where
$$
\phi_{n,k}(x)=\frac{b_n^{d/2}}{\|\phi\|_{L_2}}\phi(b_nx-x_{n,k}).
$$
Write $x_{n,k}=(x_{n,k}^1,x_{n,k}^2)\in \RR^{d_1}\times \RR^{d_2}$. Since $\phi$ can be decomposed as 
$\phi(x^1,x^2)=\phi_1(x^1)\phi_2(x^2)$,
we have 
$$\phi_{n,k}(x)=\frac{b_n^{d/2}}{\|\phi\|_{L_2}}\phi_1(b_nx^1-x_{n,k}^1)\phi_2(b_nx^2-x_{n,k}^2)$$
Hence
\begin{align*}
\int_{\RR^{d_2}} (p_{n,\bm{a}_n}-p_0)(x^1,x^2)dx^2=&r_n\sum\limits_{k=1}^{B_n}a_{n,k}\int_{\RR^{d_2}} \phi_{n,k}(x^1,x^2)dx^2\\
=&r_n\sum\limits_{k=1}^{B_n}a_{n,k}\frac{b_n^{d/2}}{\|\phi\|_{L_2}}\cdot \phi_1(b_nx^1-x_{n,k}^1)\cdot\frac{1}{b_n^{d_2}}\int_{\RR^{d_2}} \phi_2(x^2)dx^2\\
=&0
\end{align*}
since $\int_{\RR^{d_2}}\phi_2(x^2)dx^2=0.$ Similarly, $\int_{\RR^{d_1}} (p_{n,\bm{a}_n}-p_0)(x^1,x^2)dx^1=0$. The proof is therefore finished.	
\end{proof}
\vskip 25pt

\begin{proof}[Proof of Theorem \ref{th:gofadapt}]
The proof of Theorem \ref{th:gofadapt} consists of two steps. First, we bound $q_{n,\alpha}^{\rm GOF}$. To be more specific, we show that there exists $C=C(d)>0$ such that $$q_{n,\alpha}^{\rm GOF}\leq C(d)\log\log n$$ for sufficiently large $n$, which holds if 
\begin{align}\label{eq:gofadapt1}
\lim\limits_{n\rightarrow \infty}P(T_n^{\rm GOF (adapt)}\geq C(d)\log\log n)=0
\end{align}
under $H_0^{\rm GOF}$. Second, we show that there exists $c>0$ such that 
$$\liminf_{n\to\infty} \Delta_{n,s}(n/\log\log n)^{2s/(d+4s)}>c$$ 
ensures
\begin{align}\label{eq:gofadapt2}
\inf_{p \in H_1^{\rm GOF(adapt)}(\Delta_{n,s}: s\ge d/4)}P(T_n^{\rm GOF (adapt)}\geq C(d)\log\log n)\to 1
\end{align}
as $n\to \infty$.

\paragraph{Verifying (\ref{eq:gofadapt1}).}
In order to prove (\ref{eq:gofadapt1}), we first show the following two lemmas. 
The first lemma suggests that $\hat{s}_{n,\nu_n}^2$
is a consistent estimator of $\EE \bGnn^2(X_1,X_2)$ uniformly over all $\nu_n\in[1,n^{2/d}]$. Recall we have shown in the proof of Theorem \ref{th:gofnull} that for $\nu_n$ increasing at a proper rate, 
$$
\hat{s}_{n,\nu_n}^2/\EE[\bGnn(X_1,X_2)]^2\to_p1.
$$
Hence the first lemma is a uniform version of such result.

\begin{lemma}\label{consistent-est-unif}
	We have that $\hat{s}_{n,\nu_n}^2/\EE[\bGnn(X_1,X_2)]^2$ converges to $1$ uniformly over $\nu_n\in[1,n^{2/d}]$, \ie,
	$$
	\sup\limits_{1\leq \nu_n\leq n^{2/d}}\left|\hat{s}_{n,\nu_n}^2/\EE[\bGnn(X_1,X_2)]^2-1\right|=o_p(1).
	$$
\end{lemma}

We defer the proof of Lemma \ref{consistent-est-unif} to the appendix. Note that
\begin{align*}
T_n^{\rm GOF (adapt)}=&\sup\limits_{1\leq \nu_n\leq n^{2/d}}\frac{n\hat{\gamma_{\nu_n}^2}(\PP,\PP_0)}{\sqrt{2\EE [\bGnn(X_1,X_2)]^2}}\cdot\sqrt{\EE [\bGnn(X_1,X_2)]^2/\hat{s}_{n,\nu_n}^2}\\\leq &\sup\limits_{1\leq \nu_n\leq n^{2/d}}\left|\frac{n\hat{\gamma_{\nu_n}^2}(\PP,\PP_0)}{\sqrt{2\EE [\bGnn(X_1,X_2)]^2}}\right|\cdot\sup\limits_{1\leq \nu_n\leq n^{2/d}}\sqrt{\EE [\bGnn(X_1,X_2)]^2/\hat{s}_{n,\nu_n}^2}.
\end{align*}
Lemma \ref{consistent-est-unif} first ensures that
$$
\sup\limits_{1\leq \nu_n\leq n^{2/d}}\sqrt{\EE [\bGnn(X_1,X_2)]^2/\hat{s}_{n,\nu_n}^2}=1+o_p(1).
$$	
It therefore suffices to show that under $H_0^{\rm GOF}$,
$$
\widetilde{T}_n^{\rm GOF (adapt)}:=\sup\limits_{1\leq \nu_n\leq n^{2/d}}\left|\frac{n\hat{\gamma_{\nu_n}^2}(\PP,\PP_0)}{\sqrt{2\EE [\bGnn(X_1,X_2)]^2}}\right|
$$
is also of order $\log\log n$. This is the crux of our argument yet its proof is lengthy. For brevity, we shall state it as a lemma here and defer its proof to the appendix.
\begin{lemma}\label{at4}
	There exists $C=C(d)>0$ such that
	$$
	\lim\limits_{n\rightarrow \infty}P\left(\widetilde{T}_n^{\rm GOF (adapt)}\geq C\log\log n\right)=0
	$$
	under $H_0^{\rm GOF}$.
\end{lemma}

\paragraph{Verifying (\ref{eq:gofadapt2}).} Let
$$
\nu_n(s)'=\left(\frac{\log\log n}{n}\right)^{-4/(4s+d)},
$$
which is smaller than $n^{2/d}$ for $s\geq {d/4}$. Hence it suffices to show
$$
\inf_{s\geq d/4}\inf_{p\in H_1^{\rm GOF}(\Delta_{n,s};s)}P(T_{n,\nu_n(s)'}^{\rm GOF}\geq C(d)\log\log n)\to 1
$$
as $n\rightarrow \infty$.

First of all, observe 
$$
0\leq\EE\left(\tilde{s}_{n,\nu_n(s)'}^2\right)\leq\EE G_{2\nu_n(s)'}(X_1,X_2)\leq M^2 (2\nu_n(s)'/\pi)^{-d/2}
$$
and
$$
\var\left(\tilde{s}_{n,\nu_n(s)'}^2\right)\lesssim _d M^3n^{-1}(\nu_n(s)')^{-3d/4}+M^2n^{-2}(\nu_n(s)')^{-d/2}
$$
for any $s$ and $p\in H_1^{\rm GOF}(\Delta_{n,s},s)$. Further considering $1/n^2=o(M^2 (2\nu_n(s)'/\pi)^{-d/2})$ uniformly over all $s$, we obtain that
$$
\inf_{s\geq d/4}\inf_{p\in H_1^{\rm GOF}(\Delta_{n,s};s)}P\left(
\hat{s}_{n,\nu_n(s)'}^2\leq 2M^2 (2\nu_n(s)'/\pi)^{-d/2}\right)\to 1.
$$

Let $$\Delta_{n,s}\geq c(\sqrt{M}+M)(\log\log n/n)^{2s/(d+4s)}$$ for some sufficiently large $c=c(d)$. Then 
$$
\EE\hat{\gamma_{\nu_n(s)'}^2}(\PP,\PP_0)=\gamma_{\nu_n(s)'}^2(\PP,\PP_0)\geq \left(\frac{\pi}{\nu_n(s)'}\right)^{d/2}\cdot\frac{\|p-p_0\|_{L_2}^2}{4},
$$
as guaranteed by Lemma \ref{le:gaussmmd}. Further considering that
$$
\var\left(\hat{\gamma_{\nu_n(s)'}^2}(\PP,\PP_0)\right)\lesssim_d M^2n^{-2}(\nu_{n}(s)')^{-d/2}+Mn^{-1}(\nu_{n}(s)')^{-3d/4}\|p-p_0\|_{L_2}^2,
$$
we immediately have
\begin{align*}
&\lim_{n\to\infty}\inf_{s\geq d/4}\inf_{p\in H_1^{\rm GOF}(\Delta_{n,s};s)}P(T_{n,\nu_n(s)'}^{\rm GOF}\geq C(d)\log\log n)\\
\geq& \lim_{n\to\infty}\inf_{s\geq d/4}\inf_{p\in H_1^{\rm GOF}(\Delta_{n,s};s)}P\left(\frac{n\gamma_{\nu_n(s)'}^2(\PP,\PP_0)/2}{\sqrt{2\hat{s}_{n,\nu_n(s)'}^2}}\geq C(d)\log\log n\right)= 1.
\end{align*}	
\end{proof}
\vskip 25pt

\begin{proof}[Proof of Theorem \ref{th:homadapt} and Theorem \ref{th:indadapt}]
The proof of Theorem \ref{th:homadapt} and Theorem \ref{th:indadapt} is very similar to that of Theorem \ref{th:gofadapt}. Hence we only emphasize the main differences here.

\paragraph{For adaptive homogeneity test:} to verify that there exists $C=C(d)>0$ such that
$$
\lim\limits_{n\rightarrow \infty}P(T_n^{\rm HOM (adapt)}\geq C\log\log n)=0
$$	
under $H_0^{\rm HOM}$, observe that
$$
T_n^{\rm HOM (adapt)}\leq \sup\limits_{1\leq \nu_n\leq n^{2/d}}\sqrt{\frac{\EE[\bGnn(X_1,X_2)]^2}{\hat{s}_{n,m,\nu_n}^2}}\cdot \left(\frac{1}{n}+\frac{1}{m}\right)^{-1}\sup\limits_{1\leq \nu_n\leq n^{2/d}} \frac{|\hat{\gamma_{\nu_n}^2}(\PP,\QQ)|}{\sqrt{2\EE[\bGnn(X_1,X_2)]^2}}.
$$
Denote $X_1,\cdots,X_n,Y_1,\cdots,Y_m$ as $Z_1,\cdots,Z_N$. Hence
$$
2\sum_{i=1}^n\sum\limits_{j=1}^m G_{\nu_n}(X_i,Y_j)=\sum\limits_{1\leq i\neq j\leq N}G_{\nu_n}(Z_i,Z_{j})-\sum\limits_{1\leq i\neq j\leq n}G_{\nu_n}(X_i,X_{j})-\sum\limits_{1\leq i\neq j\leq m}G_{\nu_n}(Y_i,Y_j)
$$
and
\begin{align*}
&\sup\limits_{1\leq \nu_n\leq n^{2/d}} \frac{|\hat{\gamma_{\nu_n}^2}(\PP,\QQ)|}{\sqrt{2\EE[\bGnn(X_1,X_2)]^2}}\\\leq &\left(\frac{1}{n(n-1)}+\frac{1}{mn}\right)\sup\limits_{1\leq \nu_n\leq n^{2/d}}\left|\sum\limits_{1\leq i\neq j\leq n}\frac{\bar{G}_{\nu_n}(X_i,X_{j})}{\sqrt{2\EE[\bGnn(X_1,X_2)]^2}}\right|\\&+\left(\frac{1}{m(m-1)}+\frac{1}{mn}\right)\sup\limits_{1\leq \nu_n\leq n^{2/d}}\left|\sum\limits_{1\leq i\neq j\leq m}\frac{\bar{G}_{\nu_n}(Y_i,Y_j)}{\sqrt{2\EE[\bGnn(X_1,X_2)]^2}}\right|\\
&+\frac{1}{mn}\sup\limits_{1\leq \nu_n\leq n^{2/d}}\left|\sum\limits_{1\leq i\neq j\leq N}\frac{\bar{G}_{\nu_n}(Z_i,Z_{j})}{\sqrt{2\EE[\bGnn(X_1,X_2)]^2}}\right|
\end{align*}
Apply Lemma \ref{at4} to bound each term of the right hand side of the above inequality. Then we conclude that for some $C=C(d)>0$,
$$
\lim\limits_{n\rightarrow \infty}P\left(\left(\frac{1}{n}+\frac{1}{m}\right)^{-1}\sup\limits_{1\leq \nu_n\leq n^{2/d}} \frac{|\hat{\gamma_{\nu_n}^2}(\PP,\QQ)|}{\sqrt{2\EE[\bGnn(X_1,X_2)]^2}}\geq C\log\log n\right)=0.
$$

\paragraph{For adaptive independence test:} to verify that there exists $C=C(d)>0$ such that
\begin{align}\label{eq:indadapt}
\lim\limits_{n\rightarrow \infty}P(T_n^{\rm IND (adapt)}\geq C\log\log n)=0
\end{align}
under $H_0^{\rm IND}$, 
recall the decomposition 
$$
\hat{\gamma_{\nu_n}^2}(\PP,\PP^{X^1}\otimes \PP^{X^2})=D_2(\nu_n)+R_n=\frac{1}{n(n-1)}\sum\limits_{1\leq i\neq j\leq n}G_{\nu_n}^*(X_i,X_j)+R_n,
$$
where we express $R_n$ as $R_n=D_3(\nu_n)+D_4(\nu_n)$ in the proof of Theorem \ref{th:indpower}.

Following arguments similar to those in the proof of Lemma \ref{at4}, we obtain that there exists $C(d)>0$ such that for sufficiently large $n$,
$$
P\left(\sup\limits_{1\leq \nu_n\leq n^{2/d}}\left|\frac{nD_{2}(\nu_n)}{\sqrt{2\EE [G_{\nu_n}^*(X_1,X_2)]^2}}\right|\geq C(d)(\log\log n+t\log\log\log n )\right)\lesssim \exp(-t^{2/3}),
$$
Similarly, 
\begin{align*}
&P\left(\sup\limits_{1\leq \nu_n\leq n^{2/d}}\left|\frac{n^{3/2}D_{3}(\nu_n)}{\sqrt{2\EE [G_{\nu_n}^*(X_1,X_2)]^2}}\right|\geq C(d)(\log\log n+t\log\log\log n )\right)\lesssim\exp(-t^{1/2})\\
&P\left(\sup\limits_{1\leq \nu_n\leq n^{2/d}}\left|\frac{n^2D_{4}(\nu_n)}{\sqrt{2\EE [G_{\nu_n}^*(X_1,X_2)]^2}}\right|\geq C(d)(\log\log n+t\log\log\log n )\right)\lesssim\exp(-t^{2/5})
\end{align*}
for sufficiently large $n$. 

On the other hand, note that
$$
\EE [G_{\nu_n}^*(X_1,X_2)]^2=\prod\limits_{j=1}^2\EE[\bGnn(X_1^j,X_2^j)]^2,
$$
and based on results in the proof of Lemma \ref{consistent-est-unif}, 
$
\sup\limits_{1\leq \nu_n\leq n^{2/d}}\left|\tilde{s}^2_{n,j,\nu_n}/\EE[\bGnn(X_1^j,X_2^j)]^2-1\right|=o_p(1)
$
for $j=1,2$. Further considering that 
$$
1/n^2=o(\EE[G_{\nu_n}^*(X_1,X_2)]^2)
$$
uniformly over all $\nu_n\in[1,n^{2/d}]$, we obtain
$$
\sup\limits_{1\leq \nu_n\leq n^{2/d}}\left|\hat{s}_{n,\nu_n}^2/\EE[G_{\nu_n}^*(X_1,X_2)]^2-1\right|=o_p(1).
$$
They combined together ensure that (\ref{eq:indadapt}) holds.

To show that the detection boundary of $\Phi^{\rm IND(adapt)}$ is of order $O((n/\log\log n)^{-2s/(d+4s)})$, observe that
$$
0\leq\EE\left(\tilde{s}_{n,j,\nu_n(s)'}^2\right)\leq\EE G_{2\nu_n(s)'}(X_1^j,X_2^j)\leq M_j^2 (2\nu_n(s)'/\pi)^{-d_j/2}
$$
and
$$
\var\left(\tilde{s}_{n,j,\nu_n(s)'}^2\right)\lesssim _{d_j} M_j^3n^{-1}(\nu_n(s)')^{-3d_j/4}+M_j^2n^{-2}(\nu_n(s)')^{-d_j/2}
$$
for $j=1,2$, where $\nu_n(s)'=\left(\log\log n/n\right)^{-4/(4s+d)}$ as in the proof of Theorem \ref{th:gofadapt}. Therefore,
$$
\inf_{s\geq d/4}\inf_{p\in H_1^{\rm IND}(\Delta_{n,s};s)}P\left(
\left|\tilde{s}_{n,j,\nu_n(s)'}^2\right|\leq \sqrt{3/2}M_j^2 (2\nu_n(s)'/\pi)^{-d_j/2}\right)\to 1,\quad j=1,2.
$$	
Further considering $1/n^2=o(M^2 (2\nu_n(s)'/\pi)^{-d/2})$ uniformly over all $s$, we obtain that
$$
\inf_{s\geq d/4}\inf_{p\in H_1^{\rm IND}(\Delta_{n,s};s)}P\left(
\hat{s}_{n,\nu_n(s)'}^2\leq 2M^2 (2\nu_n(s)'/\pi)^{-d/2}\right)\to 1.
$$		
\end{proof}

\bibliographystyle{plainnat}
\bibliography{mmdref}

\appendix

\section{Properties of Gaussian Kernel}
We collect here a couple of useful properties of Gaussian kernel that we used repeated in the proof to the main results.
\begin{lemma}
\label{le:gausskernel}
For any $f\in L_2(\RR^d)$,
$$
\int G_{\nu}(x,y)f(x)f(y)dxdy=\left(\frac{\pi}{\nu}\right)^{\frac{d}{2}}\int\exp\left(-\frac{\|\omega\|^2}{4\nu}\right)\left\|\Fcal{f}(\omega)\right\|^2 d \omega.
$$
\end{lemma}

\begin{proof}
Denote by $Z$ a Gaussian random vector with mean $0$ and covariance matrix $2\nu I_d$. Then
\begin{align*}
\int G_{\nu}(x,y)f(x)f(y)dxdy=&\int\exp\left(-\nu\|x-y\|^2\right)f(x)f(y)d xdy\\
=&\int\EE\exp[iZ^\top (x-y)]f(x)f(y)dxdy\\
=&\EE\left\|\int\exp(-iZ^\top x)f(x)dx\right\|^2\\
=&\int\frac{1}{(4\pi\nu)^{d/2}}\exp\left(-\frac{\|\omega\|^2}{4\nu}\right)\left\|\int\exp(-i\omega^\top x)f(x)dx\right\|^2\\
=&\left(\frac{\pi}{\nu}\right)^{\frac{d}{2}}\int\exp\left(-\frac{\|\omega\|^2}{4\nu}\right)\left\|\Fcal{f}(\omega)\right\|^2d\omega,
\end{align*}
which concludes the proof.
\end{proof}

A useful consequence of Theorem \ref{le:gausskernel} is a close connection between Gaussian kernel MMD and $L_2$ norm.

\begin{lemma}
\label{le:gaussmmd}
For any $f\in \Wcal^{s,2}(M)$
$$
\left(\nu\over \pi\right)^{d/2}\int G_{\nu}(x,y)f(x)f(y)dxdy\ge {1\over 4}\|f\|^2_{L_2},
$$
provided that
$$
\nu^s\ge {4^{1-s}M^2\over (\log 3)^{s}}\cdot\|f\|_{L_2}^{-2}.
$$
\end{lemma}

\begin{proof} In light of Lemma \ref{le:gausskernel},
$$
\left(\nu\over \pi\right)^{d/2}\int G_{\nu}(x,y)f(x)f(y)dxdy=\int\exp\left(-\frac{\|\omega\|^2}{4\nu}\right)\left\|\Fcal{f}(\omega)\right\|^2d\omega.
$$
By Plancherel Theorem, for any $T>0$,
$$
\int_{\|\omega\|\leq T}\left\|\Fcal{f}(\omega)\right\|^2d\omega=\|f\|_{L^2}^2-\int_{\|\omega\|> T}\left\|\Fcal{f}(\omega)\right\|^2d\omega\geq \|f\|_{L^2}^2-\frac{M^2}{T^{2s}},
$$
Choosing
$$T=\left(\frac{2M}{\|f\|_{L^2}}\right)^{1/s},$$
yields
$$
\int_{\|\omega\|\leq T}\left\|\Fcal{f}(\omega)\right\|^2d\omega\geq \frac{3}{4}\|f\|_{L^2}^2.
$$
Hence
\begin{align*}
\int\exp\left(-\frac{\|\omega\|^2}{4\nu}\right)\left\|\Fcal{f}(\omega)\right\|^2d \omega&\geq \exp\left(-\frac{T^2}{4\nu}\right)\int_{\|\omega\|\leq T}\left\|\Fcal{f}(\omega)\right\|^2d \omega\\
&\geq \frac{3}{4} \exp\left(-\frac{T^2}{4\nu}\right)\|f\|_{L^2}^2.
\end{align*}
In particular, if
$$\nu\geq \frac{(2M)^{2/s}}{4\log 3}\cdot\|f\|_{L^2}^{-2/s},$$
then
$$
\int\exp\left(-\frac{\|\omega\|^2}{4\nu}\right)\left\|\Fcal{f}(\omega)\right\|^2d \omega\geq\frac{1}{4}\|f\|_{L^2}^2,
$$
which concludes the proof.
\end{proof}

\section{Proof of Lemma \ref{consistent-est-unif}}
	We first prove that $\sup\limits_{1\leq \nu_n\leq n^{2/d}}\left|\tilde{s}_{n,\nu_n}^2/\EE[\bGnn(X_1,X_2)]^2-1\right|=o_p(1)$ and then show the difference caused by the modification from $\tilde{s}_{n,\nu_n}^2$ to $\hat{s}_{n,\nu_n}^2$ is asymptotically negligible.
	
	Note that
	\begin{align*}
	&\sup\limits_{1\leq \nu_n\leq n^{2/d}}\left|\tilde{s}_{n,\nu_n}^2/\EE[\bGnn(X_1,X_2)]^2-1\right|\\\leq&\left(\inf\limits_{1\leq \nu_n\leq n^{2/d}}\nu_n^{d/2}\EE[\bGnn(X_1,X_2)]^2\right)^{-1}\cdot\sup\limits_{1\leq \nu_n\leq n^{2/d}}\nu_n^{d/2}\left|\tilde{s}_{n,\nu_n}^2-\EE[\bGnn(X_1,X_2)]^2\right|.
	\end{align*}
	
	For $X\sim \PP_0$, denote the distribution of $(X,X)$ as $\PP_1$. Then we have
	$$
	\EE[\bGnn(X_1,X_2)]^2=\gamma_{\nu_n}^2(\PP_1,\PP_0\otimes\PP_0).
	$$
	Hence $\EE[\bGnn(X_1,X_2)]^2>0$ for any $\nu_n>0$ since $G_{\nu_n}$ is characteristic.
	
	In addition, $\nu_n^{d/2}\EE[\bGnn(X_1,X_2)]^2$ is continuous with respect to $\nu_n$ and
	$$
	\lim\limits_{\nu_n\to \infty}\nu_n^{d/2}\EE[\bGnn(X_1,X_2)]^2=\left(\frac{\pi}{2}\right)^{d/2}\|p_0\|_{L_2}^2.
	$$
	Therefore,
	$$
	\inf\limits_{1\leq \nu_n\leq n^{2/d}}\nu_n^{d/2}\EE[\bGnn(X_1,X_2)]^2\geq \inf\limits_{\nu_n\in[0,\infty)}\nu_n^{d/2}\EE[\bGnn(X_1,X_2)]^2>0,
	$$
	and it remains to prove
	$$
	\sup\limits_{1\leq \nu_n\leq n^{2/d}}\nu_n^{d/2}\left|\tilde{s}_{n,\nu_n}^2-\EE[\bGnn(X_1,X_2)]^2\right|=o_p(1).
	$$
	Recall the expression of $\tilde{s}_{n,\nu_n}^2$. It suffcies to show that
	\begin{align}
	&\sup\limits_{1\leq \nu_n\leq n^{2/d}}\nu_n^{d/2}\left|\frac{1}{n(n-1)}\sum\limits_{1\leq i\neq j\leq n}G_{2\nu_n}(X_i,X_j)-\EE G_{2\nu_n}(X_1,X_2)\right|\label{var_comp1}\\
	&\sup\limits_{1\leq \nu_n\leq n^{2/d}}\nu_n^{d/2}\left|\frac{2(n-3)!}{n!}\sum\limits_{\substack{1\le i,j_1,j_2\le n\\ |\{i,j_1,j_2\}|=3}}G_{\nu_n}(X_i,X_{j_1})G_{\nu_n}(X_i,X_{j_2})-\EE G_{\nu_n}(X_1,X_2)G_{\nu_n}(X_1,X_3)\right|\label{var_comp2}\\
	&\sup\limits_{1\leq \nu_n\leq n^{2/d}}\nu_n^{d/2}\left|\frac{(n-4)!}{n!}\sum\limits_{\substack{1\le i_1,i_2,j_1,j_2\le n\\ |\{i_1,i_2,j_1,j_2\}|=4}}G_{\nu_n}(X_{i_1},X_{j_1})G_{\nu_n}(X_{i_2},X_{j_2})-[\EE G_{\nu_n}(X_1,X_2)]^2\right|\label{var_comp3}
	\end{align}
	are all $o_p(1)$. We shall first control (\ref{var_comp1})
	and then bound (\ref{var_comp2}) and (\ref{var_comp3}) in the same way.
	
	Let
	$$
	\HEE_n G_{2\nu_n}(X,X')=\frac{1}{n(n-1)}\sum\limits_{1\leq i\neq j\leq n}G_{2\nu_n}(X_i,X_j).
	$$
	In the rest of this proof, abbreviate $\HEE_n G_{2\nu_n}(X,X')$ and $\EE G_{2\nu_n}(X_1,X_2)$ as $\HEE_n G_{2\nu_n}$ and $\EE G_{2\nu_n}$ respectively when no confusion occurs.
	
	Divide the whole interval $[1,n^{2/d}]$ into $A$ sub-intervals, $[u_0,u_1],[u_1,u_2],\cdots,[u_{A-1},u_A]$ with $u_0=1$, $u_A=n^{2/d}$. For any $\nu_n\in[u_{a-1},u_a]$,
	\begin{align*}
	\nu_n^{d/2}\HEE_n G_{2\nu_n}-\nu_n^{d/2}\EE G_{2\nu_n}\geq &-\nu_n^{d/2}\left| \HEE_n G_{2u_{a}}-\EE G_{2u_a}\right|-\nu_n^{d/2}\left|\EE G_{2u_a}-\EE G_{2u_{a-1}}\right|\\\geq &-u_a^{d/2}\left| \HEE_n G_{2u_{a}}-\EE G_{2u_a}\right|-u_a^{d/2}\left|\EE G_{2u_a}-\EE G_{2u_{a-1}}\right|
	\end{align*}
	and
	$$
	\nu_n^{d/2}\HEE_n G_{2\nu_n}-\nu_n^{d/2}\EE G_{2\nu_n}\leq u_a^{d/2}\left|\HEE_n G_{2u_{a-1}}-\EE G_{2u_{a-1}}\right|+u_a^{d/2}\left|\EE G_{2u_a}-\EE G_{2u_{a-1}}\right|,
	$$
	which together ensure that
	\begin{align*}
	&\sup\limits_{1\leq \nu_n\leq n^{2/d}}\left|\nu_n^{d/2}\HEE_nG_{2\nu_n}-\nu_n^{d/2}\EE G_{2\nu_n}\right|\\ \leq &\sup\limits_{1\leq a\leq A}\left(\frac{u_a}{u_{a-1}}\right)^{d/2}\cdot\sup\limits_{0\leq a\leq A}u_a^{d/2}\left| \HEE_n G_{2u_{a}}-\EE G_{2u_a}\right|+\sup\limits_{1\leq a\leq A}u_a^{d/2}\left|\EE G_{2u_a}-\EE G_{2u_{a-1}}\right|\\
	\leq & \sup\limits_{1\leq a\leq A}\left(\frac{u_a}{u_{a-1}}\right)^{d/2}\cdot\sup\limits_{0\leq a\leq A}u_a^{d/2}\left| \HEE_n G_{2u_{a}}-\EE G_{2u_a}\right|+\sup\limits_{1\leq a\leq A}\left|u_a^{d/2}\EE G_{2u_a}-u_{a-1}^{d/2}\EE G_{2u_{a-1}}\right|\\
	&+\sup\limits_{1\leq a\leq A}\left(\left(u_a^{d/2}-u_{a-1}^{d/2}\right)\EE G_{2u_{a-1}}\right).
	\end{align*}
	Bound the three terms in the right hand side of the last inequality separately. 
	
	Let $\{u_a\}_{a\geq 0}$ be a geometric sequence, namely, 
	$$
	A:=\inf\{a\in\NN: r^a\geq n^{2/d}\},
	$$
	and
	$$u_a=\begin{cases}
	r^a,\quad \forall\ 0\leq a\leq A-1\\
	n^{2/d}, \quad a=A
	\end{cases},
	$$
	with $r>1$ to be determined later.
	
	Since
	$
	\lim\limits_{\nu\rightarrow \infty} \nu^{d/2}\EE G_{2\nu_n}=(\pi/2)^{d/2}\|p_0\|^2
	$
	and $\nu^{d/2}\EE G_{2\nu}$ is continuous, we obtain that for any $\varepsilon>0$, there exsits sufficiently small $r>1$ such that
	$$
	\sup\limits_{1\leq a\leq A}\left|u_a^{d/2}\EE G_{2u_a}-u_{a-1}^{d/2}\EE G_{2u_{a-1}}\right|\leq \varepsilon.
	$$
	At the same time, we can also ensure 
	$$
	\sup\limits_{1\leq a\leq A}\left(\left(u_a^{d/2}-u_{a-1}^{d/2}\right)\EE G_{2u_{a-1}}\right)\leq (r^{d/2}-1)\left(\frac{\pi}{2}\right)^{d/2}\|p_0\|^2\leq \varepsilon
	$$
	by choosing $r$ sufficiently small.
	
	Finally consider 
	$$\sup\limits_{1\leq a\leq A}\left(\frac{u_a}{u_{a-1}}\right)^{d/2}\cdot\sup\limits_{0\leq a\leq A}u_a^{d/2}\left| \HEE_n G_{2u_{a}}-\EE G_{2u_a}\right|.
	$$
	On the one hand,
	$$
	\sup\limits_{1\leq a\leq A}\left(\frac{u_a}{u_{a-1}}\right)^{d/2}\leq r^{d/2}.
	$$
	On the other hand, since 
	$$
	\var\left(\HEE_n G_{2\nu_n}\right)\lesssim \frac{1}{n}\EE G_{2\nu_n}(X,X')G_{2\nu_n}(X,X'')+\frac{1}{n^2}\EE G_{4\nu_n}(X,X')\lesssim_d \frac{\nu_n^{-3d/4}\|p_0\|^3}{n}+\frac{\nu_n^{-d/2}\|p_0\|^2}{n^2}
	$$
	for any $\nu_n\in(0,\infty)$, we have
	\begin{align*}
	&P\left(\sup\limits_{0\leq a\leq A}u_a^{d/2}\left| \HEE_n G_{2u_{a}}-\EE G_{2u_a}\right|\geq \varepsilon\right)\\\leq& \frac{\sum\limits_{a=0}^Au_a^d\var\left(\HEE_n G_{2u_a}\right)}{\varepsilon^2}\lesssim_{d,r}\frac{1}{\varepsilon^2}\left(\frac{u_A^{d/4}\|p_0\|^3}{n}+\frac{u_A^{d/2}\|p_0\|^2}{n^2}\right)\rightarrow 0
	\end{align*}
	as $n\rightarrow \infty$. Hence we conclude
	$
	\sup\limits_{1\leq \nu_n\leq n^{2/d}}\left|\nu_n^{d/2}\HEE_nG_{2\nu_n}-\nu_n^{d/2}\EE G_{2\nu_n}\right|=o_p(1)
	$.
	
	Considering that
	$$
	\lim\limits_{\nu_n\to\infty}\nu_n^{d/2}\EE G_{\nu_n}(X_1,X_2)G_{\nu_n}(X_1,X_3)=0,\quad \lim\limits_{\nu_n\to\infty}\nu_n^{d/2}[\EE G_{\nu_n}(X_1,X_2)]^2=0,
	$$
	we obtain that (\ref{var_comp2}) and (\ref{var_comp3}) are also $o_p(1)$, based on almost the same arguments. Hence
	$$\sup\limits_{1\leq \nu_n\leq n^{2/d}}\left|\tilde{s}_{n,\nu_n}^2/\EE[\bGnn(X_1,X_2)]^2-1\right|=o_p(1).$$
	
	On the other hand, since $\EE[\bGnn(X_1,X_2)]^2\gtrsim_{p_0,d}\nu_n^{-d/2}$
	for $\nu_n\in[1,n^{2/d}]$, 
	$$
	\sup\limits_{1\leq \nu_n\leq n^{2/d}} \frac{1}{n^2\EE[\bGnn(X_1,X_2)]^2}=o_p(1).
	$$
	Hence we finally conclude that
	$$\sup\limits_{1\leq \nu_n\leq n^{2/d}}\left|\hat{s}_{n,\nu_n}^2/\EE[\bGnn(X_1,X_2)]^2-1\right|=o_p(1).$$

\section{Proof of Lemma \ref{at4}}
	Let
	$$
	K_{\nu_n}(x,x')=\frac{G_{\nu_n}(x,x')}{\sqrt{2\EE G_{2\nu_n}(X_1,X_2)}},\quad\forall\ x,x'\in\mathbb{R}^d,
	$$
	and accordingly,
	$$
	\bar{K}_{\nu_n}(x,x')=\frac{\bar{G}_{\nu_n}(x,x')}{\sqrt{2\EE G_{2\nu_n}(X_1,X_2)}}.
	$$
	Hence
	$$
	\tilde{T}_{n}^{\rm GOF(adapt)}=\sup\limits_{1\leq \nu_n\leq n^{2/d}}\left|\frac{1}{n-1}\sum\limits_{i\neq j}\bar{K}_{\nu_n}(X_i,X_j)\cdot \sqrt{\frac{\EE G_{2\nu_n}(X_1,X_2)}{\EE[\bGnn(X_1,X_2)]^2}}\right|.
	$$
	
	To finish this proof, we first bound 
	\begin{align}\label{adapt}
	\sup\limits_{1\leq \nu_n\leq n^{2/d}}\left|\frac{1}{n-1}\sum\limits_{i\neq j}\bar{K}_{\nu_n}(X_i,X_j)\right|
	\end{align}
	and then control $\tilde{T}_{n}^{\rm GOF(adapt)}$.
	
	\paragraph{Step (\ro).} There are two main tools that we borrow in this step. First, we apply results in \cite{arcones1993limit} to obtain a Bernstein-type inequality for 
	$$
	\left|\frac{1}{n-1}\sum\limits_{i\neq j}\bar{K}_{\nu_0}(X_i,X_j)\right|\ \mathrm{and}\ \left|\frac{1}{n-1}\sum\limits_{i\neq j}\left(\bar{K}_{\nu_n}(X_i,X_j)-\bar{K}_{\nu_n'}(X_i,X_j)\right)\right|
	$$
	for some $\nu_0$ and arbitrary $\nu_n,\nu_n'\in[1,\infty)$.
	And based on that, we borrow Talagrand's techniques on handling Bernstein-type inequality (\eg, see \citet{talagrand2014upper}) to give a generic chaining bound of (\ref{adapt}).
	
	To be more specific, for any $\nu_0,\nu_n,\nu_n'\in[1,n^{2/d}]$, define
	$$
	d_1(\nu_n,\nu_n')=\|\bar{K}_{\nu_n'}-\bar{K}_{\nu_n}\|_{L_{\infty}},\quad d_2(\nu_n,\nu_n')=\|\bar{K}_{\nu_n'}-\bar{K}_{\nu_n}\|_{L_2}.
	$$
	Then Proposition 2.3 (c) of \cite{arcones1993limit} ensures that for any $t>0$,
	\begin{align}\label{bern1}
	P\left(\left|\frac{1}{n-1}\sum\limits_{i\neq j}\bar{K}_{\nu_0}(X_i,X_j)\right|\geq t\right)\leq C\exp\left(-C\min\left\{\frac{t}{\|\bar{K}_{\nu_0}\|_{L_2}},\left(\frac{\sqrt{n}t}{\|\bar{K}_{\nu_0}\|_{L_\infty}}\right)^{\frac{2}{3}}\right\}\right)
	\end{align}
	and
	\begin{align*}
	&P\left(\left|\frac{1}{n-1}\sum\limits_{i\neq j}\left(\bar{K}_{\nu_n}(X_i,X_j)-\bar{K}_{\nu_n'}(X_i,X_j)\right)\right|\geq t\right)\\\leq &C\exp\left(-C\min\left\{\frac{t}{d_2(\nu_n,\nu_n')},\left(\frac{\sqrt{n}t}{d_1(\nu_n,\nu_n')}\right)^{\frac{2}{3}}\right\}\right)
	\end{align*}
	for some $C>0$, and based on a chaining type argument (see, \eg, Theorem 2.2.28 in \cite{talagrand2014upper}) the latter inequality suggests there exists $C>0$ such that
	\begin{align}\label{chaining}
	P\Bigg(\sup\limits_{1\leq \nu_n\leq n^{2/d}}&\left|\frac{1}{n-1}\sum\limits_{i\neq j}\left(\bar{K}_{\nu_n}(X_i,X_j)-\bar{K}_{\nu_0}(X_i,X_j)\right)\right|\geq\\ &C\left(\frac{\gamma_{2/3 }([1,n^{2/d}],d_1)}{\sqrt{n}}t+\gamma_1([1,n^{2/d}],d_2)+D_2t\right) \Bigg)\notag\lesssim \exp(-t^{2/3}),
	\end{align}
	where $\gamma_{2/ 3}([1,n^{2/d}],d_1)$, $\gamma_1([1,n^{2/d}],d_2)$ are the so-called $\gamma$-functionals and $$D_2=\sum\limits_{l\geq 0}e_l([1,n^{2/d}],d_2)$$ with $e_l$ being the so-called entropy numbers. 
	
	A straightforward combination of (\ref{bern1}) and (\ref{chaining}) then gives
	\begin{align*}
	P\Bigg(&\sup\limits_{1\leq \nu_n\leq n^{2/d}}\left|\frac{1}{n-1}\sum\limits_{i\neq j}\bar{K}_{\nu_n}(X_i,X_j)\right|\geq\\ &C\left(\frac{\gamma_{2/3}([1,n^{2/d}],d_1)}{\sqrt{n}}t+\gamma_1([1,n^{2/d}],d_2)+D_2t+\frac{\|\bar{K}_{\nu_0}\|_{L_{\infty}}}{\sqrt{n}}+\|\bar{K}_{\nu_0}\|_{L_2}t\right) \Bigg)\lesssim \exp(-t^{2/3}).
	\end{align*}
	
	Therefore, given that the bounds on $\|\bar{K}_{\nu_0}\|_{L_2}$ and $\|\bar{K}_{\nu_0}\|_{L_{\infty}}$ can be obtained quite directly, \eg, with $\nu_0=1$,
	$$
	\|\bar{K}_{\nu_0}\|_{L_{\infty}}\leq 4\|K_{\nu_0}||_{L_{\infty}}=\frac{4}{\sqrt{2\EE G_{2}}},\qquad \|\bar{K}_{\nu_0}\|_{L_2}\leq \|K_{\nu_0}\|_{L_2}=\frac{\sqrt{2}}{2},
	$$ 
	the main focus is to bound $\gamma_{2/3}([1,n^{2/d}],d_1)$, $\gamma_1([1,n^{2/d}],d_2)$ and $D_2$ properly.
	
	First consider $\gamma_{2/3}([1,n^{2/d}],d_1)$. Note that for any $1\leq \nu_n<\nu_n'<\infty$, 
	$$
	d_1(\nu_n,\nu_n')\leq 4\|K_{\nu_n}-K_{\nu_n'}\|_{L_{\infty}}\leq 4\int_{\nu_n}^{\nu_n'}\left\|\frac{dK_u}{du}\right\|_{L_{\infty}}du
	$$
	
	Since for any $\nu_n$,
	\begin{align*}
	\frac{dK_{\nu_n}}{d\nu_n}=&(-\|x-x'\|^2)G_{\nu_n}(X_1,X_2)\left(\EE G_{2\nu_n}(X_1,X_2)\right)^{-1/2}\\-
	&\frac{1}{2}G_{\nu_n}(X_1,X_2)\left(\EE G_{2\nu_n}(X_1,X_2)\right)^{-3/2}\frac{d}{d\nu_n}\EE G_{2\nu_n}(X_1,X_2)
	\end{align*}
	where
	\begin{align*}
	\left(\EE G_{2\nu_n}(X_1,X_2)\right)^{-1/2}&=\left(\frac{\pi}{2}\right)^{-d/4}\nu_n^{d/4}\left(\int\exp\left(-\frac{\|\omega\|^2}{8\nu_n}\right)\|\Fcal{p_0}(\omega)\|^2d\omega\right)^{-1/2}\\
	&\lesssim_d \nu_n^{d/4}\left(\int\exp\left(-\frac{\|\omega\|^2}{8}\right)\|\Fcal{p_0}(\omega)\|^2d\omega\right)^{-1/2},
	\end{align*}
	\begin{align*}
	\left(\EE G_{2\nu_n}(X_1,X_2)\right)^{-3/2}\lesssim_d \nu_n^{3d/4}\left(\int\exp\left(-\frac{\|\omega\|^2}{8}\right)\|\Fcal{p_0}(\omega)\|^2d\omega\right)^{-3/2},
	\end{align*}
	and
	\begin{align*}
	&\frac{d}{d\nu_n}\EE_{2\nu_n}(X_1,X_2)\\=&\left(\frac{\pi}{2}\right)^{d/2}\nu_n^{-d/2-1}\left(-\frac{d}{2}\cdot \int\exp\left(-\frac{\|\omega\|^2}{8\nu_n}\right)\|\Fcal{p_0}(\omega)\|^2d\omega\right.\\&\left.+\int\exp\left(-\frac{\|\omega\|^2}{8\nu_n}\right)\left(\frac{\|\omega\|^2}{8\nu_n}\right)\|\Fcal{p_0}(\omega)\|^2d\omega\right),
	\end{align*}
	which together ensure 
	$$
	\left\|\frac{dK_{\nu_n}}{d\nu_n}\right\|_{L_{\infty}}\lesssim_{d,p_0} \nu_n^{d/4-1}.
	$$
	Hence
	$$
	d_1(\nu_n,\nu_n')\lesssim_{d,p_0}|\nu_n^{d/4}-(\nu_n')^{d/4}|,
	$$
	and 
	$
	\gamma_{2/3}([1,n^{2/d}],d_1)\lesssim_{d,p_0}|(n^{2/d})^{d/4}-1^{d/4}|\leq \sqrt{n}
	$.
	
	Then consider $\gamma_1([1,n^{2/d}],d_2)$. We have
	\begin{align*}
	d_2^2(\nu_n,\nu_n')\leq \|K_{\nu_n'}-K_{\nu_n}\|_{L_2}^2=1-\frac{\EE G_{\nu_n}G_{\nu_n'}}{\sqrt{\EE G_{2\nu_n}\EE G_{2\nu_n'}}}\leq -\log\left(\frac{\EE G_{\nu_n}G_{\nu_n'}}{\sqrt{\EE G_{2\nu_n}\EE G_{2\nu_n'}}}\right)
	\end{align*}
	Let $f_1(\nu_n)=\int\exp\left(-\frac{\|\omega\|^2}{8\nu_n}\right)\|\Fcal{p_0}(\omega)\|^2d\omega$. Then
	$$
	\log \left(\EE G_{2\nu_n}\right)=\frac{d}{2}\log\left(\frac{\pi}{2\nu_n}\right)+\log f_1(\nu_n)
	$$
	and hence
	\begin{align*}
	&-\log\left(\frac{\EE G_{\nu_n}G_{\nu_n'}}{\sqrt{\EE G_{2\nu_n}\EE G_{2\nu_n'}}}\right)\\=&\frac{d}{2}\left(-\frac{\log \nu_n+\log \nu_n'}{2}+\log\left(\frac{\nu_n+\nu_n'}{2}\right)\right)+\left(\frac{\log f_1(\nu_n)+\log f_1(\nu_n')}{2}-\log f_1\left(\frac{\nu_n+\nu_n'}{2}\right)\right).
	\end{align*}
	
	Note that
	$$
	\frac{\log f_1(\nu_n)+\log f_1(\nu_n')}{2}-\log f_1\left(\frac{\nu_n+\nu_n'}{2}\right)=\frac{1}{2}\int_0^{\frac{\nu_n'-\nu_n}{2}}\int_{-u}^u \left(\log f_1\left(\frac{\nu_n'+\nu_n}{2}+v\right)\right)''dvdu.
	$$
	For any $\nu_n\geq 1$,
	$$
	\left(\log f_1(\nu_n)\right)''=\frac{f_1(\nu_n)f_1''(\nu_n)-(f_1'(\nu_n))^2}{f_1^2(\nu_n)}\leq \frac{f_1''(\nu_n)}{f_1(\nu_n)},
	$$
	and
	$$
	f_1''(\nu_n)=\int\exp\left(-\frac{\|\omega\|^2}{8\nu_n}\right)\left(\frac{\|\omega\|^4}{64\nu_n^4}-\frac{\|\omega\|^2}{4\nu_n^3}\right)\|\Fcal{p_0}(\omega)\|^2d\omega\lesssim \nu_n^{-2}\|p_0\|_{L_2}^2.
	$$
	Moreover, there exists $\nu_n^*=\nu_n^*(p_0)>1$ such that $f_1(\nu_n^*)\geq \|p_0\|_{L_2}^2/2$, from which we obtain
	
	$$
	\left(\log f_1(\nu_n)\right)''\lesssim \begin{cases}
	\nu_n^{-2}\|p_0\|_{L_2}^2/f_1(1),\quad 1\leq \nu_n\leq \nu_n^*\\
	\nu_n^{-2},\quad \nu_n^*<\nu_n\leq n^{2/d}
	\end{cases},
	$$
	which suggests that for any $\nu_n,\nu_n'\in[1,\nu_n^*]$
	\begin{align*}
	d_2^2(\nu_n,\nu_n')\lesssim &
	\left(\frac{d}{2}+\frac{\|p_0\|_{L_2}^2}{f_1(1)}\right)\left(-\frac{\log \nu_n+\log \nu_n'}{2}+\log\left(\frac{\nu_n+\nu_n'}{2}\right)\right)\\\lesssim&\left(\frac{d}{2}+\frac{\|p_0\|_{L_2}^2}{f_1(1)}\right)|\log \nu_n-\log \nu_n'|,
	\end{align*}
	and for any $\nu_n,\nu_n'\in [\nu_n^*,n^{2/d}]$
	$$
	d_2^2(\nu_n,\nu_n')\lesssim \left(\frac{d}{2}+1\right)|\log \nu_n-\log \nu_n'|.
	$$
	
	Note that in addition to the bound on $d_2$ obtained above, we also have
	$$
	d_2(\nu_n,\nu_n')\leq \|\bar{K}_{\nu_n}\|_{L_2}+\|\bar{K}_{\nu_n'}\|_{L_2}\leq \|K_{\nu_n}\|_{L_2}+\|K_{\nu_n'}\|_{L_2}\leq \sqrt{2}.
	$$
	Therefore,
	\begin{align*}
	\gamma_1([1,n^{2/d}],d_2)\leq &\sum\limits_{l\geq 0}2^le_l([1,n^{2/d}],d_2)\\
	\lesssim &e_0([1,n^{2/d}],d_2)+\sum\limits_{l\geq 0}2^le_l([1,\nu_n^*],d_2)+\sum\limits_{l\geq 0}2^le_l([\nu_n^*,n^{2/d}],d_2)\\
	\lesssim & 1+\sqrt{\frac{d}{2}+\frac{\|p_0\|_{L_2}^2}{f_1(1)}}\sum\limits_{l\geq 0}2^l\sqrt{\frac{\log \nu_n^*-\log 1}{2^{2^l}}}\\&+\sqrt{\frac{d}{2}+1}\left(\sum\limits_{l\geq 0}2^l\min\left\{1,\sqrt{\frac{\log n^{2/d}-\log \nu_n^*}{2^{2^l}}}\right\}\right)\\
	\lesssim &1+\sqrt{\frac{d}{2}+\frac{\|p_0\|_{L_2}^2}{f_1(1)}}\sqrt{\log \nu_n^*}+\sqrt{\frac{d}{2}+1}\left(\sum\limits_{l\geq 0}2^l\min\left\{1,\sqrt{\frac{\log n^{2/d}}{2^{2^l}}}\right\}\right)\\
	\lesssim & 1+\sqrt{\frac{d}{2}+\frac{\|p_0\|_{L_2}^2}{f_1(1)}}\sqrt{\log \nu_n^*}+\sqrt{\frac{d}{2}+1}\left(\sum\limits_{0\leq l<l^*}2^l+\sum\limits_{l\geq l^*}2^l\sqrt{\frac{\log n^{2/d}}{2^{2^l}}}\right)\\
	\lesssim & 1+\sqrt{\frac{d}{2}+\frac{\|p_0\|_{L_2}^2}{f_1(1)}}\sqrt{\log \nu_n^*}+\sqrt{\frac{d}{2}+1}\cdot2^{l^*}
	\end{align*}
	where $l^*$ is the smallest $l$ such that $$\sqrt{\frac{\log n^{2/d}}{2^{2^l}}}\leq 1.$$
	Hence $2^{l^*}\asymp \log \log n$ and there exists $C=C(d)>0$ such that
	$$
	\gamma_1([1,n^{2/d}],d_2)\leq C(d)\log\log n
	$$
	for sufficiently large $n$.
	
	By the similar approach, we get that
	$$
	D_2\lesssim 1+\sqrt{\frac{d}{2}+\frac{\|p_0\|_{L_2}^2}{f_1(1)}}\sqrt{\log \nu_n^*}+\sqrt{\frac{d}{2}+1}\cdot l^*
	$$
	which is upper-bounded by $C(d)\log \log \log n$ for sufficiently large $n$.
	
	Therefore, we finally obtain that there exists $C(d)>0$ such that for sufficiently large $n$,
	\begin{align}\label{adapt_bound_0}
	P\left(\sup\limits_{1\leq \nu_n\leq n^{2/d}}\left|\frac{1}{n-1}\sum\limits_{i\neq j}\bar{K}_{\nu_n}(X_i,X_j)\right|\geq C(d)(\log\log n+t\log\log\log n )\right)\lesssim \exp(-t^{2/3}).
	\end{align}
	
	\paragraph{Step (\rt).} By slight abuse of notation, there exists $\nu_n^*=\nu_n^*(p_0)>1$ such that
	$$
	\frac{\EE G_{2\nu_n}(X_1,X_2)}{\EE[\bGnn(X_1,X_2)]^2}\leq 2
	$$
	for $\nu_n\geq \nu_n^*$. Therefore,
	\begin{align*}
	\tilde{T}_{n}^{\rm GOF(adapt)}\leq &\sup\limits_{1\leq \nu_n\leq \nu_n^*}\sqrt{\frac{\EE G_{2\nu_n}(X_1,X_2)}{\EE[\bGnn(X_1,X_2)]^2}}\cdot\sup\limits_{1\leq \nu_n\leq \nu_n^*}\left|\frac{1}{n-1}\sum\limits_{i\neq j}\bar{K}_{\nu_n}(X_i,X_j)\right|+\\
	&\sqrt{2}\sup\limits_{\nu_n^*\leq \nu_n\leq n^{2/d}}\left|\frac{1}{n-1}\sum\limits_{i\neq j}\bar{K}_{\nu_n}(X_i,X_j)\right|\\
	\leq &C(p_0)\sup\limits_{1\leq \nu_n\leq \nu_n^*}\left|\frac{1}{n-1}\sum\limits_{i\neq j}\bar{K}_{\nu_n}(X_i,X_j)\right|+\sqrt{2}\sup\limits_{\nu_n^*\leq \nu_n\leq n^{2/d}}\left|\frac{1}{n-1}\sum\limits_{i\neq j}\bar{K}_{\nu_n}(X_i,X_j)\right|
	\end{align*}
	for some $C(p_0)>0$.
	
	Based on arguments similar to those in the first step,
	\begin{align*}
	P\left(\sup\limits_{1\leq \nu_n\leq \nu_n^*}\left|\frac{1}{n-1}\sum\limits_{i\neq j}\bar{K}_{\nu_n}(X_i,X_j)\right|\geq C(d,p_0)t\right)\lesssim \exp(-t^{2/3})
	\end{align*}
	for some $C(d,p_0)>0$ and (\ref{adapt_bound_0}) still holds when $\nu_n$ is restricted to $[\nu_n^*,n^{2/d}]$. They together prove Lemma \ref{at4}.	

\section{Decomposition of dHSIC and Its Variance Estimation}\label{sec:HSIC_decomp}
In this section, we first derive an approximation of $\hat{\gamma^2_{\nu}}(\PP,\PP^{X^1}\otimes\cdots\otimes\PP^{X^k})$ under $H_0$ for general $k$, and then the approximation of $\var\left(\hat{\gamma^2_{\nu}}(\PP,\PP^{X^1}\otimes\cdots\otimes\PP^{X^k})\right)$ can be obtained subsequently. 

Note that
\begin{align*}
G_{\nu}(x,y)=&\int G_{\nu}(u,v)d(\delta_{x}-\PP+\PP)(u)d(\delta_y-\PP+\PP)(v)\\
=&\bar{G}_{\nu}(x,y)+(\EE G_{\nu}(x,X)-\EE G_{\nu}(X,X'))+(\EE G_{\nu}(y,X)-\EE G_{\nu}(X,X'))+\EE G_{\nu}(X,X').
\end{align*}
Similarly write 
$$
G_{\nu}(x,(y^1,\cdots,y^k))=\int G_{\nu}(u,(v^1,\cdots,v^k))d(\delta_x-\PP+\PP)d(\delta_{y^1}-\PP^{X^1}+\PP^{X^1})\cdots d(\delta_{y^k}-\PP^{X^k}+\PP^{X^k})
$$
and expand it as the summation of all $l$-variate centered components where $l\leq k+1$. Do the same expansion to $G_{\nu}((x^1,\cdots,x^k),(y^1,\cdots,y^k))$ and write it as the summation of all $l$-variate centered components where $l\leq 2k$. Plug these expansions in $\hat{\gamma_{\nu}^2}(\PP,\PP^{X^1}\otimes\cdots\otimes \PP^{X^k})$ and denote the summation of all $l$-variate centered components in such expression of $\hat{\gamma_{\nu}^2}(\PP,\PP^{X^1}\otimes\cdots\otimes \PP^{X^k})$ by $D_l(\nu)$ for $l\leq 2k$. Let the remainder $R_n=\sum\limits_{l=3}^{2k}D_l(\nu)$ so that
$$
\hat{\gamma_{\nu}^2}(\PP,\PP^{X^1}\otimes\cdots\otimes \PP^{X^k})=\gamma_{\nu}^2(\PP,\PP^{X^1}\otimes\cdots\otimes \PP^{X^k})+D_1(\nu)+D_2(\nu)+R_n.
$$
Straightforward calculation yields the following facts:
\begin{itemize}
	\item $\EE(R_n)^2\lesssim_k n^{-3}\left(\EE G_{2\nu}(X_1,X_2)+\prod\limits_{l=1}^k\EE G_{2\nu}(X_1^l,X_2^l)\right)$;
	\item under the null hypothesis, $D_1(\nu)=0$ and 
	$$
	D_2(\nu)=\frac{1}{n(n-1)}\sum_{1\leq i\neq j\leq n}G_{\nu}^*(X_i,X_j)
	$$
	where
	\begin{align*}
	G_{\nu}^*(x,y)=\bar{G}_{\nu}(x,y)-\sum\limits_{\substack{1\leq j\leq k}}g_j(x^j,y)-\sum\limits_{\substack{1\leq j\leq k}}g_j(y^j,x)+\sum\limits_{\substack{1\leq j_1,j_2\leq k}}g_{j_1,j_2}(x^{j_1},y^{j_2}).
	\end{align*}
\end{itemize}

\begin{proof}[Proof of Lemma \ref{le:var}]

Observe that under $H_0$, 
$$
\var\left(\hat{\gamma_{\nu}^2}(\PP,\PP^{X^1}\otimes\cdots\otimes \PP^{X^k})\right)=\EE(D_2(\nu))^2+\EE\left(R_n\right)^2=\frac{2}{n(n-1)}\EE[G_{\nu}^*(X_1,X_2)]^2+\EE\left(R_n\right)^2,$$
$$\EE\left(R_n\right)^2\lesssim_k n^{-3}\EE G_{2\nu}(X_1,X_2),$$
and
\begin{align*}
&\EE [G_{\nu}^*(X_1,X_2)]^2\\
=&\EE\left(\bar{G}_{\nu}(X_1,X_2)-\sum\limits_{\substack{1\leq j\leq k}}g_j(X_1^j,X_2)\right)^2\\&-\EE\left(\sum\limits_{\substack{1\leq j\leq k}}g_{j}(X_2^{j},X_1)+\sum\limits_{\substack{1\leq j_1,j_2\leq k}}g_{j_1,j_2}(X_1^{j_1},X_2^{j_2})\right)^2\\
=&\EE \bar{G}_{\nu}^2(X_1,X_2)-2\sum\limits_{1\leq j\leq k}\EE\Big(g_j(X_1^j,X_2)\Big)^2+\sum\limits_{\substack{1\leq j_1,j_2\leq k}}\EE \Big(g_{j_1,j_2}(X_1^{j_1},X_2^{j_2})\Big)^2.
\end{align*}
They together conclude the proof.
\end{proof}

Below we shall further expand $\EE \bar{G}_{\nu}^2(X_1,X_2)$, $\EE\Big(g_j(X_1^j,X_2)\Big)^2$ and $\EE \Big(g_{j_1,j_2}(X_1^{j_1},X_2^{j_2})\Big)^2$ in Lemma \ref{le:var}, based on which consistent estimator of $\var\left(\hat{\gamma^2_{\nu}}(\PP,\PP^{X^1}\otimes\cdots\otimes\PP^{X^k})\right)$ can be derived naturally.

First,
\begin{align*}
&\EE \bar{G}_{\nu}^2(X_1,X_2)\\=&\EE G_{2\nu}(X_1,X_2)-2\EE G_{\nu}(X_1,X_2)G_{\nu}(X_1,X_3)+\left(\EE G_{\nu}(X_1,X_2)\right)^2\\
=&\prod\limits_{1\leq l\leq k}\EE G_{2\nu}(X_1^l,X_2^l)-2\prod\limits_{1\leq l\leq k}\EE G_{\nu}(X_1^l,X_2^l)G_{\nu}(X_1^l,X_3^l)+\prod\limits_{1\leq l\leq k}\left(\EE G_{\nu}(X_1^l,X_2^l)\right)^2.
\end{align*}

Second,
\begin{align*}
&\EE\Big(g_j(X_1^j,X_2)\Big)^2\\=&\EE G_{2\nu}(X_1^{j},X_2^{j})\cdot\prod\limits_{l\neq j}\EE G_{\nu}(X_1^l,X_2^l)G_{\nu}(X_1^l,X_3^l)-\prod\limits_{1\leq l\leq k}\EE G_{\nu}(X_1^l,X_2^l)G_{\nu}(X_1^l,X_3^l)\\
&-\EE G_{\nu}(X_1^{j},X_2^{j})G_{\nu}(X_1^j,X_3^{j})\cdot\prod\limits_{l\neq j}(\EE G_{\nu}(X_1^l,X_2^l))^2+\prod\limits_{1\leq l\leq k}\left(\EE G_{\nu}(X_1^l,X_2^l)\right)^2.
\end{align*}
Hence
\begin{align*}
&\sum\limits_{1\leq j\leq k}\EE\Big(g_j(X_1^j,X_2)\Big)^2\notag\\=&\left(\prod\limits_{1\leq l\leq k}\EE G_{\nu}(X_1^l,X_2^l)G_{\nu}(X_1^l,X_3^l)\right)\left(\sum\limits_{1\leq j\leq k}\frac{\EE G_{2\nu}(X_1^j,X_2^j)}{\EE G_{\nu}(X_1^j,X_2^j)G_{\nu}(X_1^j,X_3^j)}-k\right)\notag\\
&-\left(\prod\limits_{1\leq l\leq k}\left(\EE G_{\nu}(X_1^l,X_2^l)\right)^2\right)\left(\sum\limits_{1\leq j\leq k}\frac{\EE G_{\nu}(X_1^j,X_2^j)G_{\nu}(X_1^j,X_3^j)}{(\EE G_{\nu}(X_1^j,X_2^j))^2}-k\right).
\end{align*}
Finally,
\begin{align*}
&\EE \Big(g_{j_1,j_2}(X_1^{j_1},X_2^{j_2})\Big)^2\\=&\begin{cases}
\EE (\bar{G}_{\nu}(X_1^{j_1},X_2^{j_1}))^2\cdot\prod\limits_{l\neq j_1}\left(\EE G_{\nu}(X_1^{l},X_2^l)\right)^2,\quad j_1=j_2\\\prod\limits_{l\in\{j_1,j_2\}}
\Big(\EE G_{\nu}(X_1^l,X_2^l)G_{\nu}(X_1^{l},X_3^l)-(\EE G_{\nu}(X_1^l,X_2^l))^2\Big)\prod\limits_{l\neq j_1,j_2}\left(\EE G_{\nu}(X_1^{l},X_2^l)\right)^2,\quad j_1\neq j_2.
\end{cases}
\end{align*}
Hence
\begin{align*}\label{HSIC_expand3}
&\sum\limits_{\substack{1\leq j_1,j_2\leq k}}\EE \Big(g_{j_1,j_2}(X_1^{j_1},X_2^{j_2})\Big)^2\\=&\left(\prod\limits_{1\leq l\leq k}\left(\EE G_{\nu}(X_1^l,X_2^l)\right)^2\right)\Bigg(\sum\limits_{1\leq j_1\leq k}\frac{\EE (\bar{G}_{\nu}(X_1^{j_1},X_2^{j_1}))^2}{(\EE G_{\nu}(X_1^{j_1},X_2^{j_1}))^2}\\&+\sum\limits_{1\leq j_1\neq j_2\leq k}\prod\limits_{l\in \{j_1,j_2\}}\left(\frac{\EE G_{\nu}(X_1^{l},X_2^{l})G_{\nu}(X_1^{l},X_2^{l})}{\left(\EE G_{\nu}(X_1^{l},X_2^l)\right)^2}-1\right)\Bigg).
\end{align*}

Then the consistent estimator $\tilde{s}_{n,\nu}^2$ of $\EE\left(G_{\nu}^*(X_1,X_2)\right)^2$ is constructed by replacing
$$
\EE G_{2\nu}(X_1^l,X_2^l),\quad\EE G_{\nu}(X_1^l,X_2^l)G_{\nu}(X_1^l,X_3^l),\quad (\EE G_{\nu}(X_1^l,X_2^l))^2
$$
in the above expansions of 
$$
\EE \bar{G}_{\nu}^2(X_1,X_2),\quad \sum\limits_{1\leq j\leq k}\EE\Big(g_j(X_1^j,X_2)\Big)^2,\quad\sum\limits_{\substack{1\leq j_1,j_2\leq k}}\EE \Big(g_{j_1,j_2}(X_1^{j_1},X_2^{j_2})\Big)^2
$$
with the corresponding unbiased estimators
\begin{gather*}
\frac{1}{n(n-1)}\sum\limits_{1\leq i\neq j\leq n}G_{2\nu_n}(X_i^l,X_j^l),\quad \frac{(n-3)!}{n!}\sum\limits_{\substack{1\le i,j_1,j_2\le n\\ |\{i,j_1,j_2\}|=3}}G_{\nu_n}(X_i^l,X_{j_1}^l)G_{\nu_n}(X_i^l,X_{j_2}^l)
\\\frac{(n-4)!}{n!}\sum\limits_{\substack{1\le i_1,i_2,j_1,j_2\le n\\ |\{i_1,i_2,j_1,j_2\}|=4}}G_{\nu_n}(X_{i_1}^l,X_{j_1}^l)G_{\nu_n}(X_{i_2}^l,X_{j_2}^l)
\end{gather*}
for $1\leq l\leq k$. Again, to avoid a negative estimate of the variance, we can replace $\tilde{s}^2_{n,\nu_n}$ with $1/n^2$ whenever it is negative or too small. Namely, let 
$$
\hat{s}^2_{n,\nu_n}=\max\left\{\tilde{s}^2_{n,\nu_n},1/n^2\right\},$$
and estimate $\var\left(\hat{\gamma^2_{\nu}}(\PP,\PP^{X^1}\otimes\cdots\otimes\PP^{X^k})\right)$ by $2\hat{s}_{n,\nu}^2/(n(n-1))$.

Therefore for general $k$, the single kernel test statistic and the adaptive test statistic are constructed as
$$
T_{n,\nu_n}^{\rm IND}={n\over\sqrt{2}}\hat{s}^{-1}_{n,\nu_n}\hat{\gamma_{\nu_n}^2}(\PP,\PP^{X^1}\otimes\cdots\otimes\PP^{X^k})\quad {\rm and}\quad T_n^{\rm IND(adapt)}=\max_{1\le \nu_n\le n^{2/d}}T_{n,\nu_n}^{\rm IND}
$$
respectively. Accordingly, $\Phi_{n,\nu_n,\alpha}^{\rm IND}$ and $\Phi^{\rm IND(adapt)}$ can be constructed as in the case of $k=2$.

\section{Theoretical Properties of Independence Tests for General $k$}
In this section, with $\Phi_{n,\nu_n,\alpha}^{\rm IND}$ and $\Phi^{\rm IND(adapt)}$ constructed in Appendix \ref{sec:HSIC_decomp} for general $k$, we confirm that Theorem \ref{th:indnull}, Theorem \ref{th:indpower} and Theorem \ref{th:indadapt} still hold. We shall only emphasize the main differences between the new proofs and the original proofs in the case of $k=2$.
\paragraph{Under the null hypothesis:} we only need to re-ensure that $\tilde{s}_{n,\nu_n}^2$ is a consistent estimator of $\EE[G_{\nu_n}^*(X_1,X_2)]^2$. Specifically, we show that
$$
\tilde{s}_{n,\nu_n}^2/\EE[G_{\nu_n}^*(X_1,X_2)]^2\to_p 1
$$
given $1\ll \nu_n\ll n^{4/d}$ for Theorem \ref{th:indnull} and 
$$
\sup\limits_{1\leq \nu_n\leq n^{2/d}}\left|\tilde{s}_{n,\nu_n}^2/\EE[G_{\nu_n}^*(X_1,X_2)]^2-1\right|=o_p(1)
$$
for Theorem \ref{th:indadapt}.

To prove the former one, since
$$
\frac{\EE[G_{\nu_n}^*(X_1,X_2)]^2}{(\pi/(2\nu_n))^{d/2}\|p\|_{L_2}^2}\to 1
$$
as $\nu_n\to \infty$, it suffices to show
$$
\nu_n^{d/2}\left|\tilde{s}_{n,\nu_n}^2-\EE[G_{\nu_n}^*(X_1,X_2)]^2\right|=o_p(1),
$$
which follows considering that
\begin{align}\label{terms}
\nu_n^{d_l/2}\EE G_{2\nu_n}(X_1^l,X_2^l),\quad\nu_n^{d/2}\EE G_{\nu_n}(X_1^l,X_2^l)G_{\nu_n}(X_1^l,X_3^l),\quad \nu_n^{d_l/2}(\EE G_{\nu_n}(X_1^l,X_2^l))^2
\end{align}
are all bounded and they are estimated consistently by their corresponding estimators. For example,
$$
\nu_n^{d_l/2}\EE G_{2\nu_n}(X_1^l,X_2^l)\to \left(\pi/2\right)^{d_l/2}\|p_l\|_{L_2}^2
$$
and
\begin{align*}
&\ \ \nu_n^{d_l}\EE\left(\frac{1}{n(n-1)}\sum\limits_{1\leq i\neq j\leq n}G_{2\nu_n}(X_i^l,X_j^l)-\EE G_{2\nu_n}(X_1^l,X_2^l)\right)^2\\=\ \ &\nu_n^{d_l}\var\left(\frac{1}{n(n-1)}\sum\limits_{1\leq i\neq j\leq n}G_{2\nu_n}(X_i^l,X_j^l)\right)\\
\lesssim\ \ &\nu_n^{d_l}\left(n^{-1}\EE G_{2\nu_n}(X_1^l,X_2^l)G_{2\nu_n}(X_1^l,X_3^l)+n^{-2}\EE G_{4\nu_n}(X_1^l,X_2^l)\right)\\
\lesssim_{d_l}&n^{-1}\nu_n^{d_l/4}\|p_l\|_{L_2}^3+n^{-2}\nu_n^{d_l/2}\|p_l\|_{L_2}^2\to 0.
\end{align*}

The proof of the latter one is similar. It sufficies to have
\begin{itemize}
	\item each term in (\ref{terms}) is bounded for $\nu_n\in[1,\infty)$, which immediately follows since each term is continuous and converges at $\infty$;
	\item the difference between each term in (\ref{terms}) and its corresponding estimator converges to $0$ uniformly over $\nu_n\in[1,n^{2/d}]$, the proof of which is the same with that of Lemma \ref{consistent-est-unif}.
\end{itemize}
\paragraph{Under the alternative hypothesis:} we only need to re-ensure that $\hat{s}_{n,\nu_n}$ is bounded. Specifically, we show
$$
\inf\limits_{p\in H_1^{\rm IND}(\Delta_n,s)}\frac{n\gamma_{\nu_n}^2(\PP,\PP^{X^1}\otimes\cdots\otimes \PP^{X^k})}{\left[\EE\left(\hat{s}_{n,\nu_n}^2\right)^{1/k}\right]^{k/2}}\rightarrow \infty
$$
for Theorem \ref{th:indpower} and 
\begin{align}\label{bdd}
\inf_{s\geq d/4}\inf_{p\in H_1^{\rm IND}(\Delta_{n,s};s)}P\left(
\hat{s}_{n,\nu_n(s)'}^2\leq 2M^2 (2\nu_n(s)'/\pi)^{-d/2}\right)\to 1
\end{align}
for Theorem \ref{th:indadapt}, where $\nu_n(s)'=(\log\log n/n)^{-4/(4s+d)}$.

The former one holds because
\begin{align*}
\EE\left(\hat{s}_{n,\nu_n}^2\right)^{1/k}\leq\ &\EE\left(\max\left\{\left|\tilde{s}_{n,\nu}^2\right|,1/n^2\right\}\right)^{1/k}\\
\leq\ &\EE\left|\tilde{s}_{n,\nu}^2\right|^{1/k}+n^{-2/k}\\\lesssim_k &\left(\prod\limits_{l=1}^k\EE G_{2\nu_n}(X_1^l,X_2^l)\right)^{1/k}+n^{-2/k}\\
\leq\ & \left(M^2(\pi/(2\nu_n))^{d/2}\right)^{1/k}+n^{-2/k}. 
\end{align*}
where the second to last inequality follows from generalized H\"{o}lder's inequality. For example,
$$
\EE\left(\prod\limits_{l=1}^k\frac{1}{n(n-1)}\sum\limits_{1\leq i\neq j\leq n}G_{2\nu_n}(X_i^l,X_j^l)\right)^{1/k}\leq \left(\prod\limits_{l=1}^k\EE G_{2\nu_n}(X_1^l,X_2^l)\right)^{1/k}.
$$

To prove the latter one, note that for $\nu_n=\nu_n(s)'$, all three terms in (\ref{terms}) are bounded by $M_l^2(\pi/2)^{d_l/2}$ and the variances of their corresponding estimators are bounded by
$$
C(d_l)\left(n^{-1}\left(\nu_n(s)'\right)^{d_l/4}M_l^3+n^{-2}\left(\nu_n(s)'\right)^{d_l/2}M_l^2\right)=o(1)
$$
uniformly over all $s$. Therefore,
$$
\inf_{s\geq d/4}\inf_{p\in H_1^{\rm IND}(\Delta_{n,s};s)}P\left(
\left(\nu_n(s)'\right)^{d/2}\left|\tilde{s}_{n,\nu_n(s)'}^2-\EE [G_{\nu_n(s)'}^*(Y_1,Y_2)]^2\right|\leq M^2 (\pi/2)^{d/2}\right)\to 1
$$
where $Y_1,Y_2\sim_{\rm iid} \PP^{X^1}\otimes\cdots\otimes \PP^{X^k}$. Further considering that
$$
\EE [G_{\nu_n(s)'}^*(Y_1,Y_2)]^2\leq \EE[\bar{G}_{\nu_n(s)'}(Y_1,Y_2)]^2\leq M^2(\pi/(2\nu_n(s)'))^{d/2}
$$
and that
$$
1/n^2=o((\nu_n(s)')^{-d/2})
$$
uniformly over all $s$, we prove (\ref{bdd}).
\end{document}